\documentclass[11pt]{amsart}

\usepackage{amssymb,mathrsfs}
\usepackage[toc,page]{appendix}
\usepackage{graphicx,color,tikz,caption,subcaption}

\usepackage[normalem]{ulem}
\RequirePackage{yhmath} 
  \renewcommand\widering[1]{\ring{#1}}
\numberwithin{equation}{section}
\theoremstyle{plain}

\newtheorem{theorem}{Theorem}[section]
\newtheorem{lemma}{Lemma}[section]
\newtheorem{proposition}{Proposition}[section]
\newtheorem{corollary}{Corollary}[section]
\newtheorem{definition}{Definition}[section]

\newtheorem{remark}{Remark}[section]

\newcommand{\N}{{\mathbb N}}     
\newcommand{\R}{{\mathbb R}}
\newcommand{\E}{{\mathbb E}}

\newcommand{\TT}{\mathsf{T}}
\def\wt{\widetilde}

 \numberwithin{equation}{section}

\begin{document}

\title
 {Multifractal analysis and Erd\"os-R\'enyi laws of large numbers for branching random walks in $\R^d$}

\author{Najmeddine Attia}
\address{Facult\'e des Sciences de Monastir, D\'epartement de Math\'ematiques, 5000-Monastir-Tunisia}\email{najmeddine.attia@fsm.rnu.tn}

\author{Julien Barral} 
\address{LAGA, CNRS UMR 7539, Institut Galil\'ee, Universit\'e Sorbonne Paris Nord, 99 avenue Jean-Baptiste Cl\'ement, 93430  Villetaneuse, 
France}\email{barral@math.univ-paris13.fr}
 

\date{}

\begin{abstract}
We revisit the multifractal analysis of $\R^d$-valued branching random walks averages by considering subsets of full Hausdorff dimension of the standard level sets, over each infinite branch of which a quantified version of the Erd\"os-R\'enyi law of large numbers holds. Assuming that the exponential moments of the increments of the walks are finite, we can indeed control simultaneously such sets when the levels belong to the interior of the compact convex domain $I$ of possible levels,  i.e. when they are  associated to so-called Gibbs measures, as well as when they belong to the subset $(\partial{I})_{\mathrm{crit}}$ of $\partial I$ made of levels associated to ``critical'' versions of these Gibbs measures.  It turns out that given such a level of one of these two types, the associated Erd\"os-R\'enyi LLN depends on the metric with which is endowed the boundary of the underlying Galton-Watson tree. To extend our control to all the boundary points in cases where $\partial I\neq (\partial{I})_{\mathrm{crit}}$, we slightly strengthen our assumption on the distribution of the increments to exhibit a natural decomposition of $\partial I\setminus (\partial{I})_{\mathrm{crit}}$ into at most countably many convex sets $J$ of affine dimension $\le d-1$ over each of which we can essentially reduce the study to that of interior and critical points associated to some $\R^{\dim J}$-valued branching random~walk.  

\end{abstract}


\maketitle

\section{Introduction and main results}\label{intro}

Let $\TT$ be a supercritical Galton-Watson tree jointly constructed with a branching random walk taking values in the Euclidean space~$\R^d$,  $d\ge 1$: on a probability space $(\Omega,\mathcal A,\mathbb P)$,  there exists a random vector $\big (N, X=(X_i)_{i\in\mathbb{N}}\big )\in \N\times (\R^d)^{\mathbb N}$, as well as $\{(N_{u},(X_{ui})_{i\in\mathbb N})\}_{u\in \bigcup_{n\ge 0} \mathbb N^n}$, a family of independent copies of $(N, X)$ indexed by the finite words over the alphabet $\mathbb N$ (with the convention that $\N^0$ contains only the empty word denoted $\epsilon$) such that: $(i)$ $\mathbb{E}(N)>1$, $\TT=\bigcup_{n=0}^\infty \TT_n$, where $\TT_0=\{\epsilon\}$, and  $\TT_n=\{ui: \, u\in\TT_{n-1},\, 1\le i\le N_u\}$ for all $n\ge 1$. The boundary of $\TT$ is then the set $\partial\TT$ of infinite words $t_1\cdots t_n\cdots$  over $\N$ such that $t_1\cdots t_k\in\TT_k$ for all $k\ge 1$. $(ii)$ For each $t=t_1t_2\cdots \in \N^\N$ and $n\ge 0$, setting  
$$
S_nX(\omega,t)=\sum_{i=1}^n X_{t_1\cdots t_n}(\omega),
$$
the restriction of $(S_nX)_{n\ge 0}$ to $\partial \TT$ is the branching random walk on $\partial\TT$ to be considered in this paper.

When $N$ is a constant interger $\ge 2$ and the components of $X$ are identically distributed and non constant,  the family $\{(S_nX(t))_{n\in\N}\}_{t\in\partial \TT}$ provides uncountably many random walks with the same law, and it turns out that the large deviations properties shared by these random walks have a counterpart in $\partial \TT$ in the sense that if we consider $E_X=\big \{t\in\partial \TT: A_X(t):=\lim_{n\to\infty} n^{-1}S_nX(t)\text{ exists}\big \}$, the set $A_X(E_X)$ is almost surely equal to a deterministic  non trivial closed convex set. The same property holds if we consider a general branching random walk as defined above. Quantifying geometrically this phenomenon, that is  measuring the sizes of the level sets $A_X^{-1}(\alpha)$, $\alpha\in A_X(E_X)$, as well as that of subsets over which the law of large numbers like property $A_X(t)=\alpha$ is refined by Erd\"os-R\'enyi law of large numbers, is the purpose of the present paper.

Conditionally on non extinction of $\TT$, that is $\partial \TT\neq\emptyset$, if the set $\partial \TT$ is endowed with some metric ${\rm d}$, the multifractal analysis of the averages of $(S_nX)_{n\in\mathbb{N}}$ consists in computing the Hausdorff dimensions of the level sets 
$$
E(X,\alpha)=\Big \{t\in\partial \TT: \lim_{n\to\infty} \frac{S_nX(t)}{n}=\alpha\Big \},\ \alpha\in\R^d
$$
and thus provides a geometric hierarchy between the level sets $E(X,\alpha)$.  A general result (Theorem A below) was obtained in \cite{AB} for these dimensions in the case that ${\rm d}$ is the restriction to $\partial \TT$ of the standard ultrametric distance on $\mathbb N^{\mathbb N}$ defined by 
$$
{\rm d}_1(s,t)=e^{-|s\land t|},
$$ 
where $s\land t$ is the maximal common prefix of $s$ and $t$ and the length of any word $w$ in $(\bigcup_{n\ge 0} \mathbb N^n)\cup\mathbb N^{\mathbb N}$ is denoted by $|w|$.

Define the  Legendre transform of a function $f:\R^d\to \R\cup\{\infty\}$ such that $f\not\equiv \infty$ as 
$$
f^*:\alpha\in\R^d\mapsto \inf\{f(q) - \langle q|\alpha\rangle: q\in\mathbb R^d\}.
$$
Also, define the function
\begin{equation}\label{Pressure function}
\widetilde P_X:q\in\mathbb R^d\mapsto \log \mathbb E\sum_{i=1}^N \exp(\langle q|X_i\rangle)
\end{equation}
as well as
\begin{equation}\label{IX}
I_X=\{\alpha\in\mathbb R^d:\, \widetilde P_X^*(\alpha)\ge 0\}.
\end{equation}
From now on we work conditionally on non extinction of $\TT$, so without loss of generality we assume that $\mathbb{P}(N\ge 1)=1$. The authors proved the following result :

\medskip
\noindent
\textbf{Theorem A (\cite[Theorem 1.1]{AB})}
\textit{With probability 1,
\begin{equation}
\forall\, \alpha\in\mathbb R^d,\ \dim E(X,\alpha)=\begin{cases} \widetilde P_X^*(\alpha)&\text{if } \alpha\in I_X\\
-\infty&\text{otherwise.}
\end{cases}
\end{equation}
}
In this paper, $\dim$ stands for the Hausdorff dimension, and we adopt the convention that for any set $E\subset \N^\N$, $\dim E=-\infty$ if and only if $E=\emptyset$. This result is a geometric counterpart of the following large deviations properties associated with $S_nX$ (\cite[Theorem~1.3]{AB}): For $u\in\TT_n$ define $S_nX(u)$ as the constant value taken by $S_nX$ restricted to the set of elements of $\partial \TT$ with common prefix of generation~$n$ equal to~$u$.  With probability 1, for all $\alpha\in\R^d$,
\begin{align*}
&\lim_{\varepsilon\to0^+} \liminf_{n\to\infty} \frac{\log\#\{u\in \TT_n,\ n^{-1}S_nX(u)\in B(\alpha,\varepsilon)\}}{n}\\
&=\lim_{\varepsilon\to0^+} \limsup_{n\to\infty} \frac{\log\#\{u\in \TT_n,\ n^{-1} S_nX(u)\in B(\alpha,\varepsilon)\}}{n}=\begin{cases} \widetilde P_X^*(\alpha)&\text{if } \alpha\in I_X\\
-\infty&\text{otherwise}
\end{cases},
\end{align*}
where $B(\alpha,\varepsilon)$ stands for the closed Euclidean ball of radius $\varepsilon$ centered at $\alpha$. 

We aim at  strengthening these information in two directions: at first quantify how, for $t\in E(X,\alpha)$, the local averages $n^{-1}(S_{j+n}X(t)-S_{j}X(t))$ ($j\ge 0$) can deviate from~$\alpha$. This will be done by using  {\it quantified} Erd\"os-R\'enyi laws of large numbers (see the next paragraphs for the definition). Specifically, we will seek for subsets $\widetilde E(X,\alpha)$ of $E(X,\alpha)$ of full Hausdorff dimension and for all points of which the local averages of $(S_{n}X(t))_{n\in\mathbb{N}}$ obey the same quantified Erd\"os-R\'enyi law. Also,  we will measure the effect of changing the standard metric to a metric associated to some branching random walk,  both by providing the new values for the Hausdorff dimensions of the sets $E(X,\alpha)$, and observing how the quantified Erd\"os-R\'enyi law invoked in $\widetilde E(X,\alpha)$ may vary with the metric. 

To begin, let us precise what we mean by quantified Erd\"os-R\'enyi law of large numbers. To concretely observe such  laws in our context, we naturally select points in $\partial\TT$ according to some Mandelbrot measures. To define a Mandelbrot measure on $\partial\TT$, jointly with $(\partial \TT,(S_nX)_{n\ge 0})$,  consider a family $\{(N_{u},(X_{ui},\psi_{u_i})_{i\ge 1})\}_{u\in \bigcup_{n\ge 0} \mathbb N^n}$ of independent copies of a random vector $\big (N,(X,\psi)=(X_i,\psi_i)_{i\ge 1}\big )$ taking values in $\mathbb N\times (\mathbb R^d 
\times \R)^{\mathbb N}$, still on $(\Omega,\mathcal A,\mathbb P)$.

Assume that 
$$
\E\sum_{i=1}^N\exp(\psi_i) =1,\  \E\sum_{i=1}^N\psi_i\exp(\psi_i) <1\text{ and } \E\Big (\sum_{i=1}^N\exp(\psi_i) \Big )\log^+\Big (\sum_{i=1}^N\exp(\psi_i) \Big )
<\infty.
$$
Then (see \cite{KP,Biggins1,Lyons}), for each $u\in \bigcup_{n\ge 0} \mathbb N^n$, defining $\TT(u)=\bigcup_{n=0}^\infty \TT_n(u)$, where $\TT_0(u)=\{u\}$, and  $\TT_n(u)=\{vi: \, v\in\TT_{n-1}(u),\, 1\le i\le N_{v}\}$ for all $n\ge 1$, the sequence
\begin{equation}\label{Yn}
Y_n(u)=\sum_{v=v_1\cdots v_n\in \TT(u)}\exp(\psi_{uv_1}+\cdots +\psi_{uv_1\cdots v_n})
\end{equation}
is a positive uniformly integrable martingale of expectation 1 with respect to its natural filtration. We denote by $Y(u)$ its almost sure limit. By construction, the random variables so obtained are identically distributed and almost surely positive. 

Now, for each $u\in  \bigcup_{n\ge 0}\N^n$, let  $[u]$ denote  the cylinder $u\cdot \mathbb N^{\mathbb N}$ and denote by $\mathcal{B}$ the $\sigma$-algebra generated by these cylinders in $\mathbb N^{\mathbb N}$ (which is nothing but the Borel $\sigma$-algebra associated with ${\rm d}_1$).  Then, define
$$
\nu([u])=\mathbf{1}_{\{u\in \TT\}}\exp(\psi_{u_1}+\cdots +\psi_{u_1\cdots u_n})\, Y(u).
$$
Due to the branching property $Y(u)=\sum_{i=1}^{N_u}\exp(\psi_{ui})Y(ui)$, this yields a non-negative additive function of the cylinders, which can be extended into a random measure $\nu_\omega=\nu_{\psi,\omega} $ on $(\mathbb N^{\mathbb N},\mathcal B)$, whose topological support is  $\partial \TT$.  Consider the so-called Peyri\`ere probability measure~$\mathcal Q$  on $(\Omega\times \mathbb N^{\mathbb N},\mathcal A\otimes\mathcal B)$, defined by 
$$
\mathcal Q(C)=\int_\Omega \int_{\mathbb N^{\mathbb N}} \mathbf{1}_C(\omega,t)\, {\rm d}\nu_\omega(t) \, {\rm d}\mathbb P(\omega).
$$ 
The random vectors $\widetilde X_n:(\omega,t)\in \Omega\times \mathbb N^{\mathbb N} \mapsto X_{t_1\cdots t_n}(\omega)$, $n\ge 1$, are independent and identically distributed with respect to $\mathcal Q$, and by definition, given $\omega\in \Omega$, $S_n\widetilde X(\omega,\cdot )=\sum_{k=1}^n\widetilde X_k(\omega,\cdot)$ coincides with $S_nX(\omega,\cdot)$ on $\partial \TT$.

Let 
$$
\Lambda_\psi: \lambda\in\R^d\mapsto \log \mathbb{E}_{\mathcal Q} \big (\exp(\langle \lambda|\widetilde X_1\rangle)\big )=\log \mathbb{E}\Big (\sum_{i=1}^N \exp( \langle \lambda| X_i\rangle+\psi_i)\Big ).
$$

Suppose that $\Lambda_\psi$ is finite on an open convex subset $\mathcal{D}_{\Lambda_\psi}$ of $\R^d$. Suppose also that $\mathcal{D}_{\Lambda_\psi}$ contains~$0$, so that $\nabla \Lambda_\psi (0)$ is well defined, and  one has   
\begin{equation}\label{alphapsi}
\nabla \Lambda_\psi (0)=\mathbb{E}_{\mathcal Q} (\widetilde X_1)=\mathbb E\Big (\sum_{i=1}^N X_i\exp( \psi_i)\Big ); \text{ set }\beta_{\Lambda_\psi}=\nabla \Lambda_\psi (0).
\end{equation}
By the strong law of large numbers $S_n\widetilde X(\omega,t)/n$ tends to $\beta_{\Lambda_\psi}$ as $n\to\infty$, $\mathcal Q$-almost surely. In other words, for $\mathbb{P}$-almost every $\omega$, the measure $\nu_\omega$ is supported on the set $E(X, \beta_{\Lambda_\psi})$. Note that $\max(\Lambda_\psi^*)=\Lambda_\psi^*(\beta_{\Lambda_\psi})=0$ since $\Lambda_\psi(0)=0$.

The classical Erd\"os-R\'enyi law of large numbers~\cite{ER} applied to $(S_n\widetilde X)_{n\in\mathbb{N}}$ claims that if $\widetilde X$ is real valued (i.e. $d=1$) and is not $\mathcal Q$-almost surely equal to the constant $\beta_{\Lambda_\psi}=\Lambda_\psi'(0)$, then for all $\alpha>\beta_{\Lambda_\psi}$ in $\mathcal D_{\Lambda_\psi}$, one has 
$$
\lim_{N\to\infty} \max_{0\le j\le N- \lfloor c_\alpha \log(N)\rfloor}(S_{j+\lfloor c_\alpha \log(N)\rfloor}\widetilde X-S_j\widetilde X)/\lfloor c_\alpha \log(N)\rfloor=\alpha, \text{ where }c_\alpha^{-1}=-\Lambda_\psi^*(\alpha).
$$
This can be reformulated as follows: Let $\widetilde k=(k(n))_{n\in\mathbb{N}}$ be an increasing sequence of integers. If  $\displaystyle\lim_{n\to\infty}n^{-1}\log(k(n))=  -\Lambda_\psi^*(\alpha)$, then one has
$$
\lim_{n\to\infty} \max_{0\le j\le nk(n)-1} (S_{j+n}\widetilde X-S_j\widetilde X)/n=\alpha.
$$ 

The quantified version of the Erd\"os-R\'enyi law of large numbers established in \cite{BL} for $(S_n\widetilde X)_{n\in\mathbb{N}}$ (see specifically \cite[section 3.4]{BL}) as a consequence of a more general statement valid for some class of weakly correlated processes, corresponds to the following large deviations properties : If $\widetilde k=(k(n))_{n\in\mathbb{N}}$ is an increasing sequence of integers, for $t\in\partial\TT$, $B$ a Borel subset of $\R^d$,  $\lambda\in\R^d$, $n\ge 1$ and $0\le j\le nk(n)-1$, consider the following two objects: 
the empirical measure associated  with the $nk(n)$ first normalised increments $n^{-1}(S_{j+n}X(t)-S_{j}(t))$ along the branch~$t$,
\begin{equation}
\label{tildemu}
\mu^t_{\tilde k,n}=\frac{1}{nk(n)}\sum_{j=0}^{nk(n)-1}\delta_{ \frac{S_{j+n}X(t)-S_{j}X(t)}{n}}\ , 
\end{equation}
as well as the logarithmic moment generating function
\begin{align}
\label{tildeLambda}
\Lambda^t_{\tilde k,n}(\lambda)=\log \int_{\mathbb{R}^d}\exp(n\langle\lambda|x\rangle)\, {\rm d}\mu_{\tilde k,n}^t (x).
\end{align}

\noindent{\bf Theorem B (\cite[section 3.4]{BL}).} \textit{Let  $\widetilde k=(k(n))_{n\in\mathbb{N}}\in\N^{\N}$ be  increasing. With probability 1,  for $\nu$-almost every  $t\in\partial\TT$,  the following large deviations properties $\mathrm{LD}({\Lambda_\psi},\widetilde k)$ hold :}
%
%
\medskip

\textit{$\mathrm{LD}({\Lambda_\psi},\widetilde k)$:}

\textit{(1) for all $\lambda \in\mathcal{D}_{\Lambda_\psi}$ such that $\displaystyle\liminf_{n\to\infty}\frac{\log(k(n))}{n}> -\Lambda_\psi^*(\nabla \Lambda_\psi(\lambda))$, one has  
$$
\lim_{n\to\infty}\frac{1}{n}\Lambda^t_{\tilde k,n}(\lambda)=\Lambda_\psi(\lambda);
$$
Hence, for all $\lambda\in\mathcal D_{\Lambda_\psi}$ such that $\displaystyle\liminf_{n\to\infty}\frac{\log(k(n))}{n}> -\Lambda_\psi^*(\nabla \Lambda_\psi(\lambda))$, due to the Gartner-Ellis theorem \cite{Ellis84a,De-Zei}, one has
$$
\lim_{\epsilon\to 0}\lim_{n\to\infty}\frac{1}{n}\log \mu^t_{\tilde k,n}(B(\nabla\Lambda(\lambda),\epsilon))=\Lambda_\psi^*(\nabla \Lambda_\psi(\lambda)).
$$
In other words 
$$
\lim_{\epsilon\to 0}\lim_{n\to\infty}\frac{1}{n}\log \frac{ \#\Big \{0\le j\le nk(n)-1: \frac{S_{j+n}X(t)-S_{j}X(t)}{n}\in B(\nabla\Lambda_\psi(\lambda),\epsilon)\Big\}}{k(n)}=\Lambda_\psi^*(\nabla \Lambda_\psi (\lambda)).
$$
}
\medskip

\textit{(2) For all $\lambda\in \mathcal D_{\Lambda_\psi}$ such that $\limsup_{n\to\infty}\frac{\log(k(n))}{n}<-\Lambda_\psi^*(\nabla \Lambda_\psi(\lambda))$, there exists $\epsilon>0$ such that for $n$ large enough, $\Big \{0\le j\le nk(n)-1: \frac{S_{j+n}X(t)-S_{j}X(t)}{n}\in B(\nabla\Lambda_\psi(\lambda),\epsilon)\Big \} =\emptyset$.}

\medskip

\textit{(3) If $\lambda\in \mathcal D_{\Lambda_\psi}$ and $\displaystyle \lim_{n\to\infty} \frac{\log k(n)}{n}=-\Lambda_\psi^*(\nabla\Lambda_\psi(\lambda))$, and if $\theta\ge 0\mapsto \Lambda_\psi (\theta\lambda)$ is strictly convex at 1,  then for all $\theta \ge 1$ one has 
\begin{equation*}
\lim_{n\to\infty} \frac{1}{n} \Lambda_{\tilde k,n}^t(\theta\lambda)=\Lambda_\psi(\lambda)+(\theta-1)\langle \lambda |\nabla\Lambda_\psi(\lambda)\rangle. 
\end{equation*}
}

\begin{remark}
Notice that by convention  the concave Legendre transform defined in this paper, which is convenient to express Hausdorff dimensions of level sets,  is the opposite of the more standard convex convention used in \cite{BL}).

In fact, in \cite{BL} such a large deviation principle is established for the $k(n)$ ``disjoint'' increments $(S_{(j+1)n}X(t)-S_{jn}X(t))_{0\le j\le k(n)-1}$, but as we will see in this paper, one can extend this result to a large deviation principle valid for $(S_{j+n}X(t)-S_{j}X(t))_{0\le j\le nk(n)-1}$, which is more faithful to the spirit of the original Erd\"os-R\'enyi law of large numbers. Moreover, the validity of $\mathrm{LD}({\Lambda_\psi},\widetilde k)$ for all sequences $\widetilde k$ implies the validity of this law.  
\end{remark}

Thus, with probability 1, $\nu$ is in fact supported on the finer set 
$$
E(X,\beta_{\Lambda_\psi}, {\rm LD}(\Lambda_\psi,\widetilde k))=\left\{t\in\partial\TT: \ \lim_{n\to\infty} \frac{S_nX(t)}{n}=\beta_{\Lambda_\psi} \text{ and } {\rm LD}(\Lambda_\psi,\widetilde k)\text{ holds}\right \}.
$$ 
It turns out that if we define 
$$
\boldsymbol{\widetilde K}=\left\{\widetilde k\in \N^\N:\, \widetilde k\text{ is increasing and }\lim_{n\to\infty}\frac{\log(k(n))}{n}\text{ exists }\right\},
$$
and say that $\mathrm{LD}({\Lambda_\psi})$  holds if  $\mathrm{LD}({\Lambda_\psi},\widetilde k)$ holds for all $k\in  \boldsymbol{\widetilde K}$, the previous theorem has the following rather direct corollary (see Section~\ref{pfcor}):
\begin{corollary}\label{cor-1}With probability 1,  for $\nu$-almost every infinite branch $t\in\partial\TT$, $\mathrm{LD}({\Lambda_\psi})$  holds. Thus,  $\nu$ is supported on 
$$
E(X,\beta_{\Lambda_\psi}, {\rm LD}(\Lambda_\psi))=\Big \{t\in\partial\TT: \ \lim_{n\to\infty} \frac{S_nX(t)}{n}=\beta_{\Lambda_\psi} \text{ and } {\rm LD}(\Lambda_\psi)\text{ holds}\Big  \}.
$$ 
\end{corollary}

Our goal is to refine Theorem A by finding, for a given $\alpha\in I_X$, that is such that $E(X,\alpha)\neq\emptyset$,  a differentiable convex function $\Lambda_\alpha$ finite over an open neighborhood $\mathcal D_{\Lambda_\alpha}$~of~0, such that $\Lambda_\alpha(0)=0$, $\alpha=\nabla\Lambda_\alpha(0)$, and the sets $E(X,\alpha,\mathrm{LD}(\Lambda_\alpha,\widetilde k))$ and  $E(X,\alpha,\mathrm{LD}(\Lambda_\alpha)$  are of maximal Hausdorff dimension in $E(X,\alpha)$. 
This requires  the finiteness of some exponential moments of $\|X\|$, and for the sake of simplicity of the discussion and exposition of our results, we assume  that 
\begin{equation}\label{allmomentsfinite}
\widetilde P_X(q)<\infty, \ \forall\, q\in\mathbb R^d.
\end{equation}
This is equivalent to requiring that $\mathbb E\big (\sum_{i=1}^N \exp (\lambda\|X_i\|)\big )<\infty$ for all $\lambda \ge 0$. 
We will discuss some possible relaxation of this assumption in Section~\ref{Relaxation}.

Without loss of generality, we also assume  the following property about $X$: 
\begin{equation}\label{pP}
 \text{$\not\exists\ (q,c) \in (\R^d\setminus\{0\})\times \R$,\ $\langle q|X_i\rangle =c \quad \forall\ 1\le i\le N$ \ almost surely (a.s.)}.
\end{equation}
If \eqref{pP}  does not hold, either  $d=1$ and the $X_i$, $1\le i\le N$, are equal to the same constant almost surely, which is a trivial case, or $d\ge 2$, and the $X_i$ belong to the same affine hyperplane so that  we can reduce our study to the case of  $\R^{d-1}$  valued random variables.

Define
$$
J_X=\big \{q\in\R^d: \widetilde P_X^*(\nabla \widetilde P_X(q))>0\big \}. 
$$
It turns out that under \eqref{allmomentsfinite} and \eqref{pP}, $I_X$ (recall \eqref{IX}) is a compact set with non-empty interior, such that (see Proposition~\ref{detI})
$$\widering{I}_X=\nabla \widetilde P (J_X).$$
When $\alpha\in \widering I_X$, i.e.  $\alpha=\nabla \widetilde P_X(q)$ for some $q\in J_X$,  setting 
$\psi_\alpha=(\exp({\langle q|X_i\rangle -\widetilde P_X(q)}))_{i\ge 1}$ and assuming that $\E\Big (\sum_{i=1}^N\exp(\psi_{\alpha,i}) \Big )\log^+\Big (\sum_{i=1}^N\exp(\psi_{\alpha,i}) \Big )<\infty$,  one obtains
a non degenerate Mandelbrot measure $\nu_\alpha=\nu_{\psi_\alpha}$ associated with the ``potential'' $\psi_\alpha$, also called Gibbs measure associated with $X$ at $q$, and the previous discussion shows that given an increasing sequence of integers $\widetilde k$, with probability 1, the measure $\nu_\alpha$ is supported on $E(X,\alpha, {\rm LD}(\Lambda_{\psi_\alpha},\widetilde k))$; moreover, due to \eqref{allmomentsfinite}, we can take $\mathcal D_{\Lambda_{\psi_\alpha}}=\R^d$. Moreover, $\Lambda_{\psi_\alpha}$ is strictly convex due to \eqref{pP}. Also, the Hausdorff dimension of $\nu_\alpha$ equals $\dim E(X,\alpha)$, which yields  $\dim E(X,\alpha, {\rm LD}(\Lambda_{\psi_\alpha},\widetilde k))=\dim E(X,\alpha)$ almost surely. 

Then, several questions arise. Let us state and comment them: 

\noindent
{\bf (Q1)}  Is it possible to get the previous property a.s. simultaneously for all $\alpha\in \widering I_X$? 

It is of course closely related to the possibility to estimate almost surely simultaneously the Hausdorff dimensions of the sets $E(X,\alpha)$, $\alpha\in I_X$. For  $\alpha\in \widering I_X$, this can be done under slightly stronger assumptions by constructing simultanously the Gibbs measures $\nu_\alpha$ (thanks to a uniform convergence result due to Biggins \cite{Biggins2}), and simultaneously for all $\alpha\in \widering I_X$ controlling the Haudorff dimension of $\nu_\alpha$ and showing that this measure is carried by $E(X,\alpha)$ (see \cite{B2,Attia}). So one may think that an adaptation of this ``uniform'' approach should give a positive answer to (Q1), since using the Gibbs measure $\nu_\alpha$ for each individual $\alpha\in \widering I_X$ does  give $\dim E(X,\alpha, {\rm LD}(\Lambda_{\psi_\alpha},\widetilde k))=\dim E(X,\alpha)$ almost surely. However,  this strategy meets an essential  difficulty (see Remark~\ref{impossibility}). To overcome it, inspired by techniques used in ergodic theory for the multifractal analysis of Birkhoff averages on hyperbolic attractors \cite{FFW,FLW}, we will use  a concatenation/approximation method to get inohomogeneous Mandelbrot measures adapted to our problem. 

\noindent
{\bf (Q2)}  When $\alpha\in \partial I_X$, is there some $(\Lambda,\mathcal D_\Lambda)$ (depending on $\alpha$) such that the equality $
\dim E(X,\alpha,{\rm LD}(\Lambda,\widetilde k))=  \dim E (X,\alpha)$ holds ?

It will be first answered positively for those $\alpha$ belonging to the subset $(\partial{I}_X)_{\mathrm{crit}}$ of $\partial I_X$ made of levels that can be associated to ``critical'' versions of the Gibbs measures mentioned above; for such a level $\alpha$ there is indeed a natural candidate $\Lambda_{\psi_\alpha}$ as well. Extending our control to all the boundary points of $I_X$ demands to be able to associate to each~$\alpha$ of $\partial I_X\setminus (\partial{I}_X)_{\mathrm{crit}}$ some quantified Erd\"os-R\'enyi LLN, possibly explicit in terms of the parameter $(N,X)$. However,  we even do not have any good description of $\partial I_X\setminus  (\partial{I}_X)_{\mathrm{crit}}$ at our disposal yet. We will strengthen a little \eqref{allmomentsfinite}  and show that there is a natural decomposition of $\partial I\setminus (\partial{I})_{\mathrm{crit}}$ into at most countably many convex sets $J$ of affine dimension $\le d-1$ over each of which we can essentially reduce the study to that of interior and critical points associated to some $\R^{\dim J}$-valued branching random walk. This decomposition yields the desired family of explicit Erd\"os-R\'enyi LLN. 

\noindent
{\bf (Q3)} For $\alpha\in I_X$, would it be possible that there were several  couples $(\Lambda,\mathcal D_\Lambda)$ (with distinct $\Lambda$) such that $
\dim E(X,\alpha,{\rm LD}(\Lambda,\widetilde k))=  \dim E (X,\alpha)$ holds? 

It remains open when $\dim E (X,\alpha)>0$, and  there is no uniqueness in general when $\dim E (X,\alpha)=0$ (see Remark~\ref{uniqueness}(4)).

\noindent
{\bf (Q4)} If $\dim E(X,\alpha, {\rm LD}(\Lambda,\widetilde k))=\dim E(X,\alpha)$ with respect to some metric, how does~$\Lambda$ depend on the choice of the metric? 

As it was said previously, we are going to consider natural metrics obtained from branching random walks. We will compute $\dim E(X,\alpha)$, $\alpha\in I_X$, with respect to such a metric and show that for the same levels $\alpha$ as under ${\mathrm d}_1$, one has $\dim E(X,\alpha,{\rm LD}(\Lambda,\widetilde k))=  \dim E (X,\alpha)$ for some large deviations properties ${\rm LD}(\Lambda,\widetilde k))$ which depends both on $\alpha$ and the metric.  

The more general metrics we consider on $\partial \TT$ are constructed jointly with $(\partial \TT,(S_nX)_{n\ge 0})$  as follows: consider a family $\{(N_{u},(X_{ui},\phi_{ui})_{i\ge 1})\}_{u\in \bigcup_{n\ge 0} \mathbb N^n}$ of independent copies of a random vector $\big (N,(X,\phi)=(X_i,\phi_i)_{i\ge 1}\big )$ taking values in $\mathbb N\times (\mathbb R^d \times \R^*_+)^{\mathbb N}$. Again, $(\Omega,\mathcal A,\mathbb P)$ stands for the probability space over which these random variables are defined. Denote by  $S_n\phi$ the (positive) branching random walk on $\partial \TT$ associated with the family $\{(N_{u},(\phi_{ui})_{i\ge 1})\}_{u\in \bigcup_{n\ge 0} \mathbb N^n}$. Symmetrically to \eqref{allmomentsfinite}, assume that
\begin{equation}\label{allmomentsfinitephi}
P_\phi(t)=\mathbb E\Big (\sum_{i=1}^N\exp (\lambda\phi_i)\Big )<\infty,\ \forall\ \lambda\in\R.
\end{equation}
Then (see Lemma~\ref{controlSnphi}), with probability 1,  $S_n\phi(u)$ tends  uniformly in $u\in\TT_n$ to $\infty$ as $n\to\infty$, so that one gets the random ultrametric distance  
\begin{equation}\label{dphi}
{\mathrm{d}}_\phi:(s,t)\mapsto \exp (-S_{|s\land t|}\phi(t))
\end{equation}
on $\partial \TT$, and $(\partial \TT,\mathrm{d}_\phi)$ is compact. Such metrics are used to study  geometric realization of Mandelbrot measures on random statistically self-similar sets 
(\cite{HW,Falc,Ol1,Mol,B2,Moerters,Biggins3}). 

We now define a family of convex functions which will be essential to describe the Hausdorff dimensions of the sets $E(X,\alpha)$ under $\mathrm{d}_\phi$. For all $(q, \alpha, t)\in\R^d  \times \R^d\times \R$, let 
\begin{equation}\label{sigma}
 \Sigma_\alpha(q,t)=\sum_{i=1}^N \exp ( \langle q| X_i-\alpha\rangle   -t \phi_i). 
\end{equation}
Under \eqref{allmomentsfinite} and \eqref{allmomentsfinitephi}, $\mathbb E(\Sigma_\alpha(q,t))$ is finite for all $(q, \alpha, t)\in\R^d  \times \R^d\times \R$, and 
since the $\phi_i$ are positive, for each $q\in\R^d$ and $\alpha\in \R^d$ there exists a unique $t=\widetilde P_{X,\phi,\alpha}(q) \in\R$ such that 
\begin{equation}\label{Palphaq}
\mathbb{E}\big (\Sigma_\alpha(q,t)\big )=1
\end{equation}
(we indicate the dependence on $(X,\phi)$ in $\wt P_{X,\phi,\alpha}$ in order to avoid confusion with $\widetilde P_X$ or $\widetilde P_\phi$). Moreover, it is direct to see that $(\alpha,q)\mapsto \widetilde P_{X,\phi,\alpha}(q)$ is real analytic by using the implicit function theorem and the real analyticity of $(\alpha,q,t)\mapsto \mathbb{E}\big (\Sigma_\alpha(q,t)\big )$. 

Notice that $\wt P_{X,\phi,\alpha}(0)$ does not depend on $\alpha$; it turns out that it is the Hausdorff dimention of $\partial\TT$ under $\mathrm{d}_\phi$. Notice also that when $\phi_i=1$ for all $i\ge 1$, one has $\mathrm{d}_\phi=\mathrm{d}_1$, and $\widetilde P_{X,\phi,\alpha}(q)= \widetilde P_X(q) - \langle q|\alpha\rangle$, hence $\wt P_{X,\phi,\alpha}^*(0)=\wt P_X^*(\alpha)$. 

Set 
\begin{equation}\label{JXphi}
J_{X,\phi}=\{(q,\alpha)\in\R^d \times I_X: \wt P_{X,\phi,\alpha}^*(\nabla\wt P_{X,\phi,\alpha}(q))>0\}.
\end{equation}
We assume also that 
\begin{equation}\label{finiteness2}
\forall\ (q,\alpha)\in J_{X,\phi},\  \exists \ \gamma>1,\ \mathbb{E}\big((\Sigma_\alpha(q,\wt P_{X,\phi,\alpha}(q))^\gamma\big)<\infty,
\end{equation}
which is automatically satisfied as soon as $\mathbb{E}(N^p)<\infty$ for some $p>1$ and both \eqref{allmomentsfinite} and \eqref{allmomentsfinitephi} hold. This assumption is quite natural in the sense that it is equivalent to requiring that the total mass of the Mandelbrot measure associated with $\psi_{\alpha,q}=\big (\langle q|X_i-\alpha\rangle -\wt P_{X,\phi,\alpha}(q) \phi_i\big )_{i\ge 1}$ does not vanish and belongs to $L^\gamma(\Omega,\mathbb P)$ for some $\gamma>1$; conditions like \eqref{finiteness2} are required in \cite{Biggins2} to construct simultaneously the limits of the martingales \eqref{Yn} when $\psi$ varies in the family $\{\psi_\alpha\}_{\alpha\in \widering I_X}$.  

Under the assumptions adopted in this paper, Theorem A has the following extension.
\begin{theorem} \label{thm-1.1}Assume \eqref{allmomentsfinite}, \eqref{pP}, \eqref{allmomentsfinitephi} and \eqref{finiteness2}, and suppose that $\partial \TT$ is endowed with the distance $\mathrm{d}_\phi$.

With probability 1, for all $\alpha\in I_X$ one has $\dim E(X,\alpha)=\widetilde P_{X,\phi,\alpha}^*(0)$. More generally, for any compact subset $K$ of $\R^d$, let 
$$
E(X,K)=\Big\{t\in\partial\TT: \bigcap_{n\in\mathbb{N}}\overline{ \Big \{\frac{S_nX(t)}{n}:n\ge N\Big\}}=K\Big\},
$$
the set of those $t\in\partial\TT$ such that the set of limit points of $(S_nX(t)/n)_{n\in\mathbb{N}}$ is equal to~$K$. Denote by $\mathscr K$ the set of compact connected subsets of $\R^d$. With probability 1, for all  $K\in\mathscr K$, one has $
\dim E(X,K)=\inf_{\alpha\in K}\widetilde P_{X,\phi,\alpha}^*(0)$.

%
%
\end{theorem}


Notice that contrarily to what happens when $\partial\TT$ is endowed with $\mathrm{d}_1$, in general the mapping $\alpha\in I_X\mapsto \dim E(X,\alpha)=\widetilde P_{X,\phi,\alpha}^*(0)$ is not concave when $\partial\TT$ is endowed with~$\mathrm{d}_\phi$. 
For instance, when  $X_i=\phi_i$ (note that in this case $d=1$), the distortion induced by~$\mathrm{d}_\phi$ can be observed by noting that in this case $\wt P_{X,\phi,\alpha} (q)= q-\wt P_X^{-1}(\alpha q)$, which implies $\wt P_{X,\phi,\alpha}^*(0)=\wt P_X^*(\alpha)/\alpha$ for $\alpha\in I_X$. 
Also, Theorem~\ref{thm-1.1} should be compared to those obtained in  \cite{BSS,FLW,Ol} in the context of Birkhoff averages on conformal repellers. 

Next we state our results on quantified Erd\"os-R\'enyi laws. They require to assume the following property: 
\begin{equation}\label{phibound}
\sup_{q\in\R^d} \mathbb{E}\Big (\sum_{i=1}^N \phi_i \exp (\langle q|X_i\rangle -\wt P_X(q))\Big )<\infty,
\end{equation} 
which holds  as soon as $\big \|\sup_{1\le i\le N}\mathbb{E}\big (\phi_i|\sigma(N,X)\big )\big \|_\infty<\infty$. We also need the following proposition. Let 
\begin{equation}\label{tIX}
\widetilde I_X=\widering I_X\cup(\partial I_X)_{\mathrm{crit}},\text{ where } (\partial I_X)_{\mathrm{crit}}=\big\{\alpha\in\nabla \widetilde P_X(\R^d):\widetilde P_X^*(\alpha)=0\big\}.
\end{equation}

\begin{proposition}\label{qalpha}
Assume \eqref{allmomentsfinite}, \eqref{pP}, \eqref{allmomentsfinitephi}, \eqref{finiteness2} and \eqref{phibound}.  For all $\alpha\in \widetilde I_X$, there exists a unique $q=q_\alpha$ such that $\nabla \wt P_{X,\phi,\alpha}(q)=0$, hence  $\wt P_{X,\phi,\alpha}(q_\alpha)=\wt P^*_{X,\phi,\alpha}(0)$. Moreover, the mapping $\alpha\mapsto q_\alpha$ is real analytic over $\widering I_X$ and continuous over $\widetilde I_X$.
\end{proposition}
Then, observe that if $\alpha\in \widering I_X$,  with probability 1, $E(X,\alpha)$ carries a Mandelbrot measure of maximal Hausdorff dimension $\widetilde P_{X,\phi,\alpha}^*(0)$ with respect to ${\mathrm d}_\phi$, namely the non degenerate Mandelbrot measure~$\nu_\alpha$ associated with $\psi_{\alpha}=\psi_{X,\phi,\alpha}=\big (\langle q_\alpha|X_i-\alpha\rangle -\wt P_{X,\phi,\alpha}(q_\alpha) \phi_i\big )_{i\ge 1}$. The domain of the associated strictly convex function 
$$
\Lambda_{\psi_\alpha}:\lambda\in\R^d\mapsto \log \mathbb{E}\Big (\sum_{i=1}^N \exp( \langle \lambda| X_i\rangle+\psi_{\alpha,i})\Big )
$$ 
is $\R^d$, and the quantified Erd\"os-R\'enyi law of large numbers of Theorem B holds with $\psi=\psi_\alpha$. If $\alpha\in (\partial I_X)_{\mathrm{crit}}$, $\psi_\alpha$ and $\Lambda_{\psi_\alpha}$ can be defined as above as well. There is no associated  Mandelbrot measure, but what is called a critical Mandelbrot measure $\nu_\alpha^c$ associated to $\alpha$ (see \cite{Kyprianou,Liu,B2,BKNSW,BuSaDyKo} for the definition and geometric properties of these objects). However, though $\nu_\alpha^c$ can be used to show that $E(X,\alpha)\neq\emptyset$, there is no associated Peyri\`ere measure, so using $\nu_\alpha^c$ to get $\dim E\big (X,\alpha, \mathrm{LD}( \Lambda_{\psi_\alpha},\widetilde k)\big )=\dim E(X,\alpha)(=0)$ is not possible.

\begin{theorem}\label{UNIFER}Assume \eqref{allmomentsfinite}, \eqref{pP}, \eqref{allmomentsfinitephi}, \eqref{finiteness2} and \eqref{phibound}. Let $\widetilde k$ be an increasing sequence of integers. With probability 1, for all $\alpha\in \widetilde I_X$, $\dim E\big (X,\alpha, \mathrm{LD}( \Lambda_{\psi_\alpha},\widetilde k)\big )=\dim E(X,\alpha)$. 
\end{theorem}

\begin{corollary}\label{cor-2}Assume \eqref{allmomentsfinite}, \eqref{pP}, \eqref{allmomentsfinitephi}, \eqref{finiteness2} and \eqref{phibound}. With probability 1, for all $\alpha\in \widetilde I_X$, $\dim E\big (X,\alpha, \mathrm{LD}( \Lambda_{\psi_\alpha})\big )=\dim E(X,\alpha)$. 
\end{corollary}

As an example, take $N$ deterministic, $\phi=(1)_{i\ge 1}$,  and $X_1,\ldots,X_N$, $N$ Gaussian vectors with at least one of them non degenerate. Then $\widetilde P_X$ is strictly convex and quadratic,  $\widetilde I_X=I_X$, and $\partial I_X= (\partial I_X)_{\mathrm{crit}}$ is an hyperellipsoid. However, in general $I_X\setminus \widetilde I_X=\partial I_X\setminus (\partial I_X)_{\mathrm{crit}}$ may be non empty, and even equal to $\partial I_X$. This is for instance the case when $(N,X_1,\ldots,X_N)$ is deterministic and satisfies \eqref{pP}; indeed, it is easily checked that in this case one has $J_X=\R^d$ so that $(\partial I_X)_{\mathrm{crit}}=\emptyset$. 

To complete Theorem~\ref{UNIFER}, we need to slightly strengthen  the assumptions. First, we replace \eqref{allmomentsfinite} and \eqref{allmomentsfinitephi} by 
\begin{equation}\label{psiX}
\mathbb{E}\Big (\sum_{i=1}^N\exp \big (\psi(\|X_i\|)\big )+\exp \big (\psi(\phi_i)\big )\Big)<\infty
\end{equation} 
for some convex non decreasing function $\psi: \R_+\to\R_+$ such that $\lim_{x\to\infty} \psi(x)/x=\infty$. Also, we assume that $\mathbb{E}(N^p)<\infty$ for some $p>1$. Note that when the $\phi_i$ are constant, this condition is implied by \eqref{finiteness2} (consider $q=0$); also, we already observed that together with \eqref{allmomentsfinite} and \eqref{allmomentsfinitephi} it implies \eqref{finiteness2}).  Moreover, since we will have to guaranty that \eqref{phibound} holds for various branching random walks deduced from $(S_nX,S_n\phi)$ by restriction to some subtrees, we will assume from the outset that
\begin{equation}\label{newphi}
\|\sup_{1\le i\le N}\mathbb{E}(\phi_i|\sigma(N,X))\|_\infty<\infty.
\end{equation}

\noindent
\textbf{Theorem 1.3. (First formulation)} \textit{Assume  that $\mathbb{E}(N^p)<\infty$ for some $p>1$, as well as \eqref{pP}, \eqref{psiX} and \eqref{newphi}. If $I_X\setminus\widetilde I_X\neq\emptyset$, there exists an explicit family of differentiable convex functions $(\Lambda_\alpha)_{\alpha\in I_X\setminus\widetilde I_X}$ with  $\R^d$ as domain, such that: for every increasing sequence of integers $\widetilde k$, with probability 1,  for all $\alpha\in I_X\setminus\widetilde I_X$, one has $\dim E\big (X,\alpha, \mathrm{LD}(\Lambda_\alpha,\widetilde k)\big )=\dim E(X,\alpha)$.}

\begin{corollary}\label{cor-3}Assume that $\mathbb{E}(N^p)<\infty$ for some $p>1$, as well as \eqref{pP}, \eqref{psiX} and \eqref{newphi}. Suppose that $I_X\setminus\widetilde I_X\neq\emptyset$, and  let $(\Lambda_\alpha)_{\alpha\in I_X\setminus\widetilde I_X}$ be as in Theorem~1.3. With probability 1,  for all $\alpha\in I_X\setminus\widetilde I_X$, one has $\dim E\big (X,\alpha, \mathrm{LD}(\Lambda_\alpha)\big )=\dim E(X,\alpha)$.
\end{corollary}
Making explicit the family $(\Lambda_\alpha)_{\alpha\in I_X\setminus\widetilde I_X}$ requires additional definitions. As mentioned above, our approach will exhibit and use a natural decomposition of $I_X\setminus \widetilde I_X$, essentially as a union of at most countably many convex subsets of the form $\widetilde I_Y$, where $Y=(Y_i)_{i\in\mathbb N}$ defines the increments of some branching random walk, and the components of $Y$ take values in some strict affine subspace of $\R^d$. 

\subsection*{Decomposition of $I_X\setminus\widetilde I_X$ and explicitation of  $(\Lambda_\alpha)_{\alpha\in I_X\setminus\widetilde I_X}$ } Denote by $\mathcal C_X$ the closure of the convex subset of $\R^d$ defined as the set of vectors $\alpha$ of the form $\mathbb{E}(\sum_{i=1}^N W_iX_i)$, where $(W_i)_{i\ge 1}$ is a non negative random element of $\mathbb R_+^{\mathbb N}$  jointly defined with $(N,(X_i)_{i\ge 1})$, such that $\mathbb{E}(\sum_{i=1}^N W_i)=1$. It is easily seen that $\mathcal C_X$ is bounded if and only if the $X_i$, $1\le i\le \|N\|_\infty$, are uniformly bounded ($\|N\|_\infty$ may be infinite). 

If $F\subset\mathcal{B}(\mathbb R^d)$, set $N^F= \#\{1\le i\le N:\, X_i\in F\}$, and if $\mathbb{E}(N^F)>0$, set 
$$
\alpha_F= \mathbb{E}\Big (\sum_{i=1}^N\mathbf{1}_F(X_i) X_i\Big)/\mathbb{E}(N^F).
$$
We refer to \cite[Ch. 18]{Roc} for an introduction to the geometric properties of convex sets. Let $\mathcal H_X$ be the set of supporting affine hyperplanes of the close  convex set $\mathcal C_X$, and $\widetilde {\mathcal H}_X$ be the set of those elements $H$ of ${\mathcal H}_X$ such that $\mathbb E(N^H)\ge 1$. Also, let $ {\mathcal F}_X$ be the set of affine subspaces $F$ of $\R^d$ such that $F\subset H$  for some  $H\in {\mathcal H}_X$ and 
\[\widehat {\mathcal F}_X=\{F\in{\mathcal F}_X:\, \mathbb E(N^F)\ge 1\text{ and }\, \forall \, G\in \mathcal F_X,\, G\subsetneq F,\, \mathbb E(N^G)<\mathbb E(N^F)\}.
\]

If $F\in\widehat{\mathcal F}_X$ and $\mathbb{E}(N^F)>1$, to the integers $N^F_u=\#\{1\le i\le N_u:\, X_{ui}\in F\}$, $u\in\bigcup_{n\ge 0}\mathbb N^n$, are naturally associated  two trees :  the supercritical Galton-Watson tree $\widetilde \TT^F$ defined as $\TT$ but with  the branching numbers $N^F_u$ instead of the  $N_u$, and  the subtree $\TT^F$ of $\TT$ defined as $\bigcup_{n\ge 0}\TT^F_n$, where $\TT^F_0=\{\epsilon\}$ and for $n\ge 1$, $\TT^F_n=\{ui:\, u\in\TT^F_{n-1},\, 1\le i\le N_u, \, X_{ui}\in F\}$. Denote by $\vec{F}$ the vector subspace $F-\alpha_F$, and denote by $X_F-\alpha_F$ the random vector $(X_{i_1}-\alpha_F,\ldots,X_{i_{N^F}}-\alpha_F,0,\ldots,0,\ldots)$, where $i_1,\ldots, i_{N_F}$ are the indices $i\in [1,N]$ such that $X_i\in F$, ranked in increasing order; also define $\phi_F=(\phi_{i_1},\ldots,\phi_{i_{N^F}},0,\ldots,0,\ldots)$.  Similarly, define $X_{F,u}-\alpha_F=(X_{ui_1}-\alpha_F,\ldots,X_{ui_{N^F_u}}-\alpha_F,0,\ldots,0,\ldots)$ and $\phi_{F,u}= (\phi_{ui_1},\ldots,\phi_{ui_{N^F_u}},0,\ldots,0,\ldots)$ for all $u\in\bigcup_{n\ge 0}\mathbb N^{\mathbb{N}}$. The trees $\widetilde \TT^F$ and $\TT^F$ are in bijection via the mapping $\boldsymbol{b}_F$ defined by $\boldsymbol{b}_F(\epsilon)=\epsilon$ and once $\boldsymbol{b}_F$ is defined as a bijection between $\widetilde\TT^F_n$ and $\TT^F_n$ for some $n\ge 0$, if $u\in \widetilde\TT^F_n$, one sets $\boldsymbol{b}_F(uj)=\boldsymbol{b}_F(u)i_j$ for $1\le j\le N^F_{\boldsymbol{b}_F(u)}$.  Conditionally on non-extinction of $\widetilde \TT^F$ (and so $\widetilde \TT^F$), $\boldsymbol{b}_F$ extends naturally into a bijection between $\partial \widetilde \TT^F$ and  $\partial \TT^F$.  Also,  the branching random walk $(S_nX-n\alpha_F)_{n\in\mathbb{N}}$ on $\partial T^F$ is related to the $\vec{F}$-valued branching random walk $(S_n(X_F-\alpha_F))_{n\in\mathbb{N}}$ associated with the random vectors $(N^F_u,(X_{F,u}-\alpha_F))$, $u\in\bigcup_{n\ge 0}\mathbb N^{\mathbb{N}}$, on  $\partial \widetilde \TT^F$ via the equality $S_nX(t)-n\alpha_F= S_n(X_F-\alpha_F)(\boldsymbol{b}_F^{-1}(t))$,  and $\boldsymbol{b}_F^{-1}$ is an isometry between $\partial \TT^F$ and  $\partial \widetilde \TT^F$ endowed with  the restriction of  $\mathrm{d}_\phi$ and  the metric $\mathrm{d}_{\phi_F}$ respectively.  

By construction, the function $\widetilde P_{X_F-\alpha_F}$ associated to $(N^F,X_F-\alpha_F)$ as $\widetilde P_X$ is to $(N,X)$ in \eqref{Pressure function} satisfies
$$
\widetilde P_{X_F-\alpha_F}(q)=\log \mathbb{E}\Big (\sum_{i=1}^N\mathbf{1}_F(X_i) \exp(\langle q|X_i-\alpha_F\rangle)\Big) \quad\text{for all }q\in \vec{F}.
$$
If, moreover, $\dim F\ge 1$, it easily seen that  under the assumptions of Theorem~\ref{UNIFER2}, $X_F-\alpha_F$ satisfies the same assumptions as $X$ since $F\in\widehat{\mathcal F}_X$ and  $\mathbb{E}(N^F)>1$.  One associates to  $\widetilde P_{X_F-\alpha_F}$ the sets $I_{X_F-\alpha_F}$ and $\widetilde I_{X_F-\alpha_F}$ in the same way as $I_X$ and $\widetilde I_X$ are associated to $\widetilde P_X$, that is $I_{X_F-\alpha_F}=\big \{\beta\in \vec{F}:\, (\widetilde P^F_{X})^*(\beta)\ge 0\big \}$, and $\widetilde I_{X_F-\alpha_F}=\{\nabla \widetilde P_{X_F-\alpha_F}(q):\, q\in \vec{F},\,  (\widetilde  P_{X_F-\alpha_F})^*(\nabla \widetilde P_{X_F-\alpha_F}(q))\ge 0\}$. Then set $I_{X}^F=\alpha_F+I_{X_F-\alpha_F}$ and $\widetilde I_{X}^F=\alpha_F+\widetilde I_{X_F-\alpha_F}$. The set $I^F_X$ has non-empty relative interior in $F$, that we denote by $\widering I^F_X$, and $\alpha_F\in \widering I^F_X$.  

One also defines  $\mathcal C_{X,F}$,  the closure of the convex subset of $F$ defined as the set of vectors $\alpha$ of the form $\mathbb{E}(\sum_{i=1}^N W_iX_i)$, where $(W_i)_{i\in\mathbb N}$ is a non negative random element of $\mathbb R_+^{\mathbb N}$  jointly defined with $(N,(X_i)_{i\in\mathbb N})$, such that $\mathbb{E}(\sum_{i=1}^N W_i)=1$ and $\mathbb{E}(\sum_{i=1}^N W_i)=\mathbb{E}(\sum_{i=1}^N \mathbf{1}_{F}(X_i)W_i)$. Note that $\mathcal{C}_X=\mathcal{C}_{X, \mathbb{R}^d}$. To  $\mathcal C_{X,F}$ are associated the sets $\mathcal H_{X,F}$,  $\widetilde {\mathcal H}_{X,F}$ and $\widehat {\mathcal F}_{X,F}$ as $\mathcal H_{X}=\mathcal H_{X,\mathbb{R}^d}$,  $\widetilde {\mathcal H}_{X}=\widetilde {\mathcal H}_{X,\mathbb{R}^d}$ and $\widehat {\mathcal F}_{X}=\widehat {\mathcal F}_{X,\mathbb{R}^d}$ are to $\mathcal C_X$. Note that if  $\dim F=0$, then, conditionally on $\partial\TT^F\neq\emptyset$, the branching random walk $(S_nX)_{n\in\mathbb{N}}$ restricted to $\partial\TT^F\neq\emptyset$ equals $(n\alpha_F)_{n\in\mathbb{N}}$.

If $F\in \widehat{\mathcal F}_X$, for $\beta$ and $q$ in  $\vec{F}$, define $\widetilde P_{X_F-\alpha_F,\phi_F,\beta}(q)$ in the same way as $\widetilde P_{X,\phi,\alpha}(q)$, that is as the unique solution of the equation 
$$
\mathbb{E}\Big (\sum_{i=1}^N \mathbf{1}_F(X_i) \exp(\langle q|X_i-\alpha_F-\beta\rangle-t\phi_i)\Big )=1.
$$
If, moreover, $\mathbb{E}(N^F)>1$ and $\dim F\ge 1$, according to Proposition~\ref{qalpha} for all $\beta\in \widetilde I^F_{X-\alpha_F}$ there exists a unique $q_\beta^F\in \vec{F}$ such that $\widetilde P_{X_F-\alpha_F,\phi_F,\beta}(q_\beta^F)=(\widetilde P_{X_F-\alpha_F,\phi_F,\beta})^*(0)$. For all $\alpha\in \widetilde I_X^F$, we then set for all $\lambda\in\R^d$
$$
\Lambda_\alpha^F(\lambda)= \log \mathbb{E}\Big (\sum_{i=1}^N \mathbf{1}_F(X_i)\exp\big  (\langle \lambda |X_i\rangle+\langle q^F_{\alpha-\alpha_F} |X_i-\alpha\rangle-\widetilde P_{X_F-\alpha_F,\phi_F,\alpha-\alpha_F}(q^F_{\alpha-\alpha_F})\phi_i\big )\Big ).
$$
Theorems~\ref{thm-1.1} and \ref{UNIFER} yield that given an increasing sequence of integers $\widetilde k$, with probability~1, conditionally on $\partial\TT_F\neq\emptyset$, for all $\alpha\in I_X^F$, one has:

 $\dim (E(X,\alpha)\cap \partial \TT_F)=(\widetilde P_{X_F-\alpha_F,\phi_F,\alpha-\alpha_F})^*(0)$, and for all $\alpha\in\widetilde I_X^F$, one has $\dim (E(X,\alpha)\cap \partial \TT_F)=\dim (E(X,\alpha,{\rm LD}(\Lambda_\alpha^F,\widetilde k)) \cap \partial \TT_F)$.
 
 \begin{remark}
One checks that choosing another reference point $\widetilde\alpha_F\in  F$ and considering the branching random walk associated with $X_F-\widetilde \alpha_F$ in $\vec{F}$ on $\partial \TT^F$ yields the same functions~$\Lambda_\alpha^F$. This is due to the fact that if $\alpha=\alpha_F+\beta=\widetilde\alpha_F+\widetilde\beta$, then for all $q\in \vec{F}$ one has $\widetilde P_{X_F-\alpha_F,\phi_F,\beta}(q)=\widetilde P_{X_F-\widetilde \alpha_F,\phi_F,\widetilde \beta}(q)$. 
 \end{remark}

If $F\in\widehat {\mathcal F}_X$ is the singleton $\{\alpha_F\}$, set $I_X^F=\widetilde I_X^F=\{\alpha_F\}$ and
$$
\Lambda_{\alpha_F}^{F}:\lambda\in\R^d\mapsto \langle \lambda|\alpha_F\rangle.
$$
Note that in this case $(\Lambda_{\alpha_F}^{F})^*=-\infty\cdot \mathbf{1}_{\R^d\setminus\{\alpha_F\}}$, so that for any increasing sequence of integers $\widetilde k$, if property $\mathrm {LD}(\Lambda^F_{\alpha_F},\widetilde k)$ holds for $(S_nX)_{n\in\mathbb{N}}$ on some infinite branch, it takes a trivial form, and does not depend on $\phi$. 

If  $F\in\widehat {\mathcal F}_X$ and $\mathbb{E}(N^F)=1$, set $I_X^F=\widetilde I_X^F=\{\alpha_F\}$ and
$$
\Lambda_{\alpha_F}^{F}:\lambda\in\R^d\mapsto \log \mathbb{E}\Big (\sum_{i=1}^N \mathbf{1}_{F}(X_i)\exp\big  (\langle \lambda |X_i\rangle\big )\Big ).
$$

Finally, we select a subcollection of $\widehat{\mathcal F}_{X,\mathbb{R^d}}$. First, define for any affine subspace $E$ of $\mathbb{R^d}$
\begin{align*}
\widetilde {\mathcal F}_{X,E}&=\big \{F\in\widehat {\mathcal F}_{X,E}:\, \dim F\ge 1, \text{ and } \mathbb E(N^F)= \mathbb E(N^H)>1\text{ for all  $H\in {\mathcal H}_{X,E}$, $F\subset H$}\big \},\\
\overline {\mathcal F}_{X,E}&=  \big \{F\in\widehat {\mathcal F}_{X,E}:\, \mathbb E(N^F)= \mathbb E(N^H)\text{ for all  $H\in {\mathcal H}_{X,E}$, $F\subset H$}\big \}\setminus \widetilde {\mathcal F}_{X,E}.
\end{align*}
Note that for $\widetilde {\mathcal F}_{X,E}$ to be non empty it is necessary that $\dim E\ge 2$, and if $F\in \overline {\mathcal F}_{X,E}$ then either $\mathbb{E}(N^F)=1$, or $\mathbb{E}(N^F)>1$ and $\dim F=0$. 

Then, define $\widetilde {\mathcal F}^1_{X}=\widetilde {\mathcal F}_{X,\mathbb{R}^d}$, $\overline {\mathcal F}^1_{X}=\overline {\mathcal F}_{X,\mathbb{R}^d}$, and for $2\le i\le d$, $\widetilde {\mathcal F}^i_{X}=\widetilde {\mathcal F}^{i-1}_{X}\bigcup \bigcup_{F\in \widetilde {\mathcal F}^{i-1}_{X}}\widetilde  {\mathcal F}_{X,F}$ and  $\overline {\mathcal F}^i_{X}=\overline {\mathcal F}^{i-1}_{X}\bigcup \bigcup_{F\in \widetilde {\mathcal F}^{i-1}_{X}}\overline {\mathcal F}_{X,F}$. Note that with the convention $\widetilde {\mathcal F}^0_{X}=\emptyset$, the elements of $\widetilde {\mathcal F}^{i}_{X}\setminus\widetilde {\mathcal F}^{i-1}_{X}$ have dimension at most $d-i$, and $\widetilde {\mathcal F}^{d-1}_{X}=\widetilde {\mathcal F}^{d}_{X}$. 

Note that $\widetilde {\mathcal F}_X^{d}\cup\overline {\mathcal F}^{d}_{X}$ is at most countable. Indeed, the mapping $\mu:F\in\mathcal{B}(\R^d)\mapsto \mathbb{E} (N^F)$ is a finite positive  Borel measure, so it cannot assign positive mass to uncountably many affine subspaces $F$ of $\R^d$ such that $\mu(G)<\mu(F)$ for all affine subspaces $G$ of $F$. 

\begin{theorem}\label{UNIFER2} Assume that $\mathbb{E}(N^p)<\infty$ for some $p>1$, as well as \eqref{pP}, \eqref{psiX} and \eqref{newphi}. The following properties hold:
\begin{enumerate}
\item $\widetilde{\mathcal H}_X=\{H\in \mathcal H_X:\, H\cap I_X\neq\emptyset\}$.

\item $\widetilde I_X\subset \widering {\mathcal{C}}_X$ and $I_X\setminus \widetilde I_X= \bigcup_{H\in\widetilde{\mathcal H}_X} H\cap I_X$. In particular, $I_X=\widetilde I_X$ if and only if $\widetilde{\mathcal H}_X=\emptyset$. 
Also, 
$I_X\setminus  \widetilde I_X=\bigsqcup_{F\in {\mathcal F}_X^{d}\cup\overline {\mathcal F}^{d}_{X}}\widetilde I^F_X$. 

\item One has $I_X\setminus \widetilde I_X=\partial I_X$, i.e. $(\partial I_X)_{\mathrm{crit}}=\emptyset$, if and only if $\widetilde{\mathcal H}_X={\mathcal H}_X$. Moreover, in this case $\mathcal C_X$ is a convex polytope and $\mathcal C_X=I_X$, which is equivalent to saying that for any exposed point $P$ of~$\mathcal C_X$ one has $\mathbb{E}(N^{\{P\}})\ge 1$. 

\item With probability 1, for all $F\in\widetilde {\mathcal F}_X^{d}\cup\overline {\mathcal F}^{d}_{X}$: 

\begin{itemize} 
\item If $\mathbb E(N^F)=1$, then $ F\cap I_X=I^F_X=\widetilde I^F_X=\{\alpha_F\}$, and $\dim E(X,\alpha_F)=0$. 

\item If $\mathbb E(N^F)>1$  then  $F\cap I_X=I^F_X$ and for all $\alpha\in F\cap I_X$ one has $\dim E(X,\alpha)=\widetilde P_{X,\phi,\alpha}^*(0)=(\widetilde P_{X_F-\alpha_F,\phi_F,\alpha-\alpha_F})^*(0)$.
\end{itemize}

\item Suppose that $I_X\setminus \widetilde I_X\neq\emptyset$. Let $\widetilde k$ be an increasing sequence of integers. With probability 1, for all $F\in\widetilde {\mathcal F}_X^{d}\cup\overline {\mathcal F}^{d}_{X}$, for all $\alpha\in \widetilde I^F_X$, one has $\dim E\big (X,\alpha, \mathrm{LD}(\Lambda^F_\alpha,\widetilde k)\big )=\dim E(X,\alpha)$.

\end{enumerate}

\end{theorem}

\begin{remark}
Except the claim about the partition of $I_X\setminus \widetilde I_X$, the properties stated in items (1) to (3) of the previous statement hold under \eqref{allmomentsfinite} and \eqref{pP} only. See Proposition~\ref{propkey}.
\end{remark}

\noindent
\textbf{Some examples.} To illustrate the previous result, we focus on examples where $\mathcal C_X$ is compact. Let $K$ be a compact convex subset of $\R^d$, with non-empty interior. Fix $\mu$ a Borel probability measure either fully supported on $\partial K$ or on the set of extremal points of~$K$.  Suppose that $(N,(X_i)_{i\in\mathbb N})$ is chosen so that the $X_i$ are identically distributed with law $\mu$, and independent of $N$.  It is clear that $\mathcal C_X\subset K$, since for any half-space $V=\{\beta\in\R^d: \langle q|\beta\rangle\le c\}$ ($(q,c)\in\mathbb{S}^{d-1}\times\R$)  which contains $K$ and any $\alpha=\mathbb{E}\big (\sum_{i=1}^N W_iX_i\big)\in\mathcal C_X$, the fact that almost surely for all $1\le i\le N$ one has $\langle q|X_i\rangle\le c$ hence $\langle q|W_iX_i\rangle\le W_ic$, implies $\langle q|\alpha\rangle\le \mathbb{E}\big (\sum_{i=1}^N W_i\big)c=c$, that is  $\alpha\in V$. To see the reverse inclusion, fix $\alpha\in K$ and a Borel probability measure on $\nu_\alpha$ supported on $\partial K$ such that $\alpha=\int_{\partial K} \beta\,\mathrm{d}\nu_\alpha(\beta)$.  Fix $\varepsilon>0$ and a finite partition $\mathcal A_\varepsilon=\{A_j\}_{j=1}^{p_\varepsilon}$ of $\partial K$ into Borel subsets of positive $\mu$-measure and diameter less than $\varepsilon$. Define the  sequence of identically distributed nonnegative random variables $W_{\mathcal A_\varepsilon,i}=(\mathbb{E}(N))^{-1} \sum_{j}\mathbf{1}_{A_j}(X_i)\frac{\nu_\alpha (A_i)}{\mu(A_i)}$, $i\ge 1$.  Note that $\mathbb{E}(\sum_{i=1}^N W_{\mathcal A_\varepsilon,i})=1$. As $\varepsilon$ tends to 0, by construction $\mathbb{E}(\sum_{i=1}^N W_{\mathcal A_\varepsilon,i}X_i)$ tends to $\alpha$. Hence $\alpha\in \mathcal C_X$. 

In the following examples we pick special examples corresponding to the situation just described.

\noindent
(1) Supppose that $K$ is a convex polytope with $n$ vertices $P_1,\ldots, P_n$.  Take $\mu=\sum_{j=1}^np_j\delta_{P_j}$, where $(p_1,\ldots,p_n)$ is a positive probability vector. Point (3) of Theorem~\ref{UNIFER2} shows that it is necessary and sufficient that  $\mathbb{E}(N)\ge \max_j p_j^{-1}$  so that $I_X=K$. 

\noindent
(2) We now give an example to illustrate the fact that the elements of $\widetilde {\mathcal F}^1_X\cup \overline{ \mathcal F}^1_X$ can have any integer dimension between $0$ and $d-1$ when $d\ge 2$ (the case $d=1$ is trivial). Suppose that  $K$ as the following properties: $\partial K$ is $C^\infty$ smooth; also, there exist a convex polytope $\widetilde K$ such that $K\subset \widetilde K$, for each face $Q_j$ of $\widetilde K$ of dimension $\ge 1$, $K_j=K\cap Q_j=\partial K\cap Q_j$ is included in the relative interior of $Q_j$ and is itself open relative to $Q_j$, and  $\partial K\setminus\bigcup_j K_j$ has positive curvature (recall that a convex subset $Q$ of dimension $\ge 1$ is a face of $P$ if  $Q=P\cap H$ where $H$ is a supporting  hyperplane of $P$).  The previous properties imply that the sets $K_j$ are pairwise disjoint and each $K_j$ is contained in a unique supporting  hyperplane of $K$ (the existence of such configurations is quite intuitive and it turns out that such a $K$ can be associated to any convex polytope $\widetilde K$ with non empty interior \cite{Ghomi}. Note also that any  convex polytope does possess  faces of dimension $k$ for all  integers $k$ between 0 and $d-1$). 

Suppose, moreover, that  $\mu$ has a topological support equal to $\partial K$, that its restriction  to each $K_j$ does not vanish and has topological support equal to $K_j$, and  that $\mu$ has a unique atom, at a point $\alpha_0\in\partial K\setminus\bigcup_j K_j$. Denote by $F_j$ the smallest affine subset containing $K_j$. Taking $\mathbb{E}(N)\ge \max(\mu(\{\alpha_0\})^{-1},\max_j\mu(K_j)^{-1})$ implies that $F_j\in \widetilde {\mathcal F}^1_X$ or $F_j\in \overline{ \mathcal F}^1_X$ according to whether $\mathbb{E}(N^{F_j})=\mathbb{E}(N)\mu(K_j)>1$ or $\mathbb{E}(N^{F_j})=\mathbb{E}(N)\mu(K_j)=1$. Similarly, since $K$ possesses a tangent space at $\{\alpha_0\}$, one has $\{\alpha_0\}\in \widetilde {\mathcal F}^1_X$ or $\{\alpha_0\}\in \overline{ \mathcal F}^1_X$ according to whether $\mathbb{E}(N^{\{\alpha_0\}})=\mathbb{E}(N)\mu(\{\alpha_0\})>1$ or $\mathbb{E}(N^{\{\alpha_0\}})=\mathbb{E}(N)\mu(\{\alpha_0\})=1$. Also, by construction and due to Theorem~\ref{UNIFER2}(1)(2), $(\partial K\setminus (\{\alpha_0\}\cup \bigcup_j K_j))\cap I_X=\emptyset$.

\noindent
(3) The last example provides a situation where $\widetilde {\mathcal F}^1_X$ is infinite (note that $d$ has to be $\ge 3$), as well as $\widetilde {\mathcal F}^i_X\setminus \widetilde {\mathcal F}^{i-1}_X$ for all $2\le i\le d-1$. 

Let $(\theta_n)_{n\in\mathbb{N}}$ be an increasing sequence in $[0,2\pi)$ converging to $2\pi$ and such that $\theta_1=0$. Define $K_2$ as the convex hull of the set $\{\alpha_n=(\cos(\theta_n),\sin(\theta_n),0,\ldots,0):\, n\ge 1\}$. Then,  pick for $3\le i\le d$, $\beta_i\in \R^i\times\{0\}^{d-i}\setminus (\R^{i-1}\times\{0\}^{d-i+1})$ and define recursively $K_i$ as the convex hull of $K_{i-1}\cup \{\alpha_i\}$. Let $(p_n)_{n\in\mathbb{N}}$ and $(q_i)_{i=3}^d$ be  positive sequences, such that $\sum_{n\in\mathbb{N}} p_n+\sum_{i=3}^dq_i=1$ and set $\mu= \sum_{n\in\mathbb{N}} p_n\delta_{\alpha_n}+\sum_{i=3}^dq_i\delta_{\beta_i}$. 

For $n\ge 1$, denote $\{\alpha_n\}$ by $F^n_0$, the line containing $[\alpha_n,\alpha_{n+1}]$ by $F^n_1$, and for $2\le i\le d-1$, the $i$-dimensional affine subspace generated by $\{\alpha_n,\alpha_{n+1},\beta_3,\ldots,\beta_{i+1}\}$ by $F^n_i$. We get an increasing sequence of convex polytopes, which are faces of $K_d$.  If $q_3\mathbb{E}(N)>1$, then for every $2\le i\le d-1$,  the affine subspaces $F^n_{i}$, $n\ge 1$, belong to $\widetilde{\mathcal F}^{1}_X$ if $i=d-1$, and to $\widetilde{\mathcal F}^{d-i}_X\setminus \widetilde{\mathcal F}^{d-i-1}$ otherwise, with the convention $\widetilde{\mathcal F}^{0}_X=\emptyset$.   Moreover, for all $n\ge 1$, denoting by $G^n_1$ the line supporting the segments $[\alpha_n,\beta_3]$, one has $G^n_1\in \widetilde{\mathcal F}^{d-1}_X\setminus \widetilde{\mathcal F}^{d-2}_X$. Also, the line supporting $[\alpha_n,\alpha_{n+1}]$ belongs to $\widetilde{\mathcal F}^{d-1}_X\setminus \widetilde{\mathcal F}^{d-2}_X$ or  $\overline {\mathcal F}^{d-1}_X\setminus \overline {\mathcal F}^{d-2}_X$according to whether $p_n+p_{n+1}>1$ of $p_n+p_{n+1}=1$, which  can happen for finitely many $n$ only, so that $(\partial I_X)_{\mathrm{crit}}\neq\emptyset$. The point $\beta_3$ belongs to $\overline {\mathcal F}^{d}_X\setminus \overline {\mathcal F}^{d-1}_X$. Note that the previous observations  do not exhaust the description of $\widetilde {\mathcal F}^d_X\cup \overline{ \mathcal F}^d_X$.

Let us finish this section with some comments and remarks.
\begin{remark}\label{uniqueness}
(1) In the case $d=1$, a partial result regarding the Hausdorff dimensions of the sets $E(X,\alpha)$ is presented   in \cite{A2}, where $\dim E(X,\alpha)$ is computed under ${\mathrm d}_\phi$ for each individual $\alpha$ of the interval $\widering I_X$, almost surely, by using the Gibbs measure  $\nu_\alpha$.

\noindent
(2) A concatenation/approximation method, somehow more elaborate than that used in the present paper, was used in \cite{AB} to construct some family of inhomogeneous Mandelbrot measures particularly adapted to get the general Theorem A, and more generally to compute $\dim E(X,K)$ under the metric $\mathrm{d}_1$ (where $K$ may belong to a larger classe of closed connected sets when $I_X$ is not bounded); $0$-$\infty$ laws are also established for the Hausdorff and packing measures of the sets $E(X,\alpha)$. This method was considered in particular to deal with those $\alpha$ belonging to $  I_X\setminus \overline {\nabla  \widetilde P_X(J_X)}$ whenever this set is non empty, a situation which occurs in presence of so-called first order phase transitions but that we will not meet here due  to \eqref{allmomentsfinite}. Adapting this method when working with a metric $\mathrm{d}_\phi$ more general than $\mathrm{d}_1$ seems quite delicate; this is related to the loss of concavity of the mapping $\alpha\mapsto \dim E(X,\alpha)$. Moreover, even when working with $\mathrm{d}_1$, quantified Erd\"os-R\'enyi laws of large numbers associated with $\alpha\in I_X\setminus \overline {\nabla \widetilde P_X(J_X)}$ seem out of reach by using the techniques currently at our disposal. 

\noindent
(3) With respect to \cite{AB}, appart from the fact that we modify the metric, in our study of the Hausdorff dimensions of the sets $E(X,\alpha)$ or $E(X,K)$, a difference will come from the way  we simultaneously estimate from below  the Hausdorff dimensions of the  inhomogeneous Mandelbrot measures coming into play. We use the fact that the elements of this uncountable family are simultaneously not killed by the actions of some percolations processes, adapting an original idea of Kahane for the action of a given multiplicative chaos on a fixed Radon measure \cite{K1,K2}. In \cite{AB}, we directly estimated the dimensions of such measures by using large deviations estimates, with different technicalities as a counterpart.

\noindent
(4) Let us come back to (Q3). When $\alpha\in(\partial I_X)_{\rm{crit}}$, for each choice of $\phi$ such that the assumptions of Theorem~\ref{UNIFER} hold, there is a function $\Lambda_{X,\phi,\alpha}$, depending on $\phi$, such that the conclusions of Corollary~\ref{cor-2} hold for $\alpha$. Moreover, for each such $\phi$, one has $\dim E(X,\alpha,\mathrm{LD}( \Lambda_{\psi_{X,\phi,\alpha}}))=0=\dim E(X,\alpha)$ both under $\mathrm{d}_1$ and $\mathrm{d}_\phi$ due to Lemma~\ref{controlSnphi}. If $\phi$ is not a multiple of $(1)_{i\in\mathbb N}$, then  $\Lambda_{\psi_{X,\phi,\alpha}}$ differs in general from $\Lambda_{\psi_{X,1,\alpha}}$, hence the choice of $(\Lambda,\mathbb R)$ is non unique. Also,  when $(X,\phi)$ is bounded, it is not hard to see that when $\alpha \in\widering I_X$ there is a unique Mandelbrot measure whose dimension equals $\dim E(X,\alpha)$, but this observation is far from being sufficient to answer the uniqueness question we raise.  

%
\end{remark}

The paper is organised as follows. In Section~\ref{Justification} we justify some properties used in the introduction. In Section~\ref{ihM} we construct inhomogeneous Mandelbrot measures and compute their Hausdorff dimensions; these measures will be used to get the sharp lower bound for $\dim E(X,K)$ and $\dim E\big (X,\alpha, \mathrm{LD}( \Lambda_{\psi_\alpha},\widetilde k)\big )$ in the proofs of our three main results. Section~\ref{proofthm1.1} establishes Theorem~\ref{thm-1.1}, while Theorems~\ref{UNIFER} and~\ref{UNIFER2} are proved in Sections~\ref{proofLD} and~\ref{proofLD2} respectively, and Corollaries~\ref{cor-1},~\ref{cor-2}  and~\ref{cor-3} are proved in Section~\ref{pfcor}. 

\section{Justification of some properties claimed in the introduction}\label{Justification} 

The following proposition was invoked in the previous section.
\begin{proposition}\label{detI}{\cite[Proposition 2.2 and Remark 2.1]{AB}}\label{jusfitication} Assume \eqref{allmomentsfinite}.
\begin{enumerate}
\item  $\widetilde P_X$ is strictly convex and $I_X$ is  convex, compact with non-empty interior.

\item $I_X=\overline{ \{\nabla \widetilde P_X(q): q\in J_X\}}$, and $\widering I_X= \nabla \widetilde P_X (J_X)$. 
\end{enumerate}
\end{proposition}

The fact that $\mathrm{d}_\phi$ is a metric is a direct consequence of the third inequality of the following elementary lemma, whose proof is left to the reader. 

\begin{lemma}\label{controlSnphi}
Assume \eqref{allmomentsfinitephi}. There exist $0<\widetilde\beta\le \beta<1$ such that, with probability 1, for $n$ large enough, 
$$
\widetilde\beta^n\le \min\{\exp (-S_n\phi(u)) :u\in \TT_n\}\le \max\{\exp (-S_n\phi(u)) :u\in \TT_n\}\le \beta^n.
$$
\end{lemma}
%
%

\begin{proof}[Proof of Proposition~\ref{qalpha}] 

Fix $q\in \R^d$ and $\alpha\in \widetilde{I}_X$. Then for $t\in \R$ define 
\begin{equation*}
\begin{split}\ell(t)&=\log \E\Big (\sum_{i=1}^N \exp \big (\langle q|X_i-\alpha\rangle -t\phi_i\big )\Big ).
\end{split}
\end{equation*}
Note that  $\ell(0)=\wt P_X(q)-\langle q|\alpha\rangle$ and $\ell'(0)=-\E\Big (\sum_{i=1}^N \phi_i\exp \big (\langle q|X_i\rangle -\wt P_X(q) \big )\Big )\ge -\lambda$, where $
\lambda=\sup_{q\in\R^d} \mathbb{E}\Big (\sum_{i=1}^N \phi_i \exp (\langle q|X_i\rangle -\wt P_X(q))\Big )
\in (0,\infty)$ due to~\eqref{phibound}. 
Moreover $\ell$ is convex, so for all $t\ge 0$ one has $\ell(t)\ge \ell (0)-\lambda t$. Since, by definition of  $\wt P_{X,\phi,\alpha}(q)$, $\ell(\wt P_{X,\phi,\alpha}(q))=0$, it follows that $\wt P_{X,\phi,\alpha} (q)\ge \lambda^{-1}(\wt P_X(q)-\langle q|\alpha\rangle)$. 

Now, observe that due to the convexity of $\wt P_{X,\phi,\alpha} (\cdot)$, as well as the strict convexity of $\wt P_X(q)-\langle q|\alpha\rangle$ and the fact that $\wt P_X(q)-\langle q|\alpha\rangle$ reaches its infimum (at $q'$ such that $\nabla \widetilde P_X(q')=\alpha$), the mapping $\wt P_{X,\phi,\alpha} (\cdot)$ reaches its infimum  at some $q_\alpha\in\R^d$. If there are two distinct such $q_\alpha$ and $q'_\alpha$, then, setting $v=q_\alpha-q'_\alpha$, one has  $\langle v|\nabla \wt P_{X,\phi,\alpha} (q)\rangle=0$ over $[q_\alpha,q_\alpha']$, that is $\E\Big (\sum_{i=1}^N \langle v|X_i-\alpha\rangle \exp(\langle q|X_i-\alpha\rangle -\wt P_{X,\phi,\alpha} (q)\phi_i)\Big ) =0$, in view of  \eqref{nabla}. Differentiating again over $[q_\alpha,q_\alpha']$ yields $\E\Big (\sum_{i=1}^N \langle v|X_i-\alpha\rangle ^2\exp(\langle q|X_i-\alpha\rangle -\wt P_{X,\phi,\alpha} (q)\phi_i)\Big ) =0$ over $[q_\alpha,q_\alpha']$, hence $ \langle v|X_i\rangle =\langle v|\alpha\rangle$ almost surely for all $1\le i\le N$. But this contradicts \eqref{pP}. By construction, $\nabla \wt P_{X,\phi,\alpha} (q_\alpha)=0$ and $q_\alpha$ is the unique $q$ at which $\nabla \wt P_{X,\phi,\alpha}$ vanishes.

To see that $\alpha\in \widering I_X\mapsto q_\alpha$ is real analytic, one first observes that the differential of $f:q\mapsto \nabla \wt P_{X,\phi,\alpha}(q)$ at $q_\alpha$ is invertible. Indeed, a calculation shows that there exists $c>0$ such that $\frac{\partial f}{\partial q_k} (q_\alpha)= c \mathbb{E}\Big (\sum_{i=1}^N (X_i-\alpha)_k (X_i-\alpha) \exp(\langle q_\alpha|X_i-\alpha\rangle -\wt P_{X,\phi,\alpha} (q_\alpha)\phi_i)\Big )$.    This implies that if the differential of $f$ at $q_\alpha$ vanishes at some $v\in\R^d\setminus\{0\}$, then once again 
$\mathbb{E}\Big (\sum_{i=1}^N \langle v|(X_i-\alpha)\rangle^2 \exp(\langle q_\alpha|X_i-\alpha\rangle -\wt P_{X,\phi,\alpha} (q_\alpha)\phi_i)\Big )=0$, hence the same contradiction as above. It is now possible to apply the implicit function theorem  to $(\alpha,q)\mapsto (\alpha,\nabla \wt P_{X,\phi,\alpha} (q))$ at~$(\alpha,q_\alpha)$. To see that $\alpha\in\widetilde I_X\mapsto q_\alpha$ is continuous,  note that if $\alpha\in \widetilde I_X\setminus \widering I_X$, then $\alpha=\nabla\widetilde P_X(q_0)$ for some $q_0$ such that $\widetilde P_X^*(\nabla\widetilde P_X(q_0))=0$. Also, for such an $\alpha$, there exists a neighborhood~ $U$ of $q_0$ and a neighborhood $V$  of $\alpha$ such that $V=\nabla\widetilde P_X(U)$ and for all $\beta\in V$ one has $\widetilde P_X^*(\beta)=\wt P_X(q')-\langle q'|\beta\rangle>-\infty$ as well ($q\mapsto \nabla \widetilde P_X(q)$ is a local diffeomorphism).  By the same argument as above, one has $\wt P_{X,\phi,\beta}(q)\ge \lambda^{-1}(\wt P_X(q)-\langle q|\beta\rangle)$, so $\inf_{q\in\R^d} \wt P_{X,\phi,\beta}(q)$ is attained at a unique $q_\beta$, and $q_\beta$ depends analytically on $\beta$ over~$V$.  
\end{proof}


\section{Some inhomogeneous Mandelbrot measures on $\partial\TT$ and simultaneous calculation of their   Hausdorff dimensions using percolation} \label{ihM}
We construct our main tool to get the simultaneous calculation of the Hausdorff dimensions of the sets we are interested in. It consists of a family of non degenerate inhomogeneous Mandelbrot measures on $\partial\TT$, of which we provide a lower bound of the lower Hausdorff dimension. This requires several steps. In Section~\ref{prelim}, we gather useful preliminary observations to determine a good set of parameters $\mathcal{R}$ to be used to define the measures. In Section~\ref{MMP} we introduce families of inhomogeneous Mandelbrot real valued martingales indexed both by $\mathcal{R}$ and the set $(0,1]$ of parameters involved in fractal percolation processes on $\partial \TT$. The non degenerate character of some of these martingales is established and used in Sections~\ref{sec3.3} and~\ref{LBM}  to define the inhomogeneous Mandelbrot measures and estimate their Hausdorff dimension from below thanks to a ``uniform'' version of the percolation argument developed by Kahane in~\cite{K1,K2}.

\subsection{ Preliminary observations, and definition of a set of parameters}\label{prelim}

A calculation shows that for $(q,\alpha)\in \R^d\times\R^d$ one has
\begin{equation}\label{nabla}
\nabla \wt P_{X,\phi,\alpha}(q)=\frac{\E\Big (\sum_{i=1}^N X_i\exp(\langle q|X_i-\alpha\rangle -\wt P_{X,\phi,\alpha} (q)\phi_i)\Big ) -\alpha}{\E\Big (\sum_{i=1}^N \phi_i\exp(\langle q|X_i-\alpha\rangle -\wt P_{X,\phi,\alpha} (q)\phi_i)\Big )}.
\end{equation}
Also, for each $(q,\alpha)\in \R^d\times\R^d$ a Mandelbrot measure $\mu_{q,\alpha}$ on $\partial \TT$ is associated with the vectors $(N_u, \exp(\langle q|X_{u1}-\alpha\rangle -\wt P_{X,\phi,\alpha} (q)\phi_{u1}),\exp(\langle q|X_{u2}-\alpha\rangle -\wt P_{X,\phi,\alpha} (q)\phi_{u2}),\ldots  )$, $u\in \bigcup_{n\ge 0}\mathbb N^n$, and this measure is non degenerate if and only if, after setting 
$$
\psi(q,\alpha)=\big (\psi_{i}(q,\alpha)=\langle q|X_{i}-\alpha\rangle -\wt P_{X,\phi,\alpha} (q)\phi_{i}\big )_{i\ge 1},
$$
the ``entropy''
\begin{equation}\label{entropie}
h(q,\alpha)=-\mathbb E \Big (\sum_{i=1}^N\psi_{i}(q,\alpha)\exp(\psi_{i}(q,\alpha))\Big )
\end{equation}
is positive, and $\mathbb E \Big (\big (\sum_{i=1}^N\exp(\psi_{i}(q,\alpha)\big )\log_+ \sum_{i=1}^N\exp(\psi_{i}(q,\alpha))\Big )<\infty$. 

Define the ``Lyapounov exponent''
\begin{equation}\label{lyap}
\lambda(q,\alpha):=\E\Big (\sum_{i=1}^N \phi_i\exp(\psi_{i}(q,\alpha))\Big )\in (0,\infty).
\end{equation} 
An identification shows that
\begin{equation}\label{entropiesurLyapounov}
\widetilde P_{X,\phi,\alpha}^*(\nabla\widetilde P_{X,\phi,\alpha}(q))=\widetilde P_{X,\phi,\alpha}(q)-\langle q|\nabla \widetilde P_{X,\phi,\alpha}(q)\rangle = \frac{h(q,\alpha)}{\lambda(q,\alpha)}.
\end{equation}
Consequently, since we assumed \eqref{finiteness2}, the measure $\mu_{q,\alpha}$ is non degenerate if and only if $\widetilde P_{X,\phi,\alpha}^*(\nabla\widetilde P_{X,\phi,\alpha}(q))>0$, that is $(q,\alpha)\in J_{X,\phi}$. Also, using the definition of $\beta_{\Lambda_\psi}$ introduced in  \eqref{alphapsi} and setting $\beta(q,\alpha)= \beta_{\Lambda_{\psi(q,\alpha)}}$, one gets
\begin{equation}\label{alphaXphi}
\beta(q,\alpha)=\E\Big (\sum_{i=1}^N X_i\exp(\langle q|X_i-\alpha\rangle -\wt P_{X,\phi,\alpha} (q)\phi_i)\Big ),
\end{equation}
and with probability 1, 
\begin{equation}\label{LLNq}
\lim_{n\to\infty} \frac{S_nX(t)}{n}=\beta(q,\alpha) \quad \text{ $\mu_{q,\alpha}$-a.e.}
\end{equation}
Note that when the assumptions of Proposition~\ref{qalpha} hold, if $\alpha\in \widering{I}_X$ and $q=q_\alpha$, one has $\nabla \wt P_{X,\phi,\alpha} (q)=0$, hence  \eqref{nabla} yields $\beta(q,\alpha)=\alpha$ and $\widetilde P_{X,\phi,\alpha}^*(\nabla\widetilde P_{X,\phi,\alpha}(q))=\widetilde P_{X,\phi,\alpha}^*(0)=\widetilde P_{X,\phi,\alpha}(q_\alpha)$. 
Moreover, for all $(q,\alpha)\in J_{X,\phi}$, with probability 1, $\lambda(q,\alpha)=\lim_{n\to\infty} \frac{S_n\phi(t)}{n}$ at $\mu_{q,\alpha}$-almost every point $t$. Consequently, for all $(q,\alpha)\in J_{X,\phi}$ one has 
\begin{equation} \label{uniformboundSnphi}
-\log (\beta)\le \lambda(q,\alpha)\le -\log (\widetilde\beta),
\end{equation} 
where $\beta$ and $\widetilde\beta$ are taken as in Lemma~\ref{controlSnphi}.

Recall that for  $(q,\alpha,t) \in  \R^d\times \R^d  \times \R$ we set
 $$
\Sigma_\alpha(q,t)=\sum_{i=1}^N \exp (\langle q|X_i-\alpha \rangle -t \phi_i).
$$
Define 
$$
L_\alpha(q,t) = \log \mathbb{E}\big (\Sigma_\alpha(q,t)\big ).
$$ 
One checks that  
\begin{equation}\label{Der1}
\frac{\partial L_\alpha}{\partial q}(q,\wt P_{X,\phi,\alpha}(q))=\beta(q,\alpha)-\alpha,\quad \frac{\partial L_\alpha}{\partial t}(q,\wt P_{X,\phi,\alpha}(q))=-\lambda(q,\alpha), 
\end{equation}
and 
\begin{align}
\label{Der3} \frac{\mathrm{d}L_\alpha((1+ u)q, (1+u)\wt P_{X,\phi,\alpha}(q))}{\mathrm{d}u}(0)=-h(q,\alpha). 
\end{align}

%

Recall the definitions~\eqref{IX} and~\eqref{tIX} of $I_X$ and $\widetilde I_X$ respectively.

\begin{lemma}\label{approxi2}
Let $D$ be a dense subset of $J_{X,\phi}$. For all $\alpha\in I_X$ there exists a sequence $(q_n,\alpha_n)_{n\in\mathbb{N}}  $ of elements of $D$ such that $\displaystyle\lim_{n\to\infty}  \beta(q_n,\alpha_n)=\alpha$ and $\displaystyle\lim_{n\to\infty} \wt P_{X,\phi,\alpha_n}(q_n)-\langle q_n|\nabla \wt P_{X,\phi,\alpha_n}(q_n)\rangle=\wt P^*_{X,\phi,\alpha}(0)$. 

Moreover, if  \eqref{phibound} holds, one can choose $D$ so that it contains a sequence of the form $(q_{\alpha_m},\alpha_m)_{m\ge1}$ such that $\{\alpha_m:m\ge 1\}$ is dense in $\widetilde{I}_X$, and for all $\alpha\in \widetilde{I}_X$ the previous sequence can be chosen so that $q_n=q_{\alpha_n}$. In particular, $\lim_{n\to\infty} (q_n,\alpha_n)=(q_\alpha,\alpha)$. 
\end{lemma}

\begin{proof}
Let $\alpha\in I_X$.  It is not hard to adapt the proof of \cite[Proposition 2.2]{AB} to show that due to~\eqref{pP}, the mapping $\widetilde P_{X,\phi,\alpha}$ is strictly convex, and the set $I_{X,\phi,\alpha}=\{\beta\in\R^d:\,  \widetilde P^*_{X,\phi,\alpha}(\beta)\ge 0\}$ is a convex compact set with non-empty interior equal to $\widering I_{X,\phi,\alpha}=\{\nabla\widetilde P_{X,\phi,\alpha}(q): \, q\in\R^d,\, \widetilde P^*_{X,\phi,\alpha}(\nabla\widetilde P_{X,\phi,\alpha}(q))>0\}$.  Now, note that due to Propositions~\ref{compPXphialpha} and~\ref{rn}, one has $\widetilde P_{X,\phi,\alpha}(0)\ge 0$. Let $\beta\in \widering I_{X,\phi,\alpha}$.  The sequence $\beta/n$ belongs to $\widering I_{X,\phi,\alpha}$ and converges~to~$0$.  Since $\wt P^*_{X,\phi,\alpha}$ is upper-semi-continuous and concave, one has $\lim_{n\to\infty} \wt P^*_{X,\phi,\alpha}(\beta/n)= \wt P^*_{X,\phi,\alpha}(0)$. Moreover, there exists a sequence $(q_n)_{n\in\mathbb{N}}$ such that $\nabla \wt P_{X,\phi,\alpha}(q_n)=\beta/n$ and  $\wt P_{X,\phi,\alpha}(q_n)-\langle q_n|\nabla \wt P_{X,\phi,\alpha}(q_n)\rangle>0$ so that $(q_n,\alpha)\in J_{X,\phi}$. Also, due to \eqref{nabla}, \eqref{uniformboundSnphi}, and  the fact that $\displaystyle\lim_{n\to\infty}\nabla \wt P_{X,\phi,\alpha}(q_n)=0$, one has $\displaystyle\lim_{n\to\infty}  \beta(q_n,\alpha)=\alpha$ (remember \eqref{alphaXphi}). Now, the mappings $(q,\alpha)\mapsto \beta(q,\alpha)$ and $(q,\alpha)\mapsto\wt P_{X,\phi,\alpha}(q)-\langle q|\nabla \wt P_{X,\phi,\alpha}(q)\rangle$ being continuous, the property  claimed about $D$ follows. The second one is clear due to Proposition~\ref{qalpha}.  
 \end{proof}


We can now start the construction of a set of parameters that will be used to define inhomogeneous Mandelbrot measures. Fix a dense subset $D$ of $J_{X,\phi}$, so that if \eqref{phibound} holds, the property claimed in the second assertion of Lemma~\ref{approxi2} holds. Let $(D_j)_{j\geq 1}$ be a non decreasing sequence of non-empty subsets  of $D$  such that $D = \bigcup_{j\ge 1} D_{j}$. Let $(N_j)_{j\ge 0}$ be a sequence of integers such that $N_0=0$, and that we will specify at the end of this section. 
Then let $(M_{j})_{j\geq 0}$ be the  increasing sequence defined as 
\begin{equation}
M_0 = 0 \qquad \text{and}\qquad M_{j}=\displaystyle\sum_{k=1}^{j}N_{k} \text{ for all $j\ge 1$}.
\end{equation}

 For  $n\in \N$, let   $j_{n}$ denote  the unique integer  satisfying 
  $$M_{j_{n}}+1\le  n \le  M_{j_{n}+1}.$$

We will construct a family of random measures indexed by the set  
\[
\mathcal{R}=\{((q_k,\alpha_k))_{k\geq 1} : \forall j \geq 0, \, \exists (q,\alpha)\in D_{j+1},\  \forall \, M_{j}+1\le  k\le  M_{j+1}, \,(q_k,\alpha_k)=(q,\alpha)\}.
\]
Since each $D_j$ is finite, the set $\mathcal{R} $ is compact, once endowed with the natural metric 
$$
d(\varrho ,\varrho') = \displaystyle \sum_{k \geq 1} 2^{-k} \frac{ |q_k- q'_k |+|\alpha'_k-\alpha_k|}{1+|q_k- q'_k |+|\alpha'_k-\alpha_k|}.
$$

For $\varrho=((q_k,\alpha_k))_{k\ge 1} \in \mathcal{R}$ and $n\ge 1$ we will denote by $\varrho_{|n}$ the sequence $((q_k,\alpha_k))_{1\le k\le n}$. 

\subsection*{Specification of  the sequences $(D_j)_{j\ge 1}$ and $(N_j)_{j\ge 1}$} The following tuning of the sequences $(D_j)_{j\ge 1}$ and $(N_j)_{j\ge 1}$ can be skipped at first reading. 

At first, assume that 
\begin{equation}\label{controlonDj}
\forall\, j\ge 1,\ 
 \begin{cases}\#D_j\le j \\\displaystyle \max_{(q,\alpha)\in D_j}\mathbb{E}\big (\mathbf{1}_{\{N=1\}}\exp(\langle q|X_1-\alpha\rangle -\wt P_{X,\phi,\alpha} (q)\phi_1)\big)\le 1-\frac{c_0}{j}
\end{cases}
\end{equation}
for some constant $c_0>0$. This is possible since $\mathbb{E}(N)>1$ and $\mathbb{E}\big (\sum_{i=1}^N\exp(\langle q|X_i-\alpha\rangle -\wt P_{X,\phi,\alpha} (q)\phi_i)\big)=1$ for all $(q,\alpha)\in J_{X,\phi}$.

\medskip

For each $\alpha\in I$ the function $L_\alpha$ is analytic. Denote by $H_{L_\alpha}$ its Hessian matrix. Also, simply denote $\widetilde P_{X,\phi,\alpha}$ by $\widetilde P_\alpha$ in \eqref{mj2} and \eqref{tmj2} below. For each $j\ge 1$, both 
\begin{multline}\label{mj2}
m_j=\displaystyle\max_{t\in [0,1]}\max_{v\in \mathbb S^{d-1}} \max_{(q,\alpha)\in D_{j}} {}^t\begin{pmatrix}v\\0\end{pmatrix}H_{L_\alpha}(q+tv,\widetilde P_\alpha(q))\begin{pmatrix}v\\0\end{pmatrix}\\
+\displaystyle\max_{t\in [0,1]}\max_{v\in\mathbb S^{d-1}} \max_{(q,\alpha)\in D_{j}} \frac{\partial^2}{\partial t^2} L_\alpha(q,\widetilde P_\alpha(q)+tv)
\end{multline}
and
\begin{equation}\label{tmj2}
\widetilde m_j=\displaystyle\max_{t\in [0,1]}\max_{p\in [1, 2]} \max_{(q,\alpha)\in D_{j}} {}^tV_{q,\alpha}H_{L_\alpha}\left (q+t(p-1)q,\widetilde P_\alpha(q)+t(p-1)\widetilde P_\alpha(q)\right)V_{q,\alpha}
\end{equation}
are finite, where $V_{q,\alpha}=\begin{pmatrix}q\\ \widetilde P_\alpha(q)\end{pmatrix}$.  Let 
\begin{equation}\label{mhatj}
\widehat m_j= \max (m_j,\wt m_j)
\end{equation}
and  $(\gamma_j)_{j\ge 1}\in (0,1]^{\mathbb N}$ be a positive  sequence such that 
\begin{equation}\label{controlgammaj}
\gamma_j^2\widehat m_j\le 1/j^2 \quad\text{(note that $\lim_{j\to\infty}\gamma_j=0$)}.
\end{equation}

Let $(\widetilde p_j)_{j\ge 1}$ be a sequence taking values in $(1,2)$ such that 
$$
\lim_{j\to\infty}(\widetilde p_j-1) \widetilde m_j=0.
$$ 
Due to \eqref{finiteness2} we can also suppose that $\widetilde p_j$ is small enough so that we also have 
$$
\sup_{(q,\alpha)\in D_{j}} \mathbb{E}(\Sigma_\alpha(q,\widetilde P_{X,\phi,\alpha}(q))^{\tilde p_j})<\infty.
$$

For each $(q,\alpha)\in J_{X,\phi}$  there exists   $ p_{q,\alpha}\in(1,2) $ such that 
$\displaystyle L_\alpha (pq,p\widetilde P_{X,\phi,\alpha} (q)) <  0$ for all $p\in (1,p_{q,\alpha})$. Indeed,   $\widetilde{P}_{X,\phi,\alpha}^{*}(\nabla\widetilde{P}_{X,\phi,\alpha}(q)) > 0$ if and only if   $\frac{\mathrm{d}}{\mathrm{d}p} (\displaystyle L_\alpha(pq,p\widetilde P_{X,\phi,\alpha} (q)))  (1^+) < 0$, and $L_\alpha(q,\widetilde P_{X,\phi,\alpha}(q))=0$ by definition of $\widetilde P_{X,\phi,\alpha}(q)$. 

\medskip

For all $j\ge 1$, set 
$$
p_j=\min\left  (\widetilde p_j,\inf_{(q,\alpha)\in D_{j+1}} p_{q,\alpha}\right ) \quad\text{and}\quad L_j=\sup_{(q,\alpha)\in D_{j}} \displaystyle L_\alpha(p_jq,p_j\wt P_{X,\phi,\alpha}(q)).
$$ 
By construction, one has $a_j<0$. Then let 
\begin{equation}\label{rj2}
s_j=\max\left \{ \left\|\Sigma_\alpha(q,\wt P_{X,\phi,\alpha}(q)))\right \|_{p_j}: (q,\alpha)\in D_{j}\right\}\quad \text{and}\quad r_j=\max\left (\frac{L_j}{p_j}, \frac{1-p_j}{2jp_j}\right ).
\end{equation}
\medskip

Now set $N_0=0$ and for $j\ge 1$ choose an integer $N_j$ big enough so that 
\begin{equation}\label{control'}
\quad \frac{(j+1)! s_{j+1} }{1-\exp( r_{j+1})}\exp (N_{j} r_{j+1})\le j^{-2},
\end{equation}
\begin{equation}\label{control2'}
\frac{(j+1)! s_{j+1}}{ (1-\exp( r_{j+1}))}+\frac{(j+2)! s_{j+2}}{ (1-\exp( r_{j+2}))}\le C_0\exp (N_j \gamma_{j+1}^2 m_{j+1}),
\end{equation}
with $C_0= \frac{ 2s_{1}}{ 1-\exp( r_{1})}+\frac{2s_{2}}{1-\exp( r_{2})}$, 
\begin{equation}\label{control0'}
N_j\ge \max \big (\log((j+1)!)^3, (\gamma_{j+1}^2\widehat m_{j+1})^{-2}),
\end{equation}
and if $j\ge 2$,
\begin{equation}\label{control2bis2}
\sum_{k=1}^{j-1}N_k\le  \frac{N_j}{j}\frac{\min(1, \{ \|\beta(q,\alpha)\|: (q,\alpha)\in D_{j}\})}{\max(1,\max\{ \|\beta(q,\alpha)\|: (q,\alpha)\in D_{j-1}\})}.
\end{equation}

\subsection{Inhomogeneous Mandelbrot martingales indexed by  fractal percolation parameters and by $\mathcal{R}$}\label{MMP}
For each $\beta \in (0, 1]$, let $\widetilde W_\beta$ be a random variable distributed according to $\beta\delta_{\beta^{-1}} +(1-\beta)\delta_0$. Consider $\{\widetilde W_{\beta,u}\}_{u\in \bigcup_{n\ge 0} \mathbb N^n}$ be a family of independent copies of $\widetilde W_\beta$ define on a probability space  $(\Omega_\beta,{\mathcal A}_\beta,{\mathbb P}_\beta)$.  

Each random variable $\widetilde W_{\beta,u}$ and the random vector $(N_{u},(X_{ui},\phi_{ui})_{i\ge 1})$ extends to $(\Omega_\beta\times\Omega,\mathcal A_\beta\otimes\mathcal A,{\mathbb P}_\beta\otimes{\mathbb P})$ as 
\begin{equation*}
\widetilde W_{\beta,u}(\omega_\beta,\omega)=\widetilde W_{\beta,u}(\omega_\beta)\text{ and }
(N_{u},(X_{ui},\phi_{ui})_{i\ge 1}\big )(\omega_\beta,\omega)=\big (N_{u}(\omega),(X_{ui}(\omega),\phi_{ui}(\omega))_{i\ge 1}\big ),
\end{equation*}
and the families $\{\widetilde W_{\beta,u}\}_{u\in \bigcup_{n\ge 0} \mathbb N^n}$ and $\{(N_u,(X_{ui},\phi_{ui})_{i\ge 1})\}_{u\in \bigcup_{n\ge 0} \mathbb N^n}$ are $\mathbb{P}_\beta\otimes \mathbb P$-independent.

We adopt the convention that $\mathbb E_{\mathbb P_\beta\otimes \mathbb P}$ is denoted by $\mathbb E$.

For each $\beta\in ( (\mathbb{E}(N))^{-1},1]$, the random integers $N_{\beta,u}(\omega_\beta,\omega)=\sum_{i=1}^{N_u(\omega)}\mathbf{1}_{\{\beta^{-1}\}}(\widetilde W_{\beta,ui}(\omega_\beta))$ define a new supercritical Galton-Watson process to which are associated trees $\TT_{\beta,n}\subset \TT_{n}$ and $\TT_{\beta,n}(u)\subset \TT_n(u)$, $u\in \bigcup_{n\ge 0} \mathbb N^n$, $n\ge 1$, as well as the infinite tree $\TT_\beta\subset \TT$ and the boundary $\partial \TT_\beta \subset \partial\TT$ conditionally on non extinction of $\TT_\beta$.

For $u\in \bigcup_{n\ge 0} \mathbb N^n$, $1\le i\le N_u$, $\beta> \mathbb{E}(N)^{-1}$,  and $\varrho = (q_k,\alpha_k)_{k\geq 1}\in \mathcal R$  set
\begin{equation*}
\left\lbrace
\begin{aligned}
W_{\varrho,ui}&=\exp\big (\langle q_{|u|+1}|X_{ui}-\alpha_{{|u|+1}}\rangle -\wt P_{\alpha_{|u|+1}} (q_{|u|+1})\phi_{ui}\big ),\\
W_{\beta,\varrho,ui}&=  \widetilde W_{\beta,ui}\cdot W_{\varrho,ui}.
\end{aligned}
\right. 
\end{equation*}

Also, for $\varrho= (q_k,\alpha_k)_{k\geq 1} \in  \mathcal{R}$, $u\in \bigcup_{n\ge 0}\mathbb N^n$, $\beta >\mathbb{E}(N)^{-1}$, and $n\ge 0$ define
 \begin{equation*}
Y_n(\varrho, u)= \displaystyle\sum_{v_1\cdots v_n \in \TT_n(u)} \displaystyle\prod_{k=1}^n W_{\varrho,u\cdot v_1\cdots v_k}
\text{ and }Y_n(\beta,\varrho, u)= \displaystyle\sum_{v_1\cdots v_n \in \TT_n(u)} \displaystyle\prod_{k=1}^n W_{\beta,\varrho,u\cdot v_1\cdots v_k}.   
\end{equation*}
When $u=\epsilon$, those quantities will be denoted by $Y_n( \varrho)$ and $Y_n(\beta,\varrho)$ respectively, and when $n=0$, their values equal 1.   

\medskip

Recall the definition of $h(q,\alpha)$ given in \eqref{entropie}. For $\beta \in (\E(N)^{-1},1]$, $\ell\in\mathbb N$ and $\varepsilon > 0$, set 
\begin{equation}\label{Rep}
  \mathcal{R}(\beta, \ell,\varepsilon) = \Big \{  \varrho \in  \mathcal{R} : \frac{1}{n} \displaystyle\sum_{k=1}^{n} h(q_k,\alpha_k)\geq -\log \beta  + \varepsilon , \forall n \geq \ell \Big \}, 
 \end{equation} 
which is a compact subset of $ \mathcal{R}$.

Notice that $h(q_k,\alpha_k)>0$, and this number  is the opposite of the derivative at $1$ of the convex function $f:\lambda\ge 0\longmapsto \log \mathbb{E}(\sum_{i=1}^NW_i^\lambda)$, with $W_i=\exp(\langle q_{k}|X_{i}-\alpha_{{k}}\rangle -\wt P_{\alpha_{k}} (q_{k})\phi_{i})$, so that $f(1)=0$ and $f(0)=\log \E(N)>0$. Thus   $h(q_k,\alpha_k)\in (0,\log \E(N)]$. Consequently, 
\begin{equation}\label{JbetaY}
\Big \{ \varrho\in \mathcal{R}: \liminf_{n\to\infty }\frac{1}{n} \displaystyle\sum_{k=1}^{n} h(q_k,\alpha_k)>0\Big \}=\bigcup_{\beta\in (\E(N)^{-1},1],\ell\ge 1, \varepsilon>0}  \mathcal{R}(\beta, \ell,\varepsilon) .
\end{equation}

For $n\ge 1$ and $\beta\in (0,1]$, set ${\mathcal{F}}_{n}=\sigma \big ((N_{u},(X_{u1},\phi_{u1},(X_{u2},\phi_{u2}),\ldots): u\in \bigcup_{k=0}^n \mathbb N^{n-1}\big )$  and ${\mathcal{F}}_{\beta,n}=\sigma \big (\widetilde W_{\beta,u1},(\widetilde W_{\beta,u2},\ldots): u\in \bigcup_{k=0}^n \mathbb N^{n-1}\big )$. Set ${\mathcal{F}}_{0}={\mathcal{F}}_{\beta,0}=\{\emptyset,\Omega\}$.

\subsection{Construction of inhomogeneous Mandelbrot measures indexed~by~$\mathcal{R}$}\label{sec3.3}
\medskip
The following statement about the simultaneous construction of inhomogeneous Mandelbrot measures is similar to that obtained in  \cite{AB} for a different family.  We include the proof, as the estimates to follow are important to derive Proposition~\ref{pro-3.5}, which will be applied in the study of the action of percolation processes on these measures. Also, these estimates yield Lemma~\ref{bettercontrol2} which will be useful in the proof of Proposition~\ref{LD51}.
\begin{proposition}\label{pro-3.4} $\ $
\begin{enumerate}
\item For all $u\in \bigcup_{n\ge 0}\mathbb N^n$, the sequence of continuous functions  $(Y_n(\cdot,u))_{n\in\mathbb{N}} $  converges uniformly on $\mathcal{R}$, almost surely and in $L^1$ norm,  to a positive limit $Y(\cdot,u)$. 
\item With probability 1, for all $ \varrho\in \mathcal{R}$, the mapping defined on the cylinders of $\N^N$ by 
$$\mu_\varrho ([u])=Y(\varrho,u)  \cdot \displaystyle\prod_{k=1}^{|u|} W_{\varrho,u_1\cdots u_k} , \ u\in \bigcup_{n\ge 0}\mathbb N^n
$$
extends to a positive Borel measure on $\N^\N$ supported on  $\partial\TT$.  
\end{enumerate}
\end{proposition}
\begin{lemma}\cite{vB-E}\label {lem-2.1}
 Let $(X_j)_{j\geq 1}$ be a sequence of centered independent  real random variables. For every finite  $I\subset\mathbb N$ and $p\in (1, 2]$, one has $\E\big (\big |\sum_{i\in I} X_{i}\big|^p\big ) \leq 2^{p-1}  \sum_{i\in I}  \E( \left|X_{i}\right|^p)$.  
\end{lemma}

\begin{lemma} \label{lem-2.2'}
Let $ \varrho\in \mathcal{R}$ and $\beta\in (0,1]$. Define  $Z_n(\beta,\varrho)=  Y_n(\beta,\varrho)- Y_{n-1}(\beta,\varrho)$ for $n\ge 0$. Recall the definition \eqref{sigma} of $\Sigma_\alpha$. For every $p\in (1,2)$ one has (writing $\widetilde P_{\alpha_n}$ for $\widetilde P_{X,\phi,\alpha_n}$)
\begin{equation} \E( |Z_n(\beta,\varrho)|^{p}) \leq (2\beta^{-1})^p\mathbb{E}\big (\Sigma_{\alpha_n}(q_n,\widetilde P_{\alpha_n} (q_n))^p\big )\displaystyle\prod_{k=1}^{n-1}\beta^{1-p}\exp \big( L_{\alpha_k} (pq_k,p\widetilde P_{\alpha_k} (q_k)) \big ).
\end{equation}
\end{lemma} 
 \begin{proof}
Fix $p\in (1,2)$. Setting $A_{\beta,\varrho,u}= \displaystyle\sum_{i=1}^{N_u} \widetilde W_{\beta, ui} W_{\varrho,ui} $ and using the branching property we can write 
$
Z_n(\beta,\varrho) = \displaystyle\sum_{u\in T_{n-1}} \Big (\prod_{k=1}^{n-1} \widetilde W_{\beta,u_1\cdots u_k} W_{\varrho,u_1\cdots u_k} \Big ) (A_{\beta,\varrho,u}- 1).
$
By construction, the random variables $(A_{\beta,\varrho,u}-1)$,  $u\in\N^{n-1}$, are centered and i.i.d., and  independent of $\mathcal{F}_{\beta,n-1}\otimes\mathcal{F}_{n-1}$. Consequently, conditionally on $\mathcal{F}_{\beta,n-1}\otimes\mathcal{F}_{n-1}$, we can apply Lemma~\ref{lem-2.1} to the  $\{ A_{\beta,\varrho,u} \prod_{k=1}^{n-1} \widetilde W_{\beta,u_1\cdots u_k} W_{\varrho,u_1\cdots u_k} \}_{u\in \TT_{n-1}}$. Since the $A_{\beta,\varrho,u}$, $u\in \N^{n-1}$, have the same distribution, this yields, with $A_{\beta,\varrho}=A_{\beta,\varrho,\epsilon}$:
$$
\E(\left|Z_n(\beta,\varrho)\right|^{p}| {\mathcal{F}}_{\beta,n-1} \otimes {\mathcal{F}}_{n-1}) \leq 2^{p-1}\mathbb{E}(|A_{\beta,\rho}-1|^p) \sum_{u\in \TT_{n-1}}\prod_{k=1}^{n-1} \widetilde W_{\beta,u_1\cdots u_k} ^pW_{\varrho,u_1\cdots u_k} ^p.
$$
Also, $\mathbb{E}(|A_{\beta,\rho}-1|^p) \le 2\mathbb{E}(A_{\beta,\rho}^p)$, since $\mathbb{E}(A_{\beta,\rho})=1$ and $p\ge 1$, and  as $0\le \widetilde W_{\beta,i}\le \beta^{-1}$,   $A_{\beta,\rho}\le \beta^{-1} \Sigma_{\alpha_n} \big(q_n,\widetilde P_{X,\phi,\alpha_n} (q_n)\big)$. 
Thus, $2^{p-1}\mathbb{E}(|A_{\beta,\rho}-1|^p) \le (2\beta^{-1})^p\mathbb{E}\big (\Sigma_{\alpha_n}(q_n,\widetilde P_{X,\phi,\alpha_n} (q_n))^p\big )$. Moreover, a recursive using of the branching property and the independence of the random vectors $(N_u,X_{u1}, \phi_{u1}, \ldots)$ and random variables $\widetilde W_{\beta,u}$ used in the constructions yields, setting $W_{q_k,i}=  \exp (\langle q_{k}|X_{i} -\alpha_k\rangle - \widetilde P_{X,\phi,\alpha_k}(q_k)\phi_i)$:
\begin{align*}
\mathbb{E}\Big (\sum_{u\in \TT_{n-1}}\prod_{k=1}^{n-1} \widetilde W_{\beta,u_1\cdots u_k} ^pW_{\varrho,u_1\cdots u_k} ^p\Big )&=\prod_{k=1}^{n-1} \mathbb{E} (\widetilde W_\beta^p)\mathbb{E}\Big (\sum_{i=1}^N W_{q_k,i}^p\Big )\\&=\displaystyle\prod_{k=1}^{n-1}\beta^{1-p}\exp \big( L_{\alpha_k} (pq_k,p\widetilde P_{X,\phi,\alpha_k} (q_k)) \big ).
\end{align*}
Collecting the previous estimates one gets the desired conclusion.
 \end{proof}
 
 \begin{proof}[Proof of Proposition~\ref{pro-3.4}] (1) It is similar to the proof of~\cite[Proposition 2.8(1)]{AB}. We detail it for reader's convenience. 
 
First, consider the case $u=\epsilon$. 
 Observe that for all $n\ge 1$,  by construction the function  $Y_n(\cdot)=Y_n(\cdot,\emptyset)$ is continuous and constant over the set of those sequences $\varrho$ having the same first terms. 
Given $n\ge 1$ and $ \varrho\in \mathcal{R}$,  since $M_{j_n}+1\le n\le M_{j_n+1}$, applying Lemma~\ref{lem-2.2'} with $p=p_{j_n+1}$ and $\beta=1$, one obtains
\begin{align*}
&\|Y_n(\varrho)-Y_{n-1}(\varrho)\|^{p_{j_n+1}}_{p_{j_n+1}}\\
&\le  2^{p_{j_n+1}}\mathbb{E}\big (\Sigma_{\alpha_n}(q_n,\widetilde P_{X,\phi,\alpha_n} (q_n))^{p_{j_n+1}}\big )\displaystyle\prod_{k=1}^{n-1}\exp \big( L_{\alpha_k} ({p_{j_n+1}}q_k,p\widetilde P_{X,\phi,\alpha_k} (q_k)) \big )\\
&\leq  2^{p_{j_n+1}} s_{j_n+1}^{p_{j_n+1}}\displaystyle\prod_{k=1}^{n-1}\exp \big( \sup_{(q,\alpha)\in D_{j_n+1}} L_{\alpha} ({p_{j_n+1}}q,p\widetilde P_{X,\phi,\alpha_k} (q)) \big )\\
& \qquad(\text{since $\{(q_k,\alpha_k): 1\le k\le n\}\subset D_{j_n+1}$}) \\
&\le  2^{p_{j_n+1}}s_{j_n+1}^{p_{j_n+1}} \exp ((n-1)p_{j_{n}+1}r_{j_n+1}) \ (\text{due \eqref{rj2}});
\end{align*}
this bound is independent of $\varrho$. Note that by definition of $\mathcal{R}$,  $\#\{\varrho_{|n}: \varrho\in \mathcal{R}\}=\prod_{j=1}^{j_n+1}\#D_j\le (j_n+1)!$; also  $Y_n(\varrho)-Y_{n-1}(\varrho) $ only depends on $\varrho_{|n}=((q_1,\alpha_1),\cdots, (q_n,\alpha_n))$. Consequently,
\begin{eqnarray*}
\big  \|\|Y_n(\cdot)-Y_{n-1}(\cdot)\|_\infty\big\|_1\le \sum_{\varrho_{|n}}\|Y_n(\varrho)-Y_{n-1}(\varrho)\|_{p_{j_n+1}}
 \le  2 (j_n+1)!s_{j_n+1} \exp ((n-1)r_{j_n+1}).
 \end{eqnarray*}
 This yields
\begin{eqnarray*}
 \sum_{n\in\mathbb{N}}\big  \|\|Y_n(\cdot)-Y_{n-1}(\cdot)\|_\infty\big\|_1&\le& \sum_{j\ge 0}\sum_{M_j+1\le n\le M_{j+1}} 2(j+1)!s_{j+1}  \exp ((n-1)r_{j+1})\\
 &\le& \sum_{j\ge 0} 2  (j+1)!s_{j+1}  \frac{\exp (M_{j}r_{j+1})}{1-\exp (r_{j+1})}<\infty,
 \end{eqnarray*}
where we used \eqref{control'}.  If follows that $(Y_n)_{n\in\mathbb{N}}$ converges uniformly, almost surely and in $L^1$ norm, to a function $Y$, as $n\to\infty$. 

Let us show that $Y$ does not vanish on $\mathcal{R}$ almost surely. 
 For each $n\ge 1$, let $\mathcal{R}_{X,\phi |n}=\{\varrho_{|n}: \varrho\in \mathcal{R} \}$, and for $\gamma\in \mathcal{R}_{X,\phi |n}$ define the event $\mathcal N_\gamma = \{ \omega\in\Omega: \exists  \varrho\in \mathcal{R},\  Y(\varrho)=0,\ \varrho_{|n}=\gamma \}$.  Let $\mathcal N=\{ \omega\in\Omega: \exists  \varrho\in \mathcal{R},\  Y(\varrho)=0\}$. Since the functions $Y_n$ are almost surely positive, this event is a tail event, and it has probability 0 or 1. The same property holds for the events $\mathcal N_\gamma$, $\gamma\in \bigcup_{n\in\mathbb{N}}\mathcal{R}_{X,\phi |n}$.  Suppose that $\mathcal N$ has probability 1. Since $\mathcal N=\bigcup_{\varrho_1\in \mathcal{R}_1} \mathcal N_{(\gamma_1)}$, necessarily, there exists $\gamma_1\in\mathcal{R}_1$ such that $\mathbb P(\mathcal N_{(\gamma_1)})>0$, and so $\mathbb P(\mathcal N_{(\gamma_1)})=1$. Iterating this remark we can build an infinite deterministic sequence $\gamma=(\gamma_k)_{k\ge 1}\in \mathcal{R}$ such that $\mathbb P(\mathcal N_{(\gamma_1,\ldots,\gamma_n)})=1$ for all $n\ge 1$. This means that almost surely, for all $n\ge 1$, there exists $\varrho^{(n)}\in  \mathcal{R} $ such that $\varrho^{(n)}_{|n}=(\gamma_1,\ldots,\gamma_n)$ and $Y(\varrho^{(n)})=0$. But $\varrho^{(n)}_{|n}=(\gamma_1,\ldots,\gamma_n)$ implies that $\varrho^{(n)}$ converges to~$\gamma$ as $n\to\infty$. Hence, by continuity of $Y$ at $\gamma$, we get $Y(\gamma)=0$ almost surely. However, a consequence of our convergence result for $Y_n$ is that the martingale $Y_n(\gamma)$ converges in $L^1$ to $Y(\gamma)$, so that $\mathbb{E}(Y(\gamma))=1$. This is a contradiction. Thus $\mathbb{P}(\mathcal N)=0$. 
 
\medskip
Now fix any $u\in \bigcup_{n\in\mathbb{N}}\mathbb N^n$. Mimicking what was done for $u=\epsilon$,  for all $n\ge 1$ one  gets
$$
  \big \|\|Y_n(\cdot,u)-Y_{n-1}(\cdot,u)\|_\infty\big \|_1\le  2 (j_{|u|+n}+1)!s_{j_{|u|+n}+1}  \exp \big ((n-1)r_{j_{|u|+n}+1}\big ).
$$
Consequently, setting $a_{j,n}(u)=2 (j_{|u|+n}+1)!s_{j_{|u|+n}+1}  \exp \big ((n-1)r_{j_{|u|+n}+1}\big )$, we can get
\begin{align*}
& \sum_{n\in\mathbb{N}}\big \|\|Y_n(\cdot,u)-Y_{n-1}(\cdot,u)\|_\infty\big \|_1
\\ &\le \sum_{n=1}^{M_{j_{|u|}+1}-|u|} a_{j_{|u|},n}(u)+ \sum_{j\ge j_{|u|}+1}\sum_{M_j+1\le |u|+n\le M_{j+1}} a_{j,n}(u)\\
 &\le  \sum_{j_{|u|}\le j\le j_{|u|}+1} \frac{2 (j+1)!s_{j+1}}{1-\exp (r_{j+1})}+ \sum_{j\ge j_{|u|}+2} 2 (j+1)!s_{j+1}  \frac{\exp ((M_{j}-|u|)r_{j+1})}{1-\exp (r_{j+1})}.
 \end{align*}
 Note that due to the inequalities $M_{j_{|u|}}+1\le |u|\le M_{j_{|u|}+1}$, for  $j\ge j_{|u|}+2$ one has  $M_j-|u|\ge N_j$, so 
 \begin{eqnarray}
\nonumber&& \sum_{n\in\mathbb{N}}\big \|\|Y_n(\cdot,u)-Y_{n-1}(\cdot,u)\|_\infty\big \|_1\\
\label{sharp}&\le& \sum_{j_{|u|}\le j\le j_{|u|}+1} \frac{2 (j+1)!s_{j+1}}{1-\exp (r_{j+1})}+ \sum_{j\ge j_{|u|}+2} 2 (j+1)!s_{j+1}  \frac{\exp (N_jr_{j+1})}{1-\exp (r_{j+1})}\\
\nonumber & \le& 2 C_0\exp (N_{j_{|u|}}\gamma_{j_{|u|}+1}^2m_{j_{|u|}+1})+2 \sum_{j\ge j_{|u|}+2}j^{-2},
 \end{eqnarray}
where  \eqref{control'} and \eqref{control2'} have been used. This implies the desired convergence to a limit $Y(\cdot, u)$. The set $\bigcup_{k\ge 0}\mathbb N^k$ being countable, the convergence holds also  almost surely, simultaneously for all $u$, and  the almost sure positivity of $Y(\cdot, u)$ is proven by using the same argument as for $u=\epsilon$.

Finally, recall the definition \eqref{mhatj} of $(\widehat m_{j})_{j\ge 1}$ and set $\varepsilon_k=\gamma_{j_{k}+1}^2\widehat m_{j_{k}+1}$ for all $k\ge 0$. The previous calculations, together with  the fact that $Y_0(\cdot, u)=1$ for all $u\in \bigcup_{k\ge 0}\mathbb N^k$ and the inequality  $|u|\ge N_{j_{|u|}}$ imply the existence of a constant $C_{X,\phi}$ such that:
\begin{equation}\label{control3'}
\|\sup_{ \varrho\in \mathcal{R} }Y(\varrho, u)\|_1\le C_{X,\phi} \exp (\varepsilon_{|u|} N_{j_{|u|}} ) \le C_{X,\phi}  \exp (\varepsilon_{|u|} |u|) \quad (\forall\ u\in \bigcup_{k\ge 0}\mathbb N^k).
\end{equation}

\noindent(2) This follows from the branching property. \end{proof}

Estimates similar to the previous ones yield the following lemma, which will be used in the proof of Proposition~\ref{LD51}.

\begin{lemma}\label{bettercontrol2}
Let $K$ be a compact subset of $J_{X,\phi}$ containing the unique element of~$D_1$. There exists $p_K\in (1,2)$ such that  $
\sup_{j\ge 1} \sup_{(q,\alpha)\in D_{j}\cap K}L_\alpha (p_Kq ,p_K\widetilde P_{X,\phi,\alpha}(q))<0$ and $\sup_{j\ge 1} \sup_{(q,\alpha)\in D_{j}\cap K}
\mathbb{E}\big (\Sigma_{\alpha}(q,\widetilde P_{X,\phi,\alpha} (q))^{p_K}\big )<\infty$. Set $\mathcal{R}(K)=\{ \varrho\in \mathcal{R} : \ \forall \ k\ge 1,\ (q_k,\alpha_k)\in K\}$. One has  
$
\|\sup_{\varrho\in \mathcal{R} (K)}Y(\varrho, u)\|_{p_K}= O((j_{|u|}+2)!).
$
\end{lemma}

\subsection{Lower bounds for the Hausdorff dimensions of the measures $\mu_\varrho$ via percolation}\label{LBM}
Let us recall the definition of the lower Hausdorff dimension of a measure and its characterisation in terme of lower local dimension (see \cite{Fan} for instance). 
\begin{definition}Let $(\mathcal Z,d)$ be a compact metric space and $\mu$ a finite Borel measure on $\mathcal Z$. Then, the lower Hausdorff dimension of $\mu$ is defined as 
\begin{equation*}
\underline\dim(\mu)=\inf\{\dim E: \, E\in \mathcal B(\mathcal Z),\, \mu(E)>0\}. 
\end{equation*}
\end{definition}
\begin{lemma}\label{characlhd} Let $(\mathcal Z,d)$ be a compact metric space. Then 
$$
\underline\dim(\mu)={\mathrm{ess\,inf}}_\mu \liminf_{r\to0^+}\frac{\log (\mu(B(z,r)))}{\log(r)}. 
$$
\end{lemma}
The goal of this section is to prove the following result.  
\begin{theorem}\label{lb2}
With probability 1, for all $ \varrho\in \mathcal{R}$, 
\begin{align*}
 \underline\dim (\mu_\varrho)\ge \liminf_{n\to\infty}\frac{\sum_{k=1}^nh(q_k,\alpha_k)}{ \sum_{k=1}^n\lambda(q_k,\alpha_k)}.
\end{align*}
\end{theorem}

We need the following two propositions. Recall the definition \eqref{Rep} of the set of parameters $ \mathcal{R}(\beta,\ell,\varepsilon)$.

\begin{proposition}\label{pro-3.5} Let $\beta\in ((\mathbb{E}(N))^{-1},1]$. 
Conditionally on non extinction of $(\TT_{\beta,n}(u))_{n\in\mathbb{N}}$, for all $\ell\geq 1$ and $\varepsilon \in\mathbb Q^*_+$,
 
 \begin{enumerate}
 \item   the sequence of continuous functions $(Y_n(\cdot ,\beta))_{n\in\mathbb{N}}$ converges uniformly, almost surely and in $L^1$ norm,  to a positive limit $Y(\beta,\cdot)$ on   $ \mathcal{R}(\beta,\ell,\varepsilon)$;
 
\item  the sequence of continuous functions 
$$
\Big(\varrho\mapsto \widetilde Y_n(\beta,\varrho)=\sum_{u\in \TT_n}(\displaystyle\prod_{k=1}^n \widetilde W_{\beta,u_1\cdots u_k} )\mu_\varrho([u])\Big)_{n\in\mathbb{N}}
$$
converges uniformly, almost surely and in $L^1$ norm,    to  $Y(\beta,\cdot)$ on   $ \mathcal{R}(\beta, \ell,\varepsilon)$.
\end{enumerate}
\end{proposition}

\begin {proposition} \label{pp33}$\ $
\begin{enumerate} 
\item With probability~$1$, for all $\varrho =(q_k,\alpha_k)_{k\geq 1} \in  \mathcal{R}$,  for $\mu_\varrho$-almost all  $t \in\partial\TT$, for $n$ large enough, one has  
 $$\lim_{n\to\infty} n^{-1}\Big |\displaystyle \log\Big (\prod_{k=1}^n W_{\varrho,t_1\cdots t_n}\Big )- \sum_{k=1}^{n} h(q_k,\alpha_k) \Big |=0=\displaystyle \lim_{n\to\infty}n^{-1}\Big |\displaystyle S_n\phi(t) - \sum_{k=1}^{n} \lambda(q_k,\alpha_k) \Big |.$$
 \item With probability 1, for all $ \varrho\in \mathcal{R}$, for $\mu_\varrho$-almost every $t\in\partial\TT$, one has 
$$
\lim_{n\to\infty} \frac{\log(\mathrm{diam}([t_{|n}]))}{-S_n\phi(t)}=1,
$$
where the diameter is measured with respect to the metric $\mathrm{d}_\phi$.
\end{enumerate} 
\end{proposition}

\begin{proof}[Proof of Theorem~\ref{lb2}]
Let $\beta\in (0,1]$ such that $\beta\mathbb{E}(N)>1$. Let $\ell\ge 1$ and $\varepsilon \in\mathbb Q^*_+$.  

For every $t\in\partial \TT$ and $\omega_\beta\in \Omega_\beta$ set 
$$
Q_{\beta,n}(t,\omega_\beta)=\prod_{k=1}^n \widetilde W_{\beta,t_{|k}}, 
$$
so that for $\varrho \in\mathcal{R}(\beta,\ell,\varepsilon)$, $\widetilde Y_{n}(\beta,\varrho)$ is the total mass of the measure $Q_{\beta,n}(t,\omega_\beta)\cdot \mathrm{d}\mu_\varrho^\omega(t)$.  

There exists a measurable subset $\Omega(\beta,\ell,\varepsilon)$ of $\Omega$, such that $\mathbb{P}(\Omega(\beta,\ell,\varepsilon))=1$ and for all $\omega\in \Omega(\beta,\ell,\varepsilon)$, there exists $\Omega_\beta^\omega\subset \Omega_\beta$ of positive probability such that for all $\omega\in \Omega(\beta,\ell,\varepsilon)$, for all $\varrho\in \mathcal{R}(\beta,\ell,\varepsilon)$, for all $\omega_\beta\in \Omega_\beta^\omega$, $\widetilde Y_{n}(\beta,\varrho)$ does not converge to $0$. In terms of the multiplicative chaos theory developed in \cite{K1}, this means that for all $\omega\in \Omega(\beta,\ell,\varepsilon)$ and $\varrho\in \mathcal{R}(\beta,\ell,\varepsilon)$, the set of those $\omega_\beta$ such that the multiplicative chaos $(Q_{\beta,n} (\cdot,\omega))_{n\in\mathbb{N}}$ has not killed the measure $\mu_\varrho$ on the compact set $\partial\TT$ has a positive $\mathbb{P}_\beta$-probability.  Moreover, under the metric $\mathrm{d}_1$, for any ball $B$ in $\partial \TT$, there exists $n\ge 0$ and $u\in \TT_n$ such that 
$B=[u]\cap \partial \TT$, $Q_{\beta,n}(t)$ is constant over $B$, and denoting by $|B|_{\mathrm{d}_1}$  the diameter of $B$ under $\mathrm{d}_1$, for any $h\in (0,1)$ we have 
\[
\mathbb{E}_\beta\big ( \displaystyle\sup_{t\in B} (Q_{\beta,n}(t))^h \big )= e^{n(1-h)\log(\beta)}=(| B |_{\mathrm{d}_1})^{-(1-h)\log(\beta)},
\] 
where $\E_\beta$ stands for  the expectation  with respect to $\mathbb{P}_\beta$. Thus, one can apply \cite[Theorem 3]{K1} and obtain that for all $\omega\in \Omega(\beta,\ell,\varepsilon)$ and all $\varrho\in \mathcal{R}(\beta,\ell,\varepsilon)$, the measure $\mu_\varrho$ is not carried by a Borel set of Hausdorff dimension less than  $-\log (\beta)$. 

Let $\Omega'=\bigcap_{\beta\in (\mathbb{E}(N)^{-1},1]\cap \mathbb Q^*_+,\ell\ge 1,\varepsilon\in\mathbb Q^*}\Omega(\beta,\ell,\varepsilon)$. This set is of $\mathbb{P}$-probability 1. 

Let $ \varrho\in \mathcal{R}$ and set  $D_\varrho=\liminf_{n\to\infty}n^{-1}\sum_{k=1}^nh(q_k,\alpha_k)$. If $D_\varrho>0$, by \eqref{JbetaY} there exists a sequence of points $(\beta_n,\ell_n, \varepsilon_n)\in (\mathbb{E}(N)^{-1},1]\times \mathbb N\times \mathbb Q_+^*$ such that  $D_\varrho\ge -\log(\beta_n) +\varepsilon_n/2$ for all $n\ge 1$, $\lim_{n\to\infty} -\log(\beta_n)=D_\varrho$, $\lim_{n\to\infty}\varepsilon_n=0$, and $\varrho\in \bigcap_{n\in\mathbb{N}}\mathcal{R}( \beta_n, \ell_n,\varepsilon_n)$. Consequently, the previous paragraph implies that,  with respect to the metric $\mathrm{d}_1$, for all $\omega\in \Omega'$, $\underline \dim (\mu_\varrho^{\omega}) \ge \limsup_{n\to\infty}-\log(\beta_n)=D_\varrho$.   
In particular, due to Proposition~\ref{pp33}(2) applied to $\phi=(1)_{n\in\N}$ (it is valid in this case as well, due to the proof of this part of the proposition) and Lemma~\ref{characlhd} one has $\displaystyle\liminf_{n\to\infty}\frac{\log \mu_\varrho ([t_{|n}])}{-n}\ge D_\varrho$, $\mu_\varrho$-almost everywhere. 


Note now  that since $\mu_\varrho ([t_{|n}])=Y(\rho,t_1\cdots t_n)\,\prod_{k=1}^{n}W_{\varrho,t_1\cdots t_n}$, we can deduce from the first limit in Proposition~\ref{pp33}(1) that $\limsup_{n\to\infty} n^{-1}\log(Y(\rho,t_1\cdots t_n))\le 0$, $\mu_\varrho$-almost everywhere. Due to the second limit in Proposition~\ref{pp33}(1), this implies that under ${\mathrm d}_\phi$ one has $\displaystyle \liminf_{n\to\infty}\frac{\log \mu_\varrho ([t_{|n}])}{\log(\mathrm{diam}([t_{|n}]))}\ge  \liminf_{n\to\infty}\frac{ \sum_{k=1}^nh(q_k,\alpha_k)}{ \sum_{k=1}^n\lambda(q_k,\alpha_k)}$ (an inequality which holds as well trivially if~$D_\varrho=0$).
\end{proof}

%
%

\begin{proof}[Proof of Proposition~\ref{pro-3.5}] (1) Let $\ell\ge 1$ and $\varepsilon>0$. For $\varrho\in\mathcal{R}(\beta, \ell,\varepsilon)$ and $n\ge 1$, Lemma~\ref{lem-2.2'} applied with $p=p_{j_n+1}$  provides us with the inequality (where $\widetilde P_{\alpha_k}$ stands for $\widetilde P_{X,\phi,\alpha_k}$)
\begin{eqnarray*}
&&\|Y_n(\beta,\varrho)-Y_{n-1}(\beta,\varrho)\|^{{p_{j_n+1}}}_{p_{j_n+1}}\\
&\le& (2\beta^{-1})^{p_{j_n+1}}\mathbb{E}\big (\Sigma_{\alpha_n}(q_n,\widetilde P_{\alpha_n} (q_n))^{p_{j_n+1}}\big )\displaystyle\prod_{k=1}^{n-1}\beta^{1-{p_{j_n+1}}}\exp \big(L_{\alpha_k}({p_{j_n+1}}q_k,{p_{j_n+1}}\widetilde P_{\alpha_k} (q_k)) \big ).
\end{eqnarray*}

Let $(q,\alpha)\in D_{j_n+1}$ and set $g_{q,\alpha}:\lambda\in\mathbb{R}\mapsto L_\alpha(pq,p\widetilde P_{X,\phi,\alpha}(q))$. By construction we have $g_{q,\alpha}(1)=0$ so for $p\in [1,2]$  
 $$ g_{q,\alpha} ( p ) =(p-1)  g_{q,\alpha}'(1) +  (p - 1)^2 \int_{0}^{1} (1-t) g_{q,\alpha}''(1 + t (p- 1) ) \,  \mathrm{d}t,
  $$
with $g_{q,\alpha}'(1)=-h(q,\alpha)$ (see \eqref{entropie} for the definition of $h(q,\alpha)$) and 
\begin{eqnarray*}
g_{q,\alpha}''(1 + t (p - 1) ) &=&   {}^t\begin{pmatrix}q\\ \widetilde P_{\alpha}(q)\end{pmatrix} H_{L_\alpha}\big (q+t(p-1)q,\widetilde P_{\alpha}(q)+t(p-1)\widetilde P_{\alpha}(q)\big )\begin{pmatrix}q\\ \widetilde P_{\alpha}(q)\end{pmatrix}\\
& \le &
  \widetilde m_{j_n+1},
\end{eqnarray*} 
where $(\widetilde m_j)_{j\ge 1}$ is defined in \eqref{tmj2}. Let $\eta_j=(p_j-1) \widetilde m_j$ for $j\ge 1$. By construction of $(p_j)_{j\ge 1}$, one has $\lim_{j\to\infty}\eta_j=0$ and specifying  $p=p_{j_n+1}$ one obtains
\begin{eqnarray*}
L_\alpha ({p_{j_n+1}}q,{p_{j_n+1}}\widetilde P_{X,\phi,\alpha} (q))
 \le (1-p_{j_n+1}) h(q,\alpha)+\eta_{j_n+1}(p_{j_n+1}-1).
 \end{eqnarray*}
We can insert this upper bound in our estimation of $Y_n(\beta,\varrho)-Y_{n-1}(\beta,\varrho)$ and get,  remembering that $\varrho\in \mathcal{R}(\beta, \ell,\varepsilon)$, for $n\ge \ell+1$  
\begin{eqnarray*}
&&\|Y_n(\beta,\varrho)-Y_{n-1}(\beta,\varrho)\|^{p_{j_n+1}}_{p_{j_n+1}}\\
&\le& (2\beta^{-1})^{p_{j_n+1}}s_{j_n+1}^{p_{j_n+1}}\exp\Big ((1-p_{j_n+1}) \sum_{k=1}^{n-1} \log(\beta) +  h(q_k,\alpha_k)  -\eta_{j_n+1}\Big)\\
&\le & (2\beta^{-1})^{p_{j_n+1}}s_{j_n+1}^{p_{j_n+1}}\exp \big ((n-1)  (1-p_{j_n+1}) (\varepsilon-\eta_{j_n+1})\big ).
\end{eqnarray*}
Let $j(\varepsilon)=\min\{j\ge \lfloor \varepsilon^{-1}\rfloor  +1:\eta_j\le \varepsilon/2\}$ and $n_\varepsilon=\min\{n\ge \ell+1: j_{n+1}\ge j(\varepsilon)\}$. For $n\ge n_\varepsilon$ on has, remembering \eqref{rj2},
\begin{align*}
 \|Y_n(\beta,\varrho)-Y_{n-1}(\beta,\varrho)\|^{p_{j_n+1}}_{p_{j_n+1}}
\le(2\beta^{-1})^{p_{j_n+1}}s_{j_n+1}^{p_{j_n+1}}\exp \big ((n-1) p_{j_n+1}r_{j_n+1})\big ).
\end{align*}
Consequently, using the estimates as in the proof of Proposition~\ref{pro-3.4} one gets
\begin{align*}
\sum_{n\ge n_\varepsilon}\Big \|\sup_{\varrho\in \mathcal{R}(\beta,\ell,\varepsilon)}|Y_n(\beta,\varrho)-Y_{n-1}(\beta,\varrho)|\Big \|_1
<\infty.
 \end{align*}
This yields the conclusion about the uniform convergence. The fact that the limit $Y(\beta,\cdot)$ does not vanish almost surely, conditionally on non extinction of $(\TT_{\beta,n})_{n\ge1 }$, follows the same lines as in the study of $Y(\cdot)$, combined with the fact that for a fixed $\varrho\in \mathcal{R}(\beta, \ell,\varepsilon)$, the probability that the limit of $Y_n(\beta,\varrho)$ be 0 equals that of the extinction of $(\TT_{\beta,n})_{n\in\mathbb{N}}$. This comes from the fact that  conditionally on non extinction, the event $\{Y(\beta,\varrho)=0\}$ is asymptotic so has probability 0 or 1, and it has probability 0 since the convergence of $Y_n(\beta,\varrho)$ to $Y(\beta,\varrho)$ holds in~$L^1$. Thus, we have the desired result for a given couple $(\ell,\varepsilon)$; but it holds  simultaneously for all $\ell\ge 1$ and $\varepsilon\in\mathbb Q^*_+$ since  $\mathbb N\times \mathbb Q^*_+$ is countable.

\medskip

\noindent
(2) The approach to follow can be interpreted as a uniform version of  the ``decomposition'' principle of multiplicative cascades on homogeneous trees due Kahane (see \cite{K2}, as well as \cite{WW} for a proof and \cite{KF'} for general multiplicative chaos). 

Fix $\ell\ge 1$ and $\varepsilon>0$. Denote by $E$ the separable Banach space of real valued continuous functions over the compact set $\mathcal{R}(\beta, \ell,\varepsilon)$ endowed with the supremum norm $\|\ \|_\infty$.   

For $n\ge m \ge 1$ and $\varrho\in\mathcal{R}(\beta, \ell,\varepsilon)$ let 
$$
\widetilde Y_{m,n}(\beta,\varrho)=\sum_{u\in \TT_m} Y_{n-m} (\varrho, u)\prod_{k=1}^{m}\widetilde W_{\beta,u_{|k}}W_{\varrho,u_{|k}}.
$$
Notice that $\widetilde Y_{n,n}(\beta,\varrho)=Y_n(\beta,\varrho)$. Moreover, since  $Y_n(\beta,\cdot)$ converges uniformly, almost surely and in $L^1$ norm to $Y(\beta,\cdot)$ as $n\to\infty$, $Y_n(\beta,\cdot)$ belongs to  $L^1_E= L^{1}_E(\Omega_\beta\times \Omega,\mathcal A_\beta\times \mathcal A,{\mathbb  P}_\beta \times{\mathbb  P})$ (where we use the notations of \cite[Section V-2]{Neveu}), so that  the continuous random function $\mathbb{E}(\widetilde Y_{n,n}(\beta,\varrho)|\mathcal F_{\beta,m}\otimes\mathcal F_{n})$ is well defined by \cite[Proposition V-2-5]{Neveu}. Also,  given $\varrho\in \mathcal{R}(\beta, \ell,\varepsilon)$, we can deduce from the definitions and the independence assumptions that 
$$
\widetilde Y_{m,n}(\beta,\varrho)=\mathbb{E}(\widetilde Y_{n,n}(\beta,\varrho)|\mathcal F_{\beta,m}\otimes\mathcal F_{n})
$$
almost surely. Consequently, by \cite[Proposition V-2-5]{Neveu} again, since $e\in E\mapsto e(\varrho)$ is a continuous linear form over $E$, we obtain $
\widetilde Y_{m,n}(\beta,\varrho)=\mathbb{E}(\widetilde Y_{n,n}(\beta,\cdot)|\mathcal F_{\beta,m}\otimes\mathcal F_{n}) (\varrho)$ almost surely. Since given any dense countable subset $\mathcal D$ of $\mathcal{R}(\beta, \ell,\varepsilon)$ this holds simultaneously for all $\varrho\in\mathcal D$, we can conclude  that  the random continuous functions $\widetilde Y_{m,n}(\beta,\cdot)$ and $\mathbb{E}(\widetilde Y_{n,n}(\beta,\cdot)|\mathcal F_{\beta,m}\otimes\mathcal F_{n})$ are equal almost surely. 

Similarly, since for each $\varrho\in \mathcal{R}(\beta, \ell,\varepsilon)$ the martingale $(Y_n(\beta,\varrho),\mathcal F_{\beta,n}\otimes\mathcal F_{n})_{n\in\mathbb{N}}$ converges to $Y(\beta,\varrho)$ almost surely and in $L^1$, and $Y(\beta,\cdot)\in L^1_E$,   by using  \cite[Proposition V-2-5]{Neveu} once more we can get 
\begin{equation}\label{zn}
\widetilde Y_{n,n}(\beta,\cdot)=\mathbb{E}(Y(\beta,\cdot)|\mathcal F_{\beta,n}\otimes\mathcal F_{n}), \text{ hence } \widetilde Y_{m,n}(\beta,\cdot)=\mathbb{E}(Y(\beta,\cdot)|\mathcal F_{\beta,m}\otimes\mathcal F_{n}),
\end{equation}
almost surely. Moreover, it follows from Proposition~\ref{pro-3.4}(1) and the definition of $\mu_\varrho([u])$ that $\widetilde Y_{m,n}(\beta,\cdot)$ converges uniformly, almost surely  and in $L^1$ norm, as $n\to\infty$, to  $\wt Y_m(\beta,\cdot)$. This and \eqref{zn} yield, using \cite[Proposition V-2-6]{Neveu},
$$
\widetilde Y_m(\beta,\cdot)= \lim_{n\to\infty} \widetilde Y_{m,n}(\cdot) =\mathbb{E}\big (Y(\beta,\cdot)|\mathcal F_{\beta,m}\otimes \sigma (\bigcup_{n\in\mathbb{N}} \mathcal F_{n})\big ),
$$
and finally
$$
\lim_{m\to\infty} \widetilde Y_m(\beta,\cdot)=\mathbb{E}\big (Y(\beta,\cdot)|\sigma(\bigcup_{m\ge 1}\mathcal F_{\beta,m})\otimes \sigma (\bigcup_{n\in\mathbb{N}} \mathcal F_{n})\big )=Y(\beta,\cdot)
$$
almost surely (since by construction $Y(\beta,\cdot)$ is $\sigma(\bigcup_{m\ge 1}\mathcal F_{\beta,m})\otimes \sigma (\bigcup_{n\in\mathbb{N}} \mathcal F_{n})$-measurable), where the convergences hold in the uniform norm. 
\end{proof}

\begin{proof}[Proof of Proposition~\ref{pp33}](1) This will be established simultaneously with Proposition~\ref{pp3'} below, which deals with $\R^d$-valued branching random walks. 
(2) Recall \eqref{uniformboundSnphi}. Observe that by construction, for all $t\in\partial \TT$ one has $\mathrm{diam}([t_{|n}]))= \exp(-S_{n+k_n(t)}\phi(t))$, where $k_n(t)=\inf\{k\ge 0: N_{t_{|n+k}}>1\}$. Consequently, if follows from the previous observation and part (1) of the proposition that the property we have to establish will follow if we show that with probability 1, for all $ \varrho\in \mathcal{R}$ one has $k_n(t)=o(n)$ for $\mu_\varrho$-almost every~$t$.

Fix $\eta\in(0,1)$. Denoting by $1^k$ the word $\underbrace{1\cdots1}_{k}$, for all $\varrho\in \mathcal{R}$ one has 
\begin{align*}
\mu_\varrho(\{t\in\partial \TT: \, k_n(t)> \lfloor n\eta\rfloor\})&\le \sum_{|u|=n} \mu_\varrho([u\cdot1^{\lfloor n\eta\rfloor}])\mathbf{1}_{\{N_u=N_{u1}=\ldots=N_{u\cdot 1^{\lfloor n\eta\rfloor-1}}=1\}}\\
&= \sum_{|u|=n} \mu_{\varrho_{|n}}([u])\Big (\prod_{k=0}^{\lfloor n\eta\rfloor-1} \mathbf{1}_{\{N_{u\cdot 1^{k}}=1\}} W_{\varrho,u\cdot 1^{k+1}}\Big ) Y(\varrho,u\cdot 1^{\lfloor n\eta\rfloor}).
\end{align*}
Thus 
\begin{align*}
&\sup_{ \varrho\in \mathcal{R}}\mu_\varrho(\{t\in\partial \TT: \, k_n(t)> \lfloor n\eta\rfloor\})\\
&\le \sum_{\varrho_{|n+\lfloor n\eta\rfloor}: \varrho\in \mathcal{R}} \sum_{|u|=n} \mu_{\varrho_{|n}}([u])\Big (\prod_{k=0}^{\lfloor n\eta\rfloor-1} \mathbf{1}_{\{N_{u\cdot 1^{k}}=1\}} W_{\varrho,u\cdot 1^{k+1}}\Big ) \sup_{ \varrho\in \mathcal{R}}Y(\varrho,u\cdot 1^{\lfloor n\eta\rfloor}).
\end{align*}
Consequently, using \eqref{controlonDj} and \eqref{control3'}, we get 
\begin{align}
\nonumber&\mathbb{E}(\sup_{ \varrho\in \mathcal{R}}\mu_\varrho(\{t\in\partial \TT: \, k_n(t)\ge \lfloor n\eta\rfloor\})\\
\nonumber& \le (\#\{\varrho_{|n+\lfloor n\eta\rfloor}: \varrho\in \mathcal{R}\})\Big (1-\frac{c_0}{j_{n+\lfloor n\eta\rfloor}}\Big )^{\lfloor n\eta\rfloor} C_{X,\phi}e^{(n+\lfloor n\eta\rfloor)\varepsilon_{n+\lfloor n\eta\rfloor}}\\
\label{aaa}&\le C_{X,\phi} (j_{n+\lfloor n\eta\rfloor})! \exp \left (-c_0\frac{\lfloor n\eta\rfloor}{j_{n+\lfloor n\eta\rfloor}}+(n+\lfloor n\eta\rfloor)\varepsilon_{n+\lfloor n\eta\rfloor}\right ).
\end{align}
Note that due to \eqref{controlgammaj} and the fact that  $\varepsilon_k=\gamma_{j_{k}+1}^2\widehat m_{j_{k}+1}$, for $n$ large enough we have $-c_0\frac{\lfloor n\eta\rfloor}{j_{n+\lfloor n\eta\rfloor}}+(n+\lfloor n\eta\rfloor)\varepsilon_{n+\lfloor n\eta\rfloor}\le -c_0\frac{n\eta}{2j_{n}}$. Moreover, $n\ge N_{j_n}\ge (j_n)!$ so $n\ge \frac{n+\lfloor n\eta\rfloor}{1+\eta}\ge \frac{(j_{n+\lfloor n\eta\rfloor})!}{1+\eta}$, and $j_n=o(\log(n))$ as $n\to\infty$. Consequently, \eqref{aaa} implies that 
$$
\sum_{n\in\mathbb{N}}\mathbb{E}(\sup_{ \varrho\in \mathcal{R}}\mu_\varrho(\{t\in\partial \TT: \, k_n(t)\ge \lfloor n\eta\rfloor\})<\infty,
$$
This holds for all $\eta\in (0,1)$ from which it follows that, with probability 1, for all positive rational number $\eta>0$, one has $\sum_{n\in\mathbb{N}}\mathbb{E}(\sup_{ \varrho\in \mathcal{R}}\mu_\varrho(\{t\in\partial \TT: \, k_n(t)\ge \lfloor n\eta\rfloor\})<\infty$. By the Borel-Cantelli lemma, this implies that with probability 1, for all $ \varrho\in \mathcal{R}$ one has $k_n(t)=o(n)$ for $\mu_\varrho$-almost every $t$, which is what had to be established.\end{proof}

\section{Proof of Theorem~\ref{thm-1.1}}\label{proofthm1.1}

Sections~\ref{up2} and \ref{LB} are respectively dedicated to establish the sharp upper bound and lower bound for $\dim E(X,K)$, almost surely for all $K\in\mathcal K$.

\subsection{Upper bounds for the Hausdorff dimensions of the sets $E(X,K)$}
\label{up2}
For each $(q,\alpha)\in \R^d\times\R^d$, recall the definition \eqref{Palphaq} of $\widetilde P_{X,\phi,\alpha}(q)$ and define  
$$
P_{X,\phi,\alpha}(q)=\inf\Big \{t\in\R: \limsup_{n\rightarrow\infty}\frac{1}{n}\log\Big (\displaystyle\sum_{u\in \TT_n}  \exp ({\langle q|S_{n}(X-\alpha)(u)\rangle}-tS_n\phi(u))\Big)\le 0\Big \}.
$$

The following proposition is a direct consequence of the log-convexity, of the mappings $(q,t)\mapsto \displaystyle\sum_{u\in \TT_n}  \exp ({\langle q|S_{n}(X-\alpha)(u)\rangle}-tS_n\phi(u))$  and $(\alpha,t)\mapsto \displaystyle\sum_{u\in \TT_n}  \exp ({\langle q|S_{n}(X-\alpha)(u)\rangle}-tS_n\phi(u))$ given $\alpha\in\R^d$ and  $q\in\R^d$ respectively. 
\begin{proposition}\label{convex}
The mappings $q\mapsto P_{X,\phi,\alpha}(q)$ and $\alpha\mapsto P_{X,\phi,\alpha}(q)$ are convex. 
\end{proposition}
\begin{proposition}\label{compPXphialpha}
With probability 1, $P_{X,\phi,\alpha}(q)\le \widetilde P_{X,\phi,\alpha}(q)$ for all $(q,\alpha)\in\R^d\times\R^d$. 
\end{proposition}
\begin{proof}
Due to Proposition~\ref{convex},  we only need to prove the inequality for each $(q,\alpha)\in\R^d\times\R^d$, almost surely. Fix  $(q,\alpha)\in\R^d\times\R^d$.  For $t>\widetilde{P}_\alpha(q)$ we have 
\begin{eqnarray*}
\E(\displaystyle\sum_{n\geq1}\displaystyle\sum_{u\in \TT_n}  \exp ({\langle q|S_{n}(X-\alpha)(u)\rangle}-tS_n\phi(u))&=& \displaystyle\sum_{n\geq1}\E (\sum_{i=1}^{N}   \exp ({\langle q|X_i-\alpha\rangle}-t\phi_i))^n<\infty.
\end{eqnarray*}
Consequently, $\sum_{u\in \TT_n}  \exp \big ({\langle q|S_{n}(X-\alpha)(u)\rangle}-tS_n\phi(u)\big )$ is bounded almost surely, so $t\ge P_{X,\phi,\alpha}(q)$ almost surely. Since $t>\wt P_{X,\phi,\alpha}(q)$ is arbitrary, we get the desired conclusion. 
\end{proof}

For $\alpha\in\R^d$, set
$$
\widehat E(X,\alpha)=\Big\{t\in\partial\TT: \alpha\in \bigcap_{n\in\mathbb{N}}\overline{ \Big \{\frac{S_nX(t)}{n}:n\ge N\Big\}}\Big\} .
$$ 

The following proposition and its corollary extend \cite[Proposition 2.5 and Corollary 2.3]{AB}, valid for ${\mathrm d}_1$, to the case of the more general metric ${\mathrm d}_\phi$.

\begin{proposition}\label{rn}
With probability 1, for all  $\alpha\in \R^d$, $\dim \widehat E(X,\alpha)\le P_{X,\phi,\alpha}^*(0)$, a negative dimension meaning that $\widehat E(X,\alpha)$ is empty. In particular $P_{X,\phi,\alpha}^*(0)\ge 0$ for all $\alpha\in I_X$.
\end{proposition}

\begin{proof}
The argument is largely inspired by the approach developed for the multifractal analysis of $\R^d$-values Birkhoff averages on conformal repellers (see \cite{BSS}; for a general multifractal formalism for vector-valued functions, see~\cite{Pey4}). Recall that for any  $E\subset\partial\mathcal T$, $\dim E=-\infty$ if $E=\emptyset$ and $\dim E= \inf\{s\in\mathbb R_+: \lim_{\delta\to 0^+} \mathcal H^s_\delta(E)=0\}$ otherwise, where 
$$
\mathcal H^s_\delta(E)=\inf\Big\{\sum_{i\in\mathbb N} \mathrm{diam} (E_i)^s:\ E\subset \bigcup_{i\in\mathbb N} E_i, \ \mathrm{diam} (E_i)\le \delta\Big\}.
$$
For every $n\ge 1$ let us denote $\widetilde r_n=\max\{\mathrm{diam}([u]): u\in \TT_n\}$. To begin, note that 
\begin{eqnarray*}
\widehat E(X,\alpha)&=&\bigcap_{\varepsilon>0}\bigcap_{n\in\mathbb{N}}\bigcup_{n\ge N}\{t\in\partial\TT: \|S_nX(t)-n\alpha\|\le n\varepsilon\}\\
&\subset& \bigcap_{q\in\R^d}\bigcap_{\varepsilon>0}\bigcap_{n\in\mathbb{N}}\bigcup_{n\ge N}\{t\in\partial\TT: |\langle q|S_nX(t)-n\alpha\rangle|\le n\|q\|\varepsilon\}.
\end{eqnarray*}
Fix $q\in\mathbb{R}^d$ and $\varepsilon>0$. For $N\ge 1$, the set $E(q,N,\varepsilon,\alpha)=\bigcup_{n\ge N}\{t\in\partial\TT: |\langle q|S_nX(t)-n\alpha\rangle|\le n\|q\|\varepsilon\}$ is covered by the union of those $[u]$ such that $u\in \TT_n$ and $\langle q|S_nX(u)-n\alpha\rangle +n\|q\|\varepsilon\ge 0$. Consequently, noting that by construction of the metric $\mathrm{d}_\phi$, for all $u\in \TT_n$ we have $\mathrm{diam}([u])\le \exp(-S_n\phi(u))$, for any $s\ge 0$ we can write
\begin{align*}
\mathcal H^s_{\tilde r_N}(E(q,N,\varepsilon,\alpha))&\le \sum_{n\ge N}\sum_{u\in \TT_n} \mathrm{diam}([u])^s \exp(\langle q|S_nX(u)-n\alpha\rangle +n\|q\|\varepsilon)\\
&\le  \sum_{n\ge N}\sum_{u\in \TT_n}\exp(\langle q|S_n(X-\alpha)(u)\rangle -sS_n\phi(u) +n\|q\|\varepsilon).
\end{align*}
Hence, if $\eta>0$ and $s>P_{X,\phi,\alpha}(q)+\eta+\|q\|\varepsilon$, by definition of $P_{X,\phi,\alpha}(q)$, for $N$ large enough  one has $
\mathcal H^s_{\tilde r_N}(E(q,N,\varepsilon,\alpha))\le \sum_{n\ge N}e^{-n\eta/2}$.
Since $\widetilde r_N\le r_n=\max\{\exp(-S_n\phi(u)): u\in \TT_n\}$,  $\widetilde r_N$ tends to 0 almost surely as $N$ tends to $\infty$, and we conclude that $\dim E(q,N,\varepsilon,\alpha)\le s$. As this holds for all $\eta>0$, we get $\dim E(q,N,\varepsilon,\alpha)\le P_{X,\phi,\alpha}(q)+\|q\|\varepsilon$. It follows that 
$
\dim \widehat E(X,\alpha)\le \inf_{q\in\mathbb{R^d}}\inf_{\varepsilon>0}P_{X,\phi,\alpha}(q)+\|q\|\varepsilon=\inf_{q\in\mathbb{R^d}}P_{X,\phi,\alpha}(q)=P^*_{X,\phi,\alpha}(0).
$ 
If $\inf_{q\in\mathbb{R^d}}P_{X,\phi,\alpha}(q)<0$, we necessarily have $ \widehat E(X,\alpha)=\emptyset$. Since for $\alpha\in I_X$ one has $\emptyset\neq E(X,\alpha)\subset \widehat E(X,\alpha)$, we get $P^*_{X,\phi,\alpha}(0)\ge 0$. 
\end{proof}
\begin{corollary}\label{upbd}
With probability 1, for all  compact connected subset $K$ of $\R^d$, one has $E(X,K)= \emptyset$  if $K\not \subset I_X$, and $
\dim E(X,K)\le \inf_{\alpha\in K}\widetilde P_{X,\phi,\alpha}^*(0)$ otherwise. 
\end{corollary}
\begin{proof}
This follows directly from Propositions~\ref{compPXphialpha} and~\ref{rn}.
\end{proof}

\subsection{Lower bounds for the Hausdorff dimensions of the set $E(X,K)$}\label{LB}
The sharp lower bound estimates  for the Hausdorff dimensions of the set $E(X,K)$ are direct consequences of Theorem~\ref{lb2}, the fact that $\liminf_{n\to\infty}\frac{ \sum_{k=1}^nh(q_k,\alpha_k)}{\sum_{k=1}^n\lambda(q_k,\alpha_k)}\ge  \liminf_{n\to\infty}\frac{h(q_n,\alpha_n)}{\lambda(q_n,\alpha_n)}$ for all $\varrho\in \mathcal R$, \eqref{entropiesurLyapounov}, and the following two propositions. Recall the definition \eqref{alphaXphi} of
$\beta(q,\alpha)$.

\begin {proposition} \label{pp3'}
With probability  $1$, for all $\varrho =((q_k,\alpha_k))_{k\geq 1} \in  \mathcal{R}$,  for $\mu_\varrho$-almost all  $t \in\partial\TT$,  one has
 $$\lim_{n\to\infty}\displaystyle n^{-1}\Big \|\displaystyle S_nX(t) - \displaystyle\sum_{k=1}^{n} \beta(q_k,\alpha_k) \Big \| =0. 
 $$
\end{proposition}

\begin{proposition}\label{conclusion2}
For all compact connected subset $K$ of $I_X$, there exists $ \varrho\in \mathcal{R}$ such that 
$$
\begin{cases}
 \displaystyle \bigcap_{n\in\mathbb{N}}\overline{ \Big \{n^{-1}\sum_{k=1}^n\beta(q_k,\alpha_k) :n\ge N\Big\}}=K\\
 \liminf_{n\to\infty}\wt P_{X,\phi,\alpha_n}^*(\nabla\wt P_{X,\phi,\alpha_n} (q_n))\ge \inf\{P_{X,\phi,\alpha}^*(0):\alpha\in K\}
\displaystyle
\end{cases}.
$$
\end{proposition}

\begin{proof}[Proofs of Proposition~\ref{pp3'} and Proposition~\ref{pp33}(1)] We will prove slightly stronger results by controlling uniformly the speed of convergence to 0. 

Fix $(\delta_n)_{n\in\mathbb{N}}$ converging to 0, and to be specified later. Let  $v$  be a vector of the canonical basis $\mathcal B$ of  $\R^d$. For  $ \varrho\in \mathcal{R}$, $v\in\mathcal B$, $\lambda\in\{-1,1\}$ and  $n \geq 1$, we set ~:
\begin{equation*}
\left\lbrace
\begin{aligned}E^\lambda_{\varrho,n,\delta_n}(v)&= \Big \{t\in \partial\TT : \lambda \Big  \langle v \Big | S_nX(t) - \sum_{k=1}^{n} \beta(q_k,\alpha_k)\Big \rangle \ge   n\delta_n \Big \}\\
F_{\varrho,n,\delta_n}^{\lambda}&= \Big \{t\in \partial\TT : \lambda \Big ( S_n\phi(t) - \sum_{k=1}^{n} \lambda(q_k,\alpha_k)\Big ) \ge   n\delta_n \Big \}\\
G_{\varrho,n,\delta_n}^{\lambda}&= \Big \{t\in \partial\TT : \lambda\Big ( \log \Big (\prod_{k=1}^nW_{\varrho,t_1\cdots t_k}\Big ) -\sum_{k=1}^{n} h(q_k,\alpha_k)\Big )\ge   n\delta_n \Big \}
\end{aligned}
\right. .
\end{equation*}
It is enough to specify $(\delta_n)_{n\in\mathbb{N}}$ such that  for  $\lambda\in\{-1,1\}$ and $v\in\mathcal B$ one has\begin{equation} \label{eq2'}
\E\Big (\displaystyle \displaystyle  \sup_{\varrho \in \mathcal R} \sum_{n\geq 1}\mu_\varrho (E_{\varrho,n,\delta_n}^{\lambda}(v))+\mu_\varrho (F_{\varrho,n,\delta_n}^{\lambda})+\mu_\varrho (G_{\varrho,n,\delta_n}^{\lambda})\Big )< \infty. 
\end{equation}
Indeed, if \eqref{eq2'} holds then,  with probability  $1$, for all  $\varrho \in  \mathcal R$,  $\lambda\in\{-1,1\}$ and $v\in \mathcal B$, one has $\displaystyle\sum_{n\geq 1} \mu_\varrho (E_{\varrho,n,\delta_n}^{\lambda}(v))+\mu_\varrho (F_{\varrho,n,\delta_n}^{\lambda})+\mu_\varrho (G_{\varrho,n,\delta_n}^{\lambda})<\infty$. Consequently, by the Borel-Cantelli lemma, for $\mu_\varrho$-almost every $t$,  for $n$ large enough, for all $v\in\mathcal B$,
\begin{equation*}
\left\lbrace
\begin{aligned}
&\Big |\Big  \langle v \Big | \Big (S_nX(t)  - \sum_{k=1}^{n} \beta(q_k,\alpha_k)\Big )\Big \rangle \Big | \le n\delta_n\\
&\max \Big (\Big |S_n\phi (t) - \sum_{k=1}^{n} \lambda(q_k,\alpha_k)\Big|,
\Big | \log \Big (\prod_{k=1}^nW_{\varrho,t_1\cdots t_k}\Big )  - \sum_{k=1}^{n} h(q_k,\alpha_k) \Big|\Big )\le n\delta_n,
\end{aligned}
\right.\end{equation*}
which yields the desired result. 

Now we prove \eqref{eq2'} when $\lambda=1$ (the case $\lambda=-1$ is similar). Let   $\varrho \in  \mathcal R$. For $n\ge1$ and $u\in\mathbb N^{\mathbb N}$, and $\gamma>0$, set
$$
\Pi^e_{n,\gamma}(\varrho,u)= \prod_{k=1}^{n} \exp \big ( \langle q_k+\gamma v |X_{u_{|k}}-\alpha_k\rangle- \widetilde P_{\alpha_k}(q_k)\phi_{u_{|k}}-\langle \gamma v|\beta(q_k,\alpha_k)-\alpha_k\rangle - \gamma   \delta_n\big ),
$$
where $\widetilde P_{\alpha_k}$ stands for $\widetilde P_{X,\phi,\alpha_k}$. For  every $\gamma>0$, using Chernov's inequality one can get 
\begin{eqnarray*}
\mu_\varrho (E_{\varrho,n,\delta_n}^{1}(v))\le e_{n,\gamma}(\varrho),
\end{eqnarray*} 
where 
\begin{align*}
 e_{n,\gamma}(\varrho)= \displaystyle\sum_{u\in \TT_n} \mu_\varrho ([u]) \prod_{k=1}^{n} \exp \big (\gamma \langle v |X_{u_{|k}} -\beta(q_k,\alpha_k)\rangle - \gamma   \delta_n\big )
=\displaystyle\sum_{u\in \TT_n }  \Pi^e_{n,\gamma}(\varrho,u) Y(\varrho,u).
 \end{align*}
Note that $\Pi^e_{n,\gamma}(\cdot,u)$ only depends on $\varrho_{|n}$, so
\begin{eqnarray*}
\sup_{ \varrho\in \mathcal{R}}  e_{n,\gamma}(\varrho)\le \displaystyle\sum_{u\in \TT_n }\sup_{\varrho_{|n}:  \varrho\in \mathcal{R}}\Pi^e_{n,\gamma}(\varrho,u)\cdot  \sup_{ \varrho\in \mathcal{R}}Y(\varrho,u) .
\end{eqnarray*}
Consequently, since $\mathbb{E} (\sup_{ \varrho\in \mathcal{R}}Y(\varrho,u))\le C_{X,\phi}\exp (\varepsilon_{|u|} |u|) $ by \eqref{control3'}, one obtains (taking into account the independences)
\begin{align*}
&\mathbb E (\sup_{ \varrho\in \mathcal{R}} e_{n,\gamma} (\varrho))\\
&\le C_{X,\phi}\, e^{n\varepsilon_n} \mathbb{E}\Big ( \sum_{u\in \TT_n} \sup_{\varrho_{|n}:  \varrho\in \mathcal{R}}\Pi^e_{n,\gamma}(\varrho,u)\Big) \\
&\le  C_{X,\phi}\, e^{n\varepsilon_n} \mathbb{E}\Big ( \sum_{u\in \TT_n}  \sum_{\varrho_{|n}:  \varrho\in \mathcal{R}}\Pi^e_{n,\gamma}(\varrho,u)\Big)\\
&= C_{X,\phi}\, e^{n\varepsilon_n} \sum_{\varrho_{|n}:  \varrho\in \mathcal{R}}\exp\Big (\sum_{k=1}^{n} L_{\alpha_k}(q_k+\gamma v, \widetilde P_{\alpha_k}(q_k))- \langle \gamma v |\beta(q_k,\alpha_k) -\alpha_k \rangle- \gamma   \delta_n\Big ).
\end{align*} 
Similarly, setting 
$$
\Pi^f_{n,\gamma}(\varrho,u)= \prod_{k=1}^{n} \exp \big ( \langle q_k|X_{u_{|k}}-\alpha_k\rangle- (\widetilde P_{\alpha_k}(q_k)+\gamma )\phi_{u_{|k}}-\gamma \lambda(q_k,\alpha_k) - \gamma   \delta_n\big ),
$$ a new application of Chernov's inequality yields 
\begin{align*}
\mu_\varrho (F_{\varrho,n,\delta_n}^{1}(v))\le f_{n,\gamma}(\varrho),
\end{align*} 
where 
\begin{align*}
 f_{n,\gamma}(\varrho)= \displaystyle\sum_{u\in \TT_n} \mu_\varrho ([u]) \prod_{k=1}^{n} \exp \big (\gamma (\phi_{u_{|k}} -\lambda(q_k,\alpha_k)) - \gamma   \delta_n\big )=\displaystyle\sum_{u\in \TT_n }  \Pi^f_{n,\gamma}(\varrho,u) Y(\varrho,u).
 \end{align*}
It follows that  
\begin{align*}
&\E (\sup_{ \varrho\in \mathcal{R}} f_{n,\gamma} (\varrho))\le C_{X,\phi}\, e^{n\varepsilon_n} 
\sum_{\varrho_{|n}:  \varrho\in \mathcal{R}}\exp\Big (\sum_{k=1}^{n} L_{\alpha_k}(q_k, \widetilde P_{\alpha_k}(q_k)+\gamma)- \gamma \lambda(q_k,\alpha_k) - \gamma   \delta_n\Big ).
\end{align*}
Also, setting $p=1+\gamma$ and 
$$
\Pi^g_{n,\gamma}(\varrho,u)= \prod_{k=1}^{n} \exp \big (p\langle q_k|X_{u_{|k}}-\alpha_k\rangle- p\widetilde P_{\alpha_k}(q_k)\phi_{u_{|k}}-\gamma h(q_k,\alpha_k) - \gamma   \delta_n\big ),
$$ it holds that 
\begin{align*}
\mu_\varrho (G_{\varrho,n,\delta_n}^{1}(v))\le g_{n,\gamma}(\varrho),
\end{align*} 
where 
\begin{align*}
 g_{n,\gamma}(\varrho)&= \displaystyle\sum_{u\in \TT_n} \mu_\varrho ([u]) \prod_{k=1}^{n} \exp \big (\gamma (\log (W_{\varrho,u_1\cdots u_k}) -h(q_k,\alpha_k)) - \gamma   \delta_n\big )\\
&=\displaystyle\sum_{u\in \TT_n }  \Pi^g_{n,\gamma}(\varrho,u) Y(\varrho,u).
 \end{align*}
This time,  one has the upper bound  
\begin{align*}
\E (\sup_{ \varrho\in \mathcal{R}} g_{n,\gamma} (\varrho)) \le   C_{X,\phi}\, e^{n\varepsilon_n} \sum_{\varrho_{|n}:  \varrho\in \mathcal{R}}\exp\Big (\sum_{k=1}^{n} L_{\alpha_k}\big (p q_k, p\widetilde P_{\alpha_k}(q_k)\big )- \gamma h(q_k,\alpha_k) - \gamma   \delta_n\Big ).
 \end{align*}

For each $ \varrho\in \mathcal{R}$, one has $q_k\in D_{j_n+1}$ for all $1\le k\le n$. Thus, reasoning as in the proof of Proposition~\ref{pro-3.5}(1) and  writing for each  $1\le k\le n$ the Taylor expansion with integral rest of order 2 of $\gamma\mapsto L_{\alpha_k}(q_k+\gamma v, \widetilde P_{\alpha_k}(q_k))- \langle \gamma v |\beta(q_k,\alpha_k) -\alpha_k \rangle$ at $0$, taking  $\gamma=\gamma_{j_n+1}$, and using \eqref{Der1} and \eqref{mj2} one gets
\begin{align*}
&\sum_{k=1}^n L_{\alpha_k}(q_k+\gamma_{j_n+1} v, \widetilde P_{\alpha_k}(q_k))- \langle \gamma_{j_n+1} v |\beta(q_k,\alpha_k)-\alpha_k\rangle -\gamma_{j_n+1}\delta _n\\
& \quad \le n\gamma_{j_n+1}^2m_{j_n+1}-n \gamma_{j_n+1}\delta _n
\end{align*}
uniformly in $ \varrho\in \mathcal{R}$. Similarly, using \eqref{Der1} and \eqref{mj2} again one gets 
$$
\sum_{k=1}^{n} L_{\alpha_k}(q_k, \widetilde P_{\alpha_k}(q_k)+\gamma)- \gamma \lambda(q_k,\alpha_k) - \gamma   \delta_n\le n\gamma_{j_n+1}^2m_{j_n+1}-n \gamma_{j_n+1}\delta _n,
$$
while using \eqref{Der3} and \eqref{tmj2} (as in the proof of Proposition~\ref{pro-3.5}) yields
$$
\sum_{k=1}^{n} L_{\alpha_k}\big (p q_k, p\widetilde P_{X,\phi,\alpha_k}(q_k)\big )- \gamma h(q_k,\alpha_k) - \gamma   \delta_n\le n\gamma_{j_n+1}^2\widetilde m_{j_n+1}-n \gamma_{j_n+1}\delta _n.
$$
Consequently, since $\varepsilon_n=2\gamma_{j_n+1}^2\widehat m_{j_n+1}$, $\max(m_{j_n+1},\widetilde m_{j_n+1})\le  \widehat  m_{j_n+1}$, and $\mathrm{card}(\{\varrho_{|n}: \varrho\in\mathcal{R}\})\le (j_n+1)!$, one obtains the following upper bound:
\begin{align*}
&\E \Big (\sup_{ \varrho\in \mathcal{R}} e_{n,\gamma_{j_{n}+1}} (\varrho)+\sup_{ \varrho\in \mathcal{R}} f_{n,\gamma_{j_{n}+1}} (\varrho)+\sup_{ \varrho\in \mathcal{R}} g_{n,\gamma_{j_{n}+1}} (\varrho)\Big )\\
&\quad \le 3 \, C_{X,\phi}\,  (j_n+1)!\exp \big ((-n\gamma_{j_n+1}(\delta_n-3\gamma_{j_n+1}^2\widehat m_{j_n+1})\big ).
\end{align*}
Let $\delta_n= 4\gamma_{j_n+1}^2\widehat m_{j_n+1}$. Note that $\lim_{n\to\infty} \delta_n=0$. Now we  use \eqref{control0'}:  $(j_n+1)!\le \exp (N_{j_n}^{1/3})\le \exp (n^{1/3})$ and $\gamma_{j_{n+1}}^2\widehat m_{j_n+1}\ge N_{j_n}^{-1/2}\ge n^{-1/2}$. Thus
$$
\E \Big (\sup_{ \varrho\in \mathcal{R}} e_{n,\gamma_{j_{n}+1}} (\varrho)+\sup_{ \varrho\in \mathcal{R}} f_{n,\gamma_{j_{n}+1}} (\varrho)+\sup_{ \varrho\in \mathcal{R}} g_{n,\gamma_{j_{n}+1}} (\varrho)\Big )\\
 \le 3 \, C_{X,\phi}\, \exp (n^{1/3})\exp(-n^{1/2}).
$$
This yields  \eqref{eq2'}.
\end{proof}

\begin{proof}[Proof of Proposition~\ref{conclusion2}]
For every integer $m\ge 1$, let $B(\widetilde \alpha_{m,\ell},1/m)_{1\le \ell\le L_m}$ be a finite covering of $K$ by balls centered on $K$, with $L_m\ge 2$. Since $K$ is connected, without loss of generality we can assume that $B(\widetilde\alpha_{m,\ell},1/m)\cap B(\widetilde\alpha_{m,\ell+1},1/m)\neq \emptyset$ for all $1\le \ell\le L_m-1$, and $B(\widetilde\alpha_{m+1,1},1/(m+1))\cap B(\widetilde\alpha_{m,L_m},1/m)\neq\emptyset$. 


Now, applying Lemma~\ref{approxi2}, for each $\widetilde\alpha_{m,\ell}$ let $(q_{m,\ell}, \alpha_{m,\ell}) \in D$ such that  $\| \beta(q_{m,\ell}, \alpha_{m,\ell}) - \widetilde\alpha_{m,\ell} \|\le 1/m$ and $|\wt P_{X,\phi, \alpha_{m,\ell}}^*(\nabla\wt P_{X,\phi, \alpha_{m,\ell}}(q_{m,\ell}))-\wt P_{X,\phi, \widetilde\alpha_{m,\ell}}^*(0)|\le 1/m$. 

Let $j_{1,1}=\min\{j\ge 1:( q_{1,1}, \alpha_{1,1}) \in D_j\}$. Then, define recursively for each $m\ge 1$ and $1\le \ell\le L_m-1$,  $j_{m,\ell+1}=\min\{j> j_{m,\ell}: (q_{m,\ell+1}, \alpha_{m,\ell+1}) \in D_j\}$, and   $j_{m+1,1}=\min\{j> j_{m,L_m}: 
(q_{m+1,1},\alpha_{m+1,1}) \in D_j\}$. 

The sequence $\varrho$ is constructed as follows. For  $1\le k\le M_{j_{1,1}-1}$, let $(q_{k}, \alpha_k) $ be equal to the unique element of~$D_1$. Then, for all $m\ge 1$, let $(q_{k}, \alpha_k) =(q_{m,\ell}, \alpha_{m,\ell})$ for $k\in  [M_{j_{m,\ell}-1}+1,M_{j_{m,\ell+1}-1}]$ if $1\le \ell\le L_m-1$ and $(q_{k}, \alpha_k) =(q_{m,L_m}, \alpha_{m,L_m} )$ for $k\in   [M_{j_{m,L_m}-1}+1,M_{j_{m+1,1}-1}]$. 

\medskip

Now let $n\ge M_{j_{2,1}}+1$. There is an integer $m_n\ge 2$ such that either $n\in [M_{j_{m_n,\ell_n}-1}+1,M_{j_{m_n,\ell_n+1}-1}]$ for some $1\le \ell_n\le L_{m_n}-1$ or $n\in  [M_{j_{m_n,L_{m_n}-1}}+1,M_{j_{m_n+1,1}-1}]$. 

In the first case,  let us write $\sum_{k=1}^n \beta(q_k, \alpha_k)=S_1+S_2+S_3$, where 
$$
S_1=\sum_{k=1}^{M_{j_{m_n,\ell_n}-2}} \beta(q_k, \alpha_k) ,\ S_2=\sum_{k=M_{j_{m_n,\ell_n}-2}+1}^{M_{j_{m_n,\ell_n}-1}}
 \beta(q_k, \alpha_k),\ S_3= \sum_{k=M_{j_{m_n,\ell_n}-1}+1}^n \beta(q_k, \alpha_k).
$$
Setting $(q, \alpha)=
(q_{m_n,\ell_{n}-1}, \alpha_{m_n,\ell_{n}-1})$ if $\ell_n\ge 2$ and 
$(q, \alpha)=(q_{m_{n}-1,L_{m_{n}-1}}, \alpha_{m_{n}-1,L_{m_{n}-1}})$ otherwise, one has
$
S_2=
(M_{j_{m_n,\ell_n}-1}-M_{j_{m_n,\ell_n}-2}) \beta(q, \alpha).
$
Thus, by construction of $(q_{m,\ell},\alpha_{m,\ell})$, setting $\widetilde \alpha= \widetilde \alpha_{m_n,\ell_{n}-1}$ if $\ell_n\ge 2$ and $\widetilde \alpha=\widetilde \alpha_{m_{n}-1,L_{m_{n}-1}}$ otherwise, one has 
$$
\|S_2-  (M_{j_{m_n,\ell_n}-1}-M_{j_{m_n,\ell_n}-2}) \widetilde \alpha\|\le  (M_{j_{m_n,\ell_n}-1}-M_{j_{m_n,\ell_n}-2})/(m_{n}-1).
$$
Also, setting $q=q_{m_n,\ell_{n}}$, $\alpha'=\alpha_{m_n,\ell_{n}}$ and $\widetilde \alpha'=\widetilde \alpha_{m_n,\ell_{n}}$, one has 
$S_3=(n-M_{j_{m_n,\ell_n}-1}) \beta(q, \alpha)$, so  
$$
\|S_3- (n-M_{j_{m_n,\ell_n}-1})\widetilde\alpha'\|\le (n-M_{j_{m_n,\ell_n}-1})/m_n.
$$ 
Moreover,  due to \eqref{control2bis2}, one has $\|S_1\|\le (j_{m_n,\ell_n}-1)^{-1}N_{j_{m_n,\ell_n}-1} \| \beta(q, \alpha')\|\le  (j_{m_n,\ell_n}-1)^{-1}n \|  \beta(q, \alpha')\|$, so  
$$
\|S_1\|\le (j_{m_n,\ell_n}-1)^{-1}n(\|\widetilde\alpha'\|+1/m_n);
$$ 
also, due to \eqref{control2bis2} again,  $\|M_{j_{m_n,\ell_n}-2}\,  \widetilde \alpha\|\le    (j_{m_n,\ell_n}-1)^{-1}n\| \widetilde \alpha' \|$.  Moreover, the choice of the balls $B(\widetilde \alpha_{m,\ell},1/m)$ implies that $\|\widetilde \alpha-\widetilde \alpha'\|\le 1/(m_n-1)$. Consequently, putting the previous estimates together one gets
$$
\Big \|\sum_{k=1}^n\beta(q_k,\alpha_k) - n\alpha_{m_n,\ell_n}\Big \|\le n \Big (\frac{3}{m_n-1}+ \frac{2\|\alpha_{m_n,\ell_n}\|+ 1/m_n}{j_{m_n,\ell_n}-1}\Big ).
$$
The same estimate holds if $n\in  [M_{j_{m_n,L_{m_n}-1}}+1,M_{j_{m_n+1,1}-1}]$. Consequently, since as $n$ tends to $\infty$ the sequence $\alpha_{m_n,\ell_n}$ describes all the $\alpha_{m,\ell}$, the set of limit points of $n^{-1}\sum_{k=1}^n\beta (q_k, \alpha_k)$ is the same as that of the sequence $((\alpha_{m,\ell})_{1\le\ell\le L_m})_{m\ge 1}$, that is $K$. 

The fact that  $\liminf_{n\to\infty}n^{-1}\wt P_{X,\phi,\alpha_n}^*(\nabla\wt P_{X,\phi,\alpha_n} (q_n))\ge \inf\{\widetilde P_{X,\phi,\alpha}^*(0):\alpha\in K\}$ is a direct consequence of the choice of the vectors $q_{m,\ell}$ and $\alpha_{m,\ell}$, since $\wt P_{\alpha_{m,\ell}} ^*(\nabla\wt P_{\alpha_{m,\ell}} (q_{m,\ell}))\ge \inf\{\widetilde P_{X,\phi,\alpha}^*(0):\alpha\in K\}-1/m$. 
\end{proof}

\section{Proof of Theorem~\ref{UNIFER}}\label{proofLD}

We need to slightly modify the set $\mathcal{R}$ by requiring, in addition to the initial conditions on $(N_j)_{j\ge 0}$, that for all $j\ge 1$
\begin{equation}\label{newN_j}
N_{j+1}> M_j(k(M_j)+1)\quad\text{and}\quad ((j+3)!)^2\exp(-N_j/j)\le j^{-2}.
\end{equation}
Since the sequence $(k(n))_{n\in\mathbb{N}}$ is increasing and $M_{j_n}+1\le n\le M_{ j_{n}+1}$, one has  $n(k(n)+1)\le M_{j_{n}+1 }(k(M_{ j_{n}+1})+1)<N_{j_n+2}<M_{j_n+2}$, so $j_{n(k(n)+1)}\le j_n+1$. 

\medskip
For each integer $m\ge 1$, define the compact set
\begin{equation*}
K_m=\{(q,\alpha)\in J_{X,\phi}\cap B(0,m): d((q,\alpha),\partial J_{X,\phi})\ge 1/m\}\cup\{(q_1,\alpha_1)\},
\end{equation*}
where $B(0,m)$ is the Euclidean ball of radius $m$ centered at 0 in $\R^{2d}$ and $(q_1,\alpha_1)$ is the unique element of $D_1$. Then, recalling that $D$ was chosen so that the second claim of Lemma~\ref{approxi2} holds, for $\ell,m\ge 1$,  let  
\begin{equation}\label{Km}
\mathcal{R}(m)=\left \{\varrho=(q_k,\alpha_k)_{k\ge 1}\in \mathcal{R}\cap K_m^{\mathbb N}: 
\exists \ \alpha\in \widetilde I_X,\ \lim_{k\to\infty}(q_k,\alpha_k)=(q_\alpha,\alpha)\right\} 
\end{equation}
and 
$$
\widetilde{I}_X(m)=\left \{\alpha\in \widetilde I_X: (q_\alpha,\alpha)\in \{\lim_{k\to\infty}\varrho_k=(q_k,\alpha_k): \varrho\in \mathcal{R}(m)\}\right \}.
$$
Note that if $\varrho\in \mathcal R(m)$, then there is a unique $\alpha\in\widetilde I_X$ such that $\lim_{k\to\infty}(q_k,\alpha_k)=(q_\alpha,\alpha)$.  By construction, 
$
\widetilde I_X=\bigcup_{m\ge 1} \widetilde{I}_X(m).
$ 
Note also that in the statement of Theorem~\ref{UNIFER} the vector  $\psi(q_\alpha,\alpha)$ is simply denoted by $\psi_\alpha$, a notation that we adopt in this section as well.

\medskip

Let $\kappa=\liminf_{n\to\infty} \log(k(n))/n$ and $\kappa'=\limsup_{n\to\infty} \log(k(n))/n$.  For all integers $\ell,m\ge 1$ and closed dyadic cube $Q$ in $\R^d$, define the sets
\begin{equation*}
\left\lbrace
\begin{aligned}
 \mathcal{R}(m,\ell,Q)&=\left\{\varrho\in  \mathcal{R}(m):\, \forall \ \lambda\in \wt Q,\  
-\Lambda^*_{\psi_\alpha}(\nabla \Lambda_{\psi_\alpha}(\lambda)) < \min (\ell, \kappa-1/\ell)\right\},\\
  \mathcal{R}'(m,\ell,Q)&=\left\{\varrho\in  \mathcal{R}(m):\, \forall\ \lambda\in Q,\  
-\Lambda^*_{\psi_\alpha}(\nabla \Lambda_{\psi_\alpha}(\lambda)) > \kappa'+1/\ell \right\},
\end{aligned}
\right.\end{equation*}
where $\wt Q$ stands for the union of $Q$ and the closed dyadic cubes of the same generation as $Q$ and  neighboring $Q$. The usefulness of considering the cubes $\widetilde Q$ will appear in the proof of the following proposition. Recall the definitions \eqref{tildemu} and \eqref{tildeLambda} of $\mu^t_n$ and $\Lambda^t_n$ respectively.

\begin{proposition}\label{LD51}
With probability 1, for all integers $\ell, m\ge 1$ and all dyadic cubes $Q$, 
\begin{enumerate}
\item for all $\varrho\in  \mathcal{R}(m,\ell,Q)$, for $\mu_\varrho$-almost every~$t$, one has  $\lim_{n\to\infty}\frac{1}{n}\Lambda^t_{\tilde k,n}(\lambda)=\Lambda_{\psi_\alpha}(\lambda)$  for all $\lambda\in Q$. 

\medskip

\item For all $\varrho\in  \mathcal{R}'(m,\ell,Q)$, for $\mu_\varrho$-almost every~$t$, for all  $\lambda\in Q$, there exists $\varepsilon>0$ such that for $n$ large enough, one has $\mu^t_{\tilde k,n}(B(\nabla\Lambda_{\psi_\alpha}(\lambda),\varepsilon)) =0$.
\end{enumerate}
\end{proposition}

We assume this proposition for the time being and prove Theorem~\ref{UNIFER}.  

\begin{proof}[Proof of Theorem~\ref{UNIFER}]Let $\mathcal C$ stand for the set of closed dyadic cubes in $\R^d$. One has 
\begin{align*}
&\{(\alpha,\lambda)\in \widetilde I_X\times \R^d:\ -\Lambda^*_{\psi_\alpha}(\nabla \Lambda_{\psi_\alpha}(\lambda)) <\kappa\}\\
&=\bigcup_{m\ge 1} \bigcup_{\ell\ge 1} \{(\alpha,\lambda)\in \widetilde I_X\times \R^d:\alpha\in \widetilde I_X(m),\ -\Lambda^*_{\psi_\alpha}(\nabla \Lambda_{\psi_\alpha}(\lambda)) < \min (\ell, \kappa-1/\ell)\}\\
&=\bigcup_{m\ge 1} \bigcup_{\ell\ge 1}\bigcup_{Q\in\mathcal C} \{\alpha\in \widetilde I_X(m): \forall\, \lambda\in \wt Q,\  -\Lambda^*_{\psi_\alpha}(\nabla \Lambda_{\psi_\alpha}(\lambda)) < \min (\ell, \kappa-1/\ell)\} \times \widetilde Q,
\end{align*}
where one used the continuity in $(\alpha,\lambda)$ of $-\Lambda^*_{\psi_\alpha}(\nabla \Lambda_{\psi_\alpha}(\lambda))$. 
Consequently, due to Proposition~\eqref{LD51}(1), with probability 1, for all $\alpha\in \widetilde I_X$, if $m$ is large enough so that $\alpha\in \widetilde I_X(m)$, for all $\varrho\in  \mathcal{R}(m)$ such that $\lim_{k\to\infty}\varrho_k=(q_\alpha,\alpha)$, since each  $\lambda\in \R^d$ such that  $-\Lambda^*_{\psi_\alpha}(\nabla \Lambda_{\psi_\alpha}(\lambda)) <\kappa$ belongs, for $\ell $ large enough, to a dyadic cube $Q$ such that one has $-\Lambda^*_{\psi_\alpha}(\nabla \Lambda_{\psi_\alpha}) < \min (\ell, \kappa-1/\ell)$ over $Q$, one has $\mu_\varrho$-almost everywhere, $\lim_{n\to\infty}n^{-1} \Lambda^t_{\tilde k,n}(\lambda)=\Lambda_{\psi_\alpha}$ for all $\lambda\in \R^d$ such that  $-\Lambda^*_{\psi_\alpha}(\nabla \Lambda_{\psi_\alpha}(\lambda)) <\kappa$, i.e. the part (1) of the large deviations properties  $\mathrm{LD}(\Lambda_{\psi_\alpha},\widetilde k)$. One uses a similar argument to derive part (2) of $\mathrm{LD}(\Lambda_{\psi_\alpha},\widetilde k)$ from part (2) of Proposition~\eqref{LD51}. Part (3) of $\mathrm{LD}(\Lambda_{\psi_\alpha},\widetilde k)$ is established as \cite[Theorem 2.3(3)]{BL}).

To get the desired lower bound for $\dim E\big (X,\alpha, \mathrm{LD}( \Lambda_{\psi_\alpha},\widetilde k)\big )$, it is enough to pick $\varrho$ such that $\lim_{k\to\infty}\varrho_k=(q_\alpha,\alpha)$ and $\lim_{k\to\infty} \wt P_{X,\phi,\alpha_k}(q_k)-\langle q_k|\nabla \wt P_{X,\phi,\alpha_k}(q_k)\rangle=\wt P^*_{X,\phi,\alpha}(0)$, as in Lemma~\ref{approxi2} (then Proposition~\eqref{lb2} yields the result).
\end{proof}
\begin{proof}[Proof of Proposition~\ref{LD51}] Fix $m,\ell,Q$. We cut the sets $ \mathcal{R}(m,\ell,Q)$ and $  \mathcal{R}'(m,\ell,Q)$ into countably many pieces as follows: for every integer $L\ge 1$, setting $\Lambda_k=\Lambda_{\psi(q_k,\alpha_k)}$, let 
\[
 \mathcal{R}(m,\ell,L,Q)=\\ \left\{\varrho\in  \mathcal{R}(m,\ell,Q):
 \begin{cases} 
 \forall \ k\ge L,\ \forall\ \lambda\in \wt Q, \\ 
-\Lambda_k^*(\nabla \Lambda_k(\lambda)) < \min (2\ell, \kappa-1/2\ell)
\end{cases}
\right\},
\]
and
\[
  \mathcal{R}'(m,\ell,L,Q)=
\left\{\varrho\in   \mathcal{R}(m,\ell,Q): \begin{cases}\forall\ k\ge L,\ \forall\ \lambda\in Q, \\ 
|\Lambda_k^*(\nabla \Lambda_k(\lambda))-\Lambda_{\psi_\alpha}^*(\nabla \Lambda_{\psi_\alpha}(\lambda))| <1/2\ell 
\end{cases}
\right\}.
\]
The mappings $\Lambda_{\psi(q,\alpha)}$ and $\nabla \Lambda_{\psi(q,\alpha)}$ are continuous  as functions of  $(q,\alpha)$ taking values in $\mathcal C(\R^d\times\R^d,\R)$ and $\mathcal C(\R^d\times\R^d,\R^d)$ respectively (these spaces being endowed with the topology of the uniform convergence over compact sets). Thus, since $\wt Q$ is compact, we have $ \mathcal{R}(m,\ell,Q)=\bigcup_{L\ge 1} \mathcal{R}(m,\ell,L,Q)$ and $ \mathcal{R}'(m,\ell,Q)=\bigcup_{L\ge 1} \mathcal{R}'(m,\ell,L,Q)$.

\medskip
\noindent
Let us prove part (1) of the proposition. Fix $L\ge 1$. For $\varrho\in  \mathcal{R}(m,\ell,L,Q)$, $\lambda\in \wt Q$, $n\ge 1$, $1\le j\le k(n)$, $0\le i\le n-1$ and $t\in\partial \TT$, set 
$$
s^{(i)}_{n,j}(\varrho,\lambda)=\sum_{k=i+(j-1)n+1}^{i+jn}\Lambda_{\psi(q_k,\alpha_k)}(\lambda)
$$
and 
\begin{equation}\label{Znr}
Z^{(i)}_{n,j}(\varrho,\lambda,t)=\exp(\langle \lambda|(S_{i+jn}X(t)-S_{i+(j-1)n}X(t)\rangle-s^{(i)}_{n,j}(\varrho,\lambda)).
\end{equation}
It is enough to prove that for every $\lambda\in \wt Q$ and $\varepsilon>0$, one has  
\begin{equation}\label{control(1)}
\mathbb{E}\Big (\sum_{n\in\mathbb{N}} \sup_{\varrho\in  \mathcal{R}(m,\ell,L,Q)} \mu_\varrho(E(n,\varrho,\lambda))\Big )<\infty,
\end{equation}
where 
$$
E(n,\varrho,\lambda)=\Big\{t\in\partial \TT: \exists \,0\le i\le n-1,\ \Big |k(n)^{-1}\sum_{j=1}^{k(n)}(Z^{(i)}_{n,j}(\varrho,\lambda,t)-1)\Big |>\varepsilon \Big \}.
$$
Indeed, suppose that \eqref{control(1)} holds true. Then, for every $\lambda\in \wt Q$ and $\varepsilon\in (0,1)$, with probability 1, for all $\varrho\in  \mathcal{R}(m,\ell,L,Q)$, applying the Borel-Cantelli lemma to $\mu_\varrho$ yields, for $\mu_\varrho$-almost every $t$, an integer    $n_\varrho\ge 1$ such that for all $n\ge n_\varrho$, for all $0\le i\le n-1$,
$$
1-\varepsilon\le k(n)^{-1}\sum_{j=1}^{k(n)}Z^{(i)}_{n,j}(\varrho,\lambda,t)\le 1+\varepsilon.
$$
Moreover, given $\varrho\in   \mathcal{R}(m,\ell,L,Q)$, there exists $k_\varrho>1$ such that for all $k\ge k_\varrho$, one has $|\Lambda_{\psi(q_k,\alpha_k)}(\lambda)-\Lambda_{\psi_\alpha}(\lambda)|\le \varepsilon$, hence for $n\ge k_\varrho $ and $0\le i\le n-1$ one has  
$$
|s^{(i)}_{n,j}(\varrho,\lambda)-n\Lambda_{\psi_\alpha}(\lambda)|\le n\varepsilon+C_\varrho(\lambda),
$$ 
where $C_\varrho(\lambda)=\sum_{k=1}^{k_\varrho} |\Lambda_{\psi(q_k,\alpha_k)}(\lambda)-\Lambda_{\psi_\alpha}(\lambda)|$. Consequently,  setting 
$$
\mathcal{I}_n(t,\lambda)=\sum_{i=0}^{n-1}\sum_{j=1}^{k(n)}\exp(\langle \lambda|(S_{i+jn}X(t)-S_{i+(j-1)n}X(t)\rangle),
$$ 
for $n\ge \max( k_\varrho, n_\varrho)$, one obtains
$$
(1-\varepsilon)\exp (n\Lambda_{\psi_\alpha}(\lambda)-n\varepsilon-C_\varrho(\lambda))
\le \frac{\mathcal{I}_n(t,\lambda)}{nk(n)}\le (1+\varepsilon) \exp (n\Lambda_{\psi_\alpha}(\lambda)+ n\varepsilon+C_\varrho(\lambda)).
$$
By definition \eqref{tildeLambda} of $\Lambda_{\tilde k,n}^t(\lambda)$, one also has 
$\exp\big (\Lambda_{\tilde k,n}^t(\lambda)\big )= (nk(n))^{-1}\mathcal{I}_n(t,\lambda)$.
Letting $\varepsilon $ go to 0 along a discrete family, this gives that almost surely, for all $\varrho\in  \mathcal{R}(m,\ell,L,Q)$, for $\mu_\varrho$-almost every $t$, $\lim_{n\to\infty} n^{-1} \Lambda_{\tilde k,n}^t(\lambda)=\Lambda_{\psi_\alpha}(\lambda)$. Then, this convergence holds almost surely for a countable and dense subset of elements $\lambda$ of $\wt Q$, and finally the convexity of the functions $\Lambda_{\tilde k,n}^t$ gives the convergence for all $\lambda\in Q$, since $Q$ is included in the interior of~$\wt Q$. 

To prove \eqref{control(1)}, we need the following lemma, in which $\mathcal Q_\varrho$ stands for the Peyri\`ere measure associated with $\mu_\varrho$, that is the measure on $(\Omega\times \mathbb N^{\mathbb N},\mathcal A\otimes\mathcal B)$, defined by 
$$
\mathcal Q_\varrho(C)=\int_\Omega \int_{\mathbb N^{\mathbb N}} \mathbf{1}_C(\omega,t)\, {\rm d}\mu_{\varrho,\omega}(t) \, {\rm d}\mathbb P(\omega).
$$ 

\begin{lemma}\label{control4}
\begin{enumerate}
\item Let $\varrho \in   \mathcal{R}(m,\ell,L,Q)$, $\lambda\in \wt Q$ and $n\in\N$. For each $0\le i\le n-1$, the random variables $(\omega,t)\mapsto Z^{(i)}_{n,j}(\varrho,\lambda,t)-1$, $1\le j\le k(n)$, defined on $\Omega\times \mathbb N^{\mathbb N}$ are centered and independent with respect to~$\mathcal Q_\varrho$.
\item There exists $p(m,\ell,L,Q)\in (1,2]$ and $C(m,\ell,L,Q)>0$ such that for all $p\in (1,p(m,\ell,L,c)]$,  for all $\varepsilon>0$, for all $n\in\mathbb N$ and $0\le i\le n-1$, one has 
\begin{eqnarray*}
\mathcal Q_\varrho \Big (\Big |k(n)^{-1}\sum_{j=1}^{k(n)}\big (Z^{(i)}_{n,j}(\varrho,\lambda,t)-1\big )\Big |>\varepsilon \Big )&\le& C(m,\ell,L,Q) \exp(-n(p-1)\ell/4)
\end{eqnarray*}
independently of $\varrho\in  \mathcal{R}(m,\ell,L,Q)$ and $\lambda\in \wt Q$.
\end{enumerate}
\end{lemma} 
We postpone the proof of this lemma to the end of the section.

\medskip

Now, for $\varrho \in   \mathcal{R}(m,\ell,L,Q)$, $\lambda\in \wt Q$, $n\in\N$, $0\le i\le n-1$ and $t\in\partial \TT$ let 
$$
V^{(i)}_n(\varrho,\lambda,t)=\Big |k(n)^{-1}\sum_{j=1}^{k(n)}(Z^{(i)}_{n,j}(\varrho,\lambda,t)-1\big )\Big |,$$
and notice that  by construction $V^{(i)}_n(\varrho,\lambda,t)$ is constant over each cylinder $[u]$ of generation $nk(n)+i$, so that we also denote it by $V^{(i)}_n(\varrho,\lambda,u)$. One has
$$
\mu_\varrho (\{t\in\partial \TT:\,\exists \, 0\le i\le n-1,\,  V^{(i)}_n(\varrho,\lambda,t)>\varepsilon\})\le \sum_{i=0}^{n-1}\sum_{u\in\TT_{nk(n)+i}}\mathbf{1}_{\{V^{(i)}_n(\varrho,\lambda,u)>\varepsilon\}}\mu_\varrho([u]).
$$
Recall that by  definition $\mu_\varrho([u])=(\prod_{k=1}^n W_{\varrho,u_1\cdots u_k} )   Y(\varrho,u)$, with $\E(Y(\varrho,u))=1$, and due to Lemma~\ref{bettercontrol2}, since  $ \mathcal{R}(m,\ell,L,Q)\subset \mathcal{R}(K_m)$, one has $
\|\sup_{ \varrho\in \mathcal{R}(m,\ell,L,Q)}Y(\varrho, u)\|_1= O((j_{|u|}+2)!)$.  Fix $u^{n,i}\in \N^{nk(n)+i}$ and  denote $\|\sup_{ \varrho\in \mathcal{R}(m,\ell,L,Q)}Y(\varrho, u^{n,i})\|_1$ by $B_{n,i}$. Using the independence between $ (\prod_{k=1}^{nk(n)+i} W_{\varrho,u_1\cdots u_k} )_{\varrho\in  \mathcal{R}(m,\ell,L,Q)}$ and $Y(\cdot, u)$ for all $u\in \N^{nk(n)+i}$, one gets 

\begin{eqnarray*}
&&\mathbb{E}\Big (\sup\Big \{\mu_\varrho (\{t\in\partial \TT:\,\exists \, 0\le i\le n-1,\, V^{(i)}_n(\varrho,\lambda,t)>\varepsilon\}):\varrho\in  \mathcal{R}(m,\ell,L,Q)\Big\}\Big )\\
&\le& \sum_{i=0}^{n-1}\mathbb{E}\Big (\sum_{u\in  \TT_{nk(n)+i}} \sup_{ \varrho\in \mathcal{R}(m,\ell,L,Q)}\mathbf{1}_{\{V^{(i)}_n(\varrho,\lambda,u)>\varepsilon\}}\prod_{k=1}^{nk(n)+i} W_{\varrho,u_1\cdots u_k} \Big )\ B_{n,i}.
\end{eqnarray*}
From this inequality one can obtain
\begin{align*}
&\mathbb{E}\Big (\sup\Big \{\mu_\varrho (\{t\in\partial \TT:\,\exists \, 0\le i\le n-1,\, V_n(\varrho,\lambda,t)>\varepsilon\}):\varrho\in  \mathcal{R}(m,\ell,L,Q)\Big\}\Big )\\
&\le \sum_{i=0}^{n-1} \sum_{\substack{\varrho_{|nk(n)+i}:\\\varrho\in  \mathcal{R}(m,\ell,L,Q)}}\mathbb{E}\Big (\sum_{u\in  \TT_{nk(n)+i}}\mathbf{1}_{\{V^{(i)}_n(\varrho,\lambda,u)>\varepsilon\}}\prod_{k=1}^{nk(n)+i} W_{\varrho,u_1\cdots u_k} \Big )\ B_{n,i}.
\end{align*}
Then, noting that  
$$
\mathbb{E}\Big (\sum_{u\in  \TT_{nk(n)+i}}\mathbf{1}_{\{V^{(i)}_n(\varrho,\lambda,u)>\varepsilon\}}\Big (\prod_{k=1}^{nk(n)+i} W_{\varrho,u_1\cdots u_k}\Big )\Big )=\mathcal{Q}_\varrho(V^{(i)}_n(\varrho,\lambda,t)>\varepsilon)$$ 
yields
\begin{align*}
&\mathbb{E}\Big (\sup\Big \{\mu_\varrho (\{t\in\partial \TT:\,\exists \, 0\le i\le n-1,\,V^{(i)}_n(\varrho,\lambda,t)>\varepsilon\}):\varrho\in  \mathcal{R}(m,\ell,L,Q)\Big\}\Big )\\
&\le \sum_{i=0}^{n-1}\sum_{\varrho_{|nk(n)+i}:\varrho\in  \mathcal{R}(m,\ell,L,Q)}\mathcal{Q}_\varrho(V^{(i)}_n(\varrho,\lambda,t)>\varepsilon) \ B_{n,i}\\
&\le  \sum_{i=0}^{n-1} (\#\{ \varrho_{|nk(n)+i}:\varrho\in  \mathcal{R}\})\Big ( \sup_{ \varrho\in \mathcal{R}(m,\ell,L,Q)} \mathcal{Q}_\varrho(V^{(i)}_n(\varrho,\lambda,t)>\varepsilon) \Big ) \ B_{n,i}\\
&\le n  (j_{n(k(n)+1)}+1!) C(m,\ell,L,Q) \exp(-n(p-1)\ell/4) O((j_{n(k(n)+1)}+2)!),
\end{align*}
where  Lemma~\ref{control4}(2) was used. Now recall that due to  \eqref{newN_j} one has $j_{n(k(n)+1)}\le j_n+1$, hence 
\begin{align*}
&\mathbb{E}\Big (\sup\Big \{\mu_\varrho (\{t\in\partial \TT:\,\exists \, 0\le i\le n-1,\, V^{(i)}_n(\varrho,\lambda,t)>\varepsilon\}):\varrho\in  \mathcal{R}(m,\ell,L,Q)\Big\}\Big )\\&\quad= O(1) C(m,\ell,L,Q) n(j_n+3)!^2 \exp(-n(p-1)\ell/4).
\end{align*}
It follows that 
\begin{align*}
&\sum_{n\in\mathbb{N}}\mathbb{E}\Big (\sup\Big \{\mu_\varrho (\{t\in\partial \TT:\,\exists \, 0\le i\le n-1,\,V^{(i)}_n(\varrho,\lambda,t)>\varepsilon\}):\varrho\in  \mathcal{R}(m,\ell,L,Q)\Big\}\Big )\\
& = O\Big ( \sum_{j\ge 0} \sum_{M_j+1\le n\le  M_{j+1}}((j+3)!)^2 n\exp(-n(p-1)\ell/4)\Big )\\
&= O\Big ( \sum_{j\ge 0} \frac{((j+3)!)^2}{1-\exp(-(p-1)\ell/8)}\exp (-M_j(p-1)\ell/8)\Big) \\
&=O\Big ( \sum_{j\ge 0} ((j+3)!)^2\exp (-N_j(p-1)\ell/8)\Big).
\end{align*}
Due to \eqref{newN_j} the above series converges. 

\begin{remark}\label{impossibility}The reason why we did not succeed in proving that with probability 1, the Gibbs measure $\nu_\alpha$ is carried by $E(X,\alpha, {\rm LD}(\Lambda_{\psi_\alpha},\widetilde k))$ simultaneously for all $\alpha\in\widering I_X$, is the following. As said in the introduction, it is tempting to adapt to the present problem what was done in \cite{B2,Attia} to get that with probability 1, the  measure $\nu_\alpha$ is carried by $E(X,\alpha)$ simultaneously for all $\alpha\in\widering I_X$. To do so, notice first that for $\alpha\in\widering I_X$, the measure $\nu_\alpha$ is the homogeneous Mandelbrot measure $\mu_\varrho$, where $\varrho=\varrho^{(\alpha)}$ is the constant sequence $((q_\alpha,\alpha))_{n\in\mathbb N}$.  Then, what we would need is to have at our disposal a suitable almost sure upper bound $g_n$ for the mapping $\alpha\mapsto \mu_{\varrho^{(\alpha)}} (\{t\in\partial \TT:\,V^{(i)}_n(\varrho^{(\alpha)},\lambda,t)>\varepsilon\})$, with the following properties: $g_n$ possesses almost surely an holomorphic extension on a deterministic complex neiborhood $V_\alpha$ of any $\alpha\in \widering I_X$;  $\sum_{n\ge 0} \mathbb{E}(|g_n(z)|)$ converges uniformly on any compact subset of $\bigcup_{\alpha\in \widering I_X} V_\alpha$, so that an application of the Cauchy-formula yields the almost surely simultaneous convergence of  $\sum_{n\ge 0} g_n(\alpha)$ for all $\alpha\in \widering I_X $;  but we could not find any way to apply such a strategy. \end{remark}

Now we prove part (2) of the proposition. This situation is not empty only if $\kappa'=\limsup_{n\to\infty} \log(k(n))/n<\infty$. 

Fix $L>0$. Given $\varrho\in  \mathcal{R}'(m,\ell,L,Q)$, $\lambda\in Q$, $\varepsilon>0$, $n\ge 1$, $1\le j\le k(n)$ and $0\le i\le n-1$, set 
$$
\widetilde V^{(i)}_j(\varrho,\lambda,t)=\mathbf{1}_{B(0,n\varepsilon)} (S_{i+jn}X(t)-S_{i+(j-1)n}X(t)-n\nabla \Lambda_{\psi_\alpha}(\lambda)).
$$
Mimicking what was done above, one can  get  
\begin{align*}
&\E\Big (\sup_{\varrho\in  \mathcal{R}'(m,\ell,L,Q)}\mu_\varrho\Big (\Big \{t\in \TT: \,\exists \, 0\le i\le n-1,\, \sum_{j=1}^{k(n)}\widetilde V^{(i)}_j(\varrho,\lambda,t)\ge 1\Big \}\Big )\Big)\\
&\le \sum_{i=0}^{n-1}(\# \varrho_{|nk(n)+i}:\varrho\in  \mathcal R)\Big ( \sup_{ \varrho\in \mathcal{R}'(m,\ell,L,Q)} \mathcal{Q}_\varrho\Big (\Big\{\sum_{j=1}^{k(n)}\widetilde V^{(i)}_j(\varrho,\lambda,t)\ge 1\Big \}\Big ) \Big ) \ B_{n,i}\\
&=O((j_n+3)!^2)\sum_{i=0}^{n-1}\sup_{ \varrho\in \mathcal{R}'(m,\ell,L,Q)} \mathcal{Q}_\varrho\Big (\Big\{\sum_{j=1}^{k(n)}\widetilde V^{(i)}_j(\varrho,\lambda,t)\ge 1\Big \}\Big ). 
\end{align*}

For $\varrho\in \mathcal{R}'(m,\ell,L,Q)$, 
the fact that the inclusion $S_{i+jn}X(t)-S_{i+(j-1)n}(t)-n\nabla \Lambda_{\psi_\alpha}(\lambda)\in B(0,n\varepsilon)$ implies $\langle\lambda|S_{i+jn}X(t)-S_{i+(j-1)n}(t)-n\nabla \Lambda_{\psi_\alpha}(\lambda)\rangle\ge -n\varepsilon\|\lambda\|$, yields
\begin{align*}&\mathcal Q_\varrho\Big (\Big\{\sum_{j=1}^{k(n)}\widetilde V^{(i)}_j(\varrho,\lambda,t)\ge 1\Big \}\Big )\\
&\le \sum_{j=1}^{k(n)}\mathcal Q_\varrho\big ( \langle\lambda|S_{i+jn}X(t)-S_{i+(j-1)n}(t)-n\nabla \Lambda_{\psi_\alpha}(\lambda)\rangle\ge -n\varepsilon\|\lambda\|\big ).
\end{align*}
Then, applying Markov inequality gives
\begin{align*}
&\mathcal Q_\varrho\Big (\Big\{\sum_{j=1}^{k(n)}\widetilde V^{(i)}_j(\varrho,\lambda,t)\ge 1\Big \}\Big )\\
&\le   \sum_{j=1}^{k(n)} \exp \big (n\varepsilon\|\lambda\| -n \langle\lambda|\nabla \Lambda_{\psi_\alpha}(\lambda)\rangle\big )\mathbb E_{\mathcal Q_\varrho}\big (\exp(\langle \lambda|S_{i+jn}X(t)-S_{i+(j-1)n}(t)\rangle)\big )\\
&= \sum_{j=1}^{k(n)} \exp \big (n\varepsilon\|\lambda\| -n \langle\lambda|\nabla \Lambda_{\psi_\alpha}(\lambda)\rangle\big ) \exp \Big (\sum_{k=i+(j-1)n+1}^{i+jn}\Lambda_{\psi(q_k,\alpha_k)}(\lambda)\Big ).
\end{align*}
Now, using the definition of $ \mathcal{R}'(m,\ell,L,Q)$ and setting 
$$
A=\sup_{\varrho\in  \mathcal{R}'(m,\ell,L,Q)} \sup_{\lambda\in Q}\sum_{k=1}^L |\Lambda_{\psi_\alpha}(\lambda)-\Lambda_{\psi(q_k,\alpha_k)}(\lambda)|,
$$ 
one obtains 
\begin{align*}
\mathcal Q_\varrho\Big (\Big\{\sum_{j=1}^{k(n)}\widetilde V^{(i)}_j(\varrho,\lambda,t)\ge 1\Big \}\Big )& \le  e^A \sum_{j=1}^{k(n)} \exp \big (n\varepsilon\|\lambda\| -n \langle\lambda|\nabla \Lambda_{\psi_\alpha}(\lambda)\rangle\big ) \exp \Big (n/2\ell +n \Lambda_{\psi_\alpha}(\lambda)\Big )\\
&= e^A \sum_{j=1}^{k(n)} \exp \big (n\varepsilon\|\lambda\|+n/2\ell +n\Lambda^*_{\psi_\alpha}(\nabla \Lambda_{\psi_\alpha}(\lambda))\big )\\
&
\le  e^A\, k(n) \exp \big (n\varepsilon\|\lambda\|-n/2\ell-n\kappa'\big ).
\end{align*}

Thus, taking $0<\varepsilon=\varepsilon_Q$ small enough so that $\varepsilon\|\lambda\|\le \ell/8$, since $\log(k(n))<n(\kappa'+\ell/8)$ for $n$ large enough, one gets 
$$
\sup_{ \varrho\in \mathcal{R}'(m,\ell,L,Q)} \mathcal{Q}_\varrho\Big (\Big\{\sum_{j=1}^{k(n)}\widetilde V^{(i)}_j(\varrho,\lambda,t)\ge 1\Big \}\Big )=O(\exp (-n/4\ell)).
$$
Mimicking  the end of the proof of part (1) of this proposition, we can get that given $\lambda\in Q$,  with $\varepsilon=\varepsilon_Q$
$$
\E\Big (\sum_{n\in\mathbb{N}} \sup_{\varrho\in  \mathcal{R}'(m,\ell,L,Q)}\mu_\varrho\Big (\Big \{t: \,\exists \, 0\le i\le n-1,\, \sum_{j=1}^{k(n)}\widetilde V^{(i)}_j(\varrho,\lambda,t)\ge 1\Big \}\Big )\Big)<\infty,
$$
hence, with probability 1,  by the Borel-Cantelli Lemma applied to each $\mu_\varrho$, one has that for all $\varrho\in  \mathcal{R}'(m,\ell,L,Q)$, for $\mu_\varrho$-almost every $t$, for $n$ large enough, for all $0\le i\le n-1$,  $E_{n,i}(t):=\{1\le j\le k(n): n^{-1}(S_{i+jn}X(t)-S_{i+(j-1)n}X(t))\in B(\nabla \Lambda_{\psi_\alpha}(\lambda),\varepsilon_Q)\}=\emptyset$.   

Now, for each $\lambda\in Q$, there exists $\eta_\lambda>0$ such that for all $\lambda'\in B(\lambda, \eta_\lambda)$, for all $ \varrho\in  \mathcal{R}'(m,\ell,L,Q)$ one has $B(\nabla \Lambda_{\psi_\alpha}(\lambda'),\varepsilon_Q/2)\subset B(\nabla \Lambda_{\psi_\alpha}(\lambda),\varepsilon_Q)$. One can extract from $(B(\lambda,\eta_\lambda))_{\lambda\in Q}$ a finite subfamily $(B(\lambda_s,\eta_{\lambda_s}))_{1\le s\le r}$ which covers $Q$. Since this family is finite, with probability 1, for all $\varrho\in  \mathcal{R}'(m,\ell,L,Q)$, for $\mu_\varrho$-almost every $t$, for $n$ large enough, for all $1\le s\le r $ and $0\le i\le n-1$, one has $E_{n,i}(t)=\emptyset$; consequently, by construction of $(B(\lambda_s,\eta_{\lambda_s}))_{1\le i\le r}$, for all $\lambda'\in Q$ and $0\le i\le n-1$, $E_{n,i}(t)=\emptyset$.  This finishes the proof of the proposition.  
\end{proof}

\begin{proof}[Proof of Lemma~\ref{control4}] (1) This is elementary, but we detail it for reader's convenience. Let $n\ge 1$ and $(f_1,\ldots,f_{k(n)})$ be $k(n)$ positive  Borel functions defined on $\R_+$. Fix $0\le i\le n-1$. One has 
$$
\E_{\mathcal Q_\varrho}\Big (\prod_{j=1}^{k(n)} f_j \big ( Z^{(i)}_{n,j}(\varrho,\lambda,t)\big)\Big )
=\E\Big (\int\prod_{j=1}^{k(n)} f_j ( Z^{(i)}_{n,j}(\varrho,\lambda,t))\,\mathrm{d}\mu_\varrho(t)\Big ).
$$
For each word $u$ of generation $nk(n)$, we denote  by $Z^{(i)}_{n,j}(\varrho,\lambda,u)$ the constant value of $Z^{(i)}_{n,j}(\varrho,\lambda,t)$ over the cylinder $[u]$. Using the fact that $Z^{(i)}_{n,j}(\varrho,\lambda,t)$ is $\sigma(N_u,X_{us}: u\in \bigcup_{k=i+(j-1)n}^{i+jn-1}\N^{k}, s\in \mathbb N)\otimes \mathcal C$-measurable, as well as  the definition of $\mu_\varrho$, the  independence between generations, and the branching property yields 
\begin{align*}
&\E_{\mathcal Q_\varrho}\Big (\prod_{j=1}^{k(n)} f_j \big ( Z^{(i)}_{n,j}(\varrho,\lambda,t)\big)\Big )\\
&=\E\Big (\sum_{u\in\TT_{nk(n)}}Y(\varrho,u)\Big[\prod_{k=(j-1)n+1}^{jn} W_{\varrho,u_1\cdots u_k}\Big ]\cdot \Big [\prod_{j=1}^{k(n)}  f_j \big( Z^{(i)}_{n,j}(\varrho,\lambda,u)\big)\Big ] \Big )\\
&=\E\Big (\sum_{u\in\TT_{nk(n)}}\prod_{j=1}^{k(n)}  \Big [ f_j \big( Z^{(i)}_{n,j}(\varrho,\lambda,u)\big) \prod_{k=i+(j-1)n+1}^{i+jn} W_{\varrho,u_1\cdots u_k}\Big ]\Big ).
\end{align*}
Recall that $W_{\varrho,u_1\cdots u_k}=\exp \big (\langle q_k|X_{u_1\cdots u_k}-\alpha_k\rangle -\wt P_{X,\phi,\alpha_k}(q_k)\phi_{u_1\cdots u_k}\big )$, and set
$$
U_{n,j}(u)=f_j ( Z^{(i)}_{n,j}(\varrho,\lambda,u)) \prod_{k=i+(j-1)n+1}^{i+jn} \exp (\langle q_k|X_{u_1\cdots u_k}-\alpha_k\rangle -\wt P_{X,\phi,\alpha_k}(q_k)\phi_{u_1\cdots u_k}).
$$
Note that this random variable is measurable with respect to  the $\sigma$-algebra $\mathcal{G}_{n,j}$ defined as $\mathcal{G}_{n,j}=\sigma \big ((N_w,(X_{w1},\phi_{w1}),\ldots ): w\in \bigcup_{k=i+j(n-1)}^{i+jn-1} \mathbb N^k\big )$.  Now, the above equality~rewrites
$$
\E_{\mathcal Q_\varrho}\Big (\prod_{j=1}^{k(n)} f_j ( Z^{(i)}_{n,j}(\varrho,\lambda,t))\Big )=\mathbb{E}\Big (\sum_{u\in \TT_{n(k(n)-1)}}\sum_{v\in \TT_n(u)} \prod_{j=1}^{k(n)} U_{n,j}(uv)\Big ).
$$
Conditioning on  $\sigma(\mathcal{G}_{n,j}:1\le j\le k(n)-1)$ and using the independences and identity in distribution between the random variables of the construction one gets
$$
\E_{\mathcal Q_\varrho}\Big (\prod_{j=1}^{k(n)} f_j ( Z^{(i)}_{n,j}(\varrho,\lambda,t))\Big )=\mathbb{E}\Big (\sum_{u\in \TT_{n(k(n)-1)}}\prod_{j=1}^{k(n)-1} U_{n,j}(uv)\Big ) \widetilde U_{n,k(n)},
$$
where  for $1\le j\le k(n)$ one set
\begin{align*}
&\widetilde U_{n,j}=\E\Big (\sum_{u\in\TT_{n}} f_{j} \big ( \exp(\langle \lambda|S_{n}X(u)\rangle-s^{(i)}_{n,j}(\varrho,\lambda)\big ) \\ & \quad\quad\quad\quad\cdot \prod_{k=i+1}^{i+n}\exp (\langle q_{(j-1)n+k}|X_{u_1\cdots u_k}-\alpha_{(j-1)n+k}\rangle -\wt P_{X,\phi,\alpha_{(j-1)n+k}}(q_{(j-1)n+k})\phi_{u_1\cdots u_k})\Big ).
\end{align*}
Iterating the previous calculation yields  
$$
\E_{\mathcal Q_\varrho}\Big (\prod_{j=1}^{k(n)} f_j ( Z^{(i)}_{n,j}(\varrho,\lambda,t))\Big )=\prod_{j=1}^{k(n)} \widetilde U_{n,j},
$$
and applying this with $f_{j'}=1$ for $j'\neq j$ one naturally obtains  $\widetilde U_{n,j}=\E_{\mathcal Q_\varrho}\big ( f_j ( Z^{(i)}_{n,j}(\varrho,\lambda,t))\big )$, hence
$$
\E_{\mathcal Q_\varrho}\Big (\prod_{j=1}^{k(n)} f_j ( Z^{(i)}_{n,j}(\varrho,\lambda,t))\Big )=\prod_{j=1}^{k(n)}\E_{\mathcal Q_\varrho}\big ( f_j ( Z^{(i)}_{n,j}(\varrho,\lambda,t))\big ),
$$
which is the desired independence. Then, taking $f_j(z)=z$ and $f_{j'}(z)=1$ for $j'\neq j$ yields, writing $k'$ for $i+(j-1)n+k$ and dropping $X,\phi$ in $\widetilde P_{X,\phi,\alpha}$:
\begin{align*}
&\E_{\mathcal Q_\varrho}\big (Z_{n,j}(\varrho,\lambda,t)\big )\\
&=\E\Big (\sum_{u\in \TT_n}\prod_{k=1}^{n} \exp \big (\langle \lambda|X_{u_1\cdots u_k}\rangle- \Lambda_{\psi(q_{k'},\alpha_{k'})}(\lambda) +\langle q_{k'}|X_{u_1\cdots u_k}-\alpha_{k'}\rangle -\wt P_{\alpha_{k'}}(q_{k'})\phi_{u_1\cdots u_k}\big )\Big )\\
&=\prod_{k=1}^n \E\Big (\sum_{i=1}^N \exp (\big (\langle \lambda|X_{i}\rangle +\langle q_{k'}|X_{i}-\alpha_{k'}\rangle -\wt P_{\alpha_{k'}}(q_{k'})\phi_{i}- \Lambda_{\psi(q_{k'},\alpha_{k'})}(\lambda)\big )\Big )=1
\end{align*}
by definition of $\Lambda_{\psi(q_{k'},\alpha_{k'})}$. Finally, the random variables  $Z^{(i)}_{n,j}(\varrho,\lambda,t)-1$ are $\mathcal{Q}_\varrho$-independent and centered. 

\smallskip

\noindent
(2) Thanks to (1), we can apply Lemma~\ref{lem-2.1} and get 
\begin{align*}
&\mathcal Q_\varrho \Big (\Big |k(n)^{-1}\sum_{j=1}^{k(n)}(Z^{(i)}_{n,j}(\varrho,\lambda,t)-1)\Big |>\varepsilon \Big )\le (\varepsilon k(n))^{-p}\E_{\mathcal Q_\varrho}\Big (\Big |\sum_{j=1}^{k(n)}(Z^{(i)}_{n,j}(\varrho,\lambda,t)-1)\Big |^p\Big )\\
&\le 2^{p-1}(\varepsilon k(n))^{-p} \sum_{j=1}^{k(n)}\E_{\mathcal Q_\varrho}( |Z^{(i)}_{n,j}(\varrho,\lambda,t)-1|^p)
 \le 2^{2p-1}(\varepsilon k(n))^{-p} \sum_{j=1}^{k(n)}\E_{\mathcal Q_\varrho}( Z^{(i)}_{n,j}(\varrho,\lambda,t)^p)
 \end{align*}
since $\E_{\mathcal Q_\varrho}( Z_{n,j}(\varrho,\lambda,t))=1$. Moreover, calculations similar to those used to establish part (1) of this lemma yield
$$
\E_{\mathcal Q_\varrho}( Z^{(i)}_{n,j}(\varrho,\lambda,t)^p)=\exp\Big( \sum_{k=i+(j-1)n+1}^{i+jn} \Lambda_{\psi(q_k,\alpha_k)}(p\lambda)-p\Lambda_{\psi(q_k,\alpha_k)}(\lambda)\Big ).
$$
Since $ \mathcal{R}(m,\ell,L,Q)\subset \mathcal{R}(K_m)$, and $K_m$ is a compact subset of $J_{X,\phi}$, using Taylor's expansion one gets $\Lambda_{\psi(q_k,\alpha_k)}(p\lambda)-p\Lambda_{\psi(q_k,\alpha_k)}(\lambda)=(1-p)\Lambda^*_{\psi(q_k,\alpha_k)}(\nabla\Lambda_{\psi(q_k,\alpha_k)}\lambda)+O((p-1)^2)$ uniformly in $\varrho\in  \mathcal{R}(m,\ell,L,Q)$, $\lambda\in\wt Q$ and $p-1$ small enough. Consequently, by  definition of $ \mathcal{R}(m,\ell,L,Q)$, for $k\ge L$ one obtains 
$$
\Lambda_{\psi(q_k,\alpha_k)}(p\lambda)-p\Lambda_{\psi(q_k,\alpha_k)}(\lambda)\le  (p-1)\min (2\ell, \kappa-1/2\ell)+  O((p-1)^2),
$$
hence for all $1\le j\le k(n)$
$$
\sum_{k=i+(j-1)n+1}^{i+jn} \Lambda_{\psi(q_k,\alpha_k)}(p\lambda)-p\Lambda_{\psi(q_k,\alpha_k)}(\lambda)\le  A+n\big ( (p-1)\min (2\ell, \kappa-1/2\ell)+  O((p-1)^2)\big )$$
uniformly in $\varrho\in  \mathcal{R}(m,\ell,L,Q)$, $\lambda\in\wt Q$ and $p-1$ small enough, where 
$$
A=\displaystyle  \sup_{p\in [1,2]}\sup_{\varrho \in  \mathcal{R}(m,\ell,L,Q)}\sup_{\lambda\in \wt Q}\sum_{k=1}^L |\Lambda_{\psi(q_k,\alpha_k)}(p\lambda)-p\Lambda_{\psi(q_k,\alpha_k)}(\lambda)|.
$$
The previous estimates yield 
\begin{align*}&\mathcal Q_\varrho \Big (\Big |k(n)^{-1}\sum_{j=1}^{k(n)}(Z^{(i)}_{n,j}(\varrho,\lambda,t)-1)\Big |>\varepsilon \Big )\\&\le  e^A \varepsilon^{-p} (k(n))^{1-p} \exp\big  (n\big ( (p-1)\min (2\ell, \kappa-1/2\ell)+  O((p-1)^2)\big )\big )
\end{align*}
in the same uniform manner as above. Take $p$ close enough to $1$ so that $O((p-1)^2)\le (p-1)/8\ell$. 

Now, for $n$ large enough, one has  $k(n)\ge \exp (n(\min (2\ell,\kappa-1/8\ell)))$, so that 
$$
\mathcal Q_\varrho \Big (\Big |k(n)^{-1}\sum_{j=1}^{k(n)}(Z^{(i)}_{n,j}(\varrho,\lambda,t)-1)\Big |>\varepsilon \Big )\le  e^A \varepsilon^{-p} \exp\big (n\big (1-p)\ell/4)
$$
uniformly in $\varrho\in  \mathcal{R}(m,\ell,L,Q)$, $\lambda\in\wt Q$ and $0\le i\le n-1$. 
\end{proof}

\section{Proof of Theorem~\ref{UNIFER2}}\label{proofLD2}

Let us start by stating the following proposition. 

\begin{proposition}\label{propkey} Assume \eqref{allmomentsfinite} and \eqref{pP}.
\begin{enumerate}

\item $\widetilde I_X\subset \widering {\mathcal{C}}_X$ and $I_X\setminus \widetilde I_X\subset \partial \mathcal{C}_X$. 

\item If $H\in \mathcal H_X\setminus \widetilde {\mathcal H}_X$ then $H\cap I_X=\emptyset$.

\item If $\mathcal H_X\setminus \widetilde {\mathcal H}_X\neq\emptyset$, then $(\partial I_X)_{\rm{crit}}=\widetilde I_X\setminus \widering I_X\neq\emptyset$. 

Assume now that $\mathbb{E}(N^p)<\infty$ for some $p>1$.  

\item If $H\in \widetilde{\mathcal H}_X$, then $H\cap I_X\neq\emptyset$.

Assume, in addition, \eqref{psiX} and \eqref{newphi}.

\item For all $F\in \widetilde {\mathcal F}^1_X\cup \overline {\mathcal F}^1_X$, for all $H\in\widetilde H_X$ such that $F\subset H$, one has $H\cap I_X=F\cap I_X=I^F_X$. Moreover, the conclusions of items (4) and (5) of Theorem~\ref{UNIFER2} hold true for $\alpha\in I^F_X$ and $\alpha\in \widetilde I^F_X$ respectively.

\item If $F,F'\in \widetilde {\mathcal F}^1_X\cup\overline {\mathcal F}^1_X$ with $F\neq F'$  then $\widetilde I^F_{X}\cap\widetilde I^{F'}_{X}=\emptyset$.

\end{enumerate}
\end{proposition}

\begin{proof}[Proof of Theorem~\ref{UNIFER2}] 

The properties $\widetilde{\mathcal H}_X=\{H\in \mathcal H_X:\, H\cap I_X\neq\emptyset\}$, $\widetilde I_X\subset  \widering {\mathcal{C}}_X$,  and $I_X\setminus \widetilde I_X= \bigcup_{H\in\widetilde{\mathcal H}_X} H\cap I_X$, that is point (1) and the first part of point (2) of the theorem follow directly from Proposition~\ref{propkey}(1)-(4). This is also the case of the first part of point (3), that is the property that $(\partial I_X)_{\mathrm{crit}}=\emptyset$ if and only if $\widetilde{\mathcal H}_X={\mathcal H}_X$.

To see why the second part of point (3) holds, recall that the set of exposed points is dense in the set of extremal points of $\mathcal C_X$ (see \cite[Theorem 18.6]{Roc}).  If $\widetilde{\mathcal H}_X={\mathcal H}_X$, the previous properties imply that any supporting  hyperplane $H$ of $\mathcal C_X$ does intersect $\partial I_X$. In particular, for each exposed point $P$ of $\mathcal C_X$, choosing  $H\in \mathcal H_X$ which intersects $C_X$ only at  $P$, we get $P\in I_X$, and  the condition $\mathbb{E}(N^H)\ge 1$ is equivalent to $\mathbb{E}(N^{\{P\}})\ge 1$. This implies that the set of exposed points of  $\mathcal C_X$ is finite (since $\mathbb{E}(N)<\infty$) and coincides with its set of extremal points. It follows that the set $\mathcal C_X$ is a polytope and equals $I_X$. Again due to the properties previously established, this is equivalent to the fact that every exposed point $P$ of $\mathcal C_X$ satisfies $\mathbb{E}(N^{\{P\}})\ge 1$.

Suppose that $I_X\setminus \widetilde I_X\neq\emptyset$. It follows from the first part of point (5) of Proposition~\ref{propkey} that $I_X\setminus \widetilde I_X=\bigcup_{F\in \widetilde {\mathcal F}^1_X\cup \overline {\mathcal F}^1_X}I^F_X$. Then, the second part of this fifth point combined with point (6) of the same proposition and a recursion on the dimension of $I_X$ yields both that $I_X\setminus \widetilde I_X=\bigsqcup_{F\in \widetilde {\mathcal F}^d_X\cup \overline {\mathcal F}^d_X}\widetilde I^F_X$ and points (4) and (5) of Theorem~\ref{UNIFER2}.
 \end{proof} 

If $H\in\mathcal H_X$, there is a unique couple $(e,c)\in \mathbb S^{d-1}\times \mathbb R$ such that 
\begin{equation}\label{H1}
H=L_{e,c}^{-1}(\{0\})\text{ and } \mathcal C_X\subset L_{e,c}^{-1}(\R_-), \text{ where } L_{e,c}:\beta\in\R^d\mapsto  \langle e| \beta \rangle-c.
\end{equation}
We set 
$$
H_+=L_{e,c}^{-1}(\R_+^*).
$$
The following preliminary observation will be useful.
\begin{lemma}\label{lemreduc}
Let $H\in \mathcal H_X$. One has $\mathbb{E}\Big (\sum_{i=1}^N\mathbf{1}_{H_+}(X_i)\Big )=0$.
\end{lemma}
\begin{proof}
Suppose that $\theta=\mathbb{E}\Big (\sum_{i=1}^N\mathbf{1}_{H_+}(X_i)\Big )>0$ and set $W_i=\theta^{-1}\mathbf{1}_{H_+}(X_i)\mathbf{1}_{[1,N]}(i)$ for all $i\ge 1$. By construction of $(W_i)_{i\in\mathbb{N}}$ and by definition of $\mathcal{C}_X$ and $\mathcal H_X$, $\beta=\mathbb{E}\Big (\sum_{i=1}^NW_iX_i\Big )\in  \mathcal{C}_X\subset \R^d\setminus H_+$, so $\langle e|\beta\rangle\le c$. However, $\langle e|\beta\rangle=\theta^{-1} \mathbb{E}\Big (\sum_{i=1}^N\mathbf{1}_{H_+}(X_i)\langle e|X_i\rangle \Big )>c$. This contradiction implies that $\theta=0$. 
\end{proof}

We also define, for $\varepsilon>0$, 
\begin{equation}\label{H2}
U_\varepsilon = L_{e,c}^{-1}((-\infty,-\varepsilon))\text{ and }V_\varepsilon=\R^d\setminus U_\varepsilon.
\end{equation}

\begin{proof}[Proof of Proposition~\ref{propkey}] (1) The fact that $\widetilde I_X\subset \widering {\mathcal{C}}_X$ follows from the  inclusion  $\nabla\widetilde P_X(\R^d)\subset \mathcal{C}_X$ and the fact that the mapping $q\in\R^d \mapsto \nabla\widetilde P_X(q)$ is open.

Suppose, by contradiction, that there exists $\alpha\in (I_X\setminus \widetilde I_X) \cap \widering{\mathcal C}_X$.  Given such an $\alpha$, there exists a sequence $(q_n)_{n\in\N}$ of elements of $J_X$ such that $\alpha=\lim_{n\to+\infty}(\alpha_n=\nabla \widetilde P_X(q_n))$. Moroever, $\lim_{n\to\infty} \|q_n\|=+\infty$, for otherwise it is easily seen that $\alpha\in \widetilde I_X$.  

Without loss of generality, let us assume that $e_n=q_n/\|q_n\|$ converges to $ e\in\mathbb{S}^{d-1}$. 

Recall that 
\[
\alpha_n=\nabla \widetilde P_X(q_n)=\mathbb{E}\Big (\sum_{i=1}^N\exp (\|q_n\|\langle e_n|X_i\rangle-\widetilde P_X(q_n)) X_i\Big ).
\]
Since we assumed that $\alpha$ is an interior point of $\mathcal C_X$, there exists $c>\langle e|\alpha\rangle$ such that $L_{e,c}^{-1}(\{0\})\cap \widering {\mathcal C}_X\neq\emptyset$. This implies that  for all $\varepsilon>0$, one has $\mathbb{E}\Big (\sum_{i=1}^N\mathbf{1}_{V_{\varepsilon}}(X_i)\Big )>0$; indeed, otherwise there would exist $\varepsilon_1>0$ such that $\mathbb{P}\big (X_i\in U_{\varepsilon_1}, \,\forall \, 1\le i\le N\big )=1$, hence for each non negative random vector $(W_i)_{i\in\mathbb N}$  jointly defined with $(N,(X_i))_{i\in\mathbb N}$ and such that  $\mathbb E(\sum_{i=1}^N W_i)=1$, one would have   $\big\langle e|\mathbb E(\sum_{i=1}^N W_iX_i)\big \rangle \le c-\varepsilon_1$, contradicting  $L_{e,c}^{-1}(\{0\})\cap {\mathcal C}_X\neq\emptyset$. 

Consequently, since $(e_n)_{n\in\mathbb{N}}$ converges to $e$, for all $\varepsilon>0$ there exist $n_\varepsilon\in\mathbb N$ and $A_\varepsilon>0$ such that if $n\ge n_\varepsilon$, setting $U_{n,\varepsilon} = L_{e_n,c}^{-1}((-\infty,-\varepsilon))\text{ and }V_{n,\varepsilon} =\R^d\setminus U_{n,\varepsilon}$:
$$
\mathbb{E}\Big (\sum_{i=1}^N\mathbf{1}_{V_{n,\varepsilon}}(X_i)\Big )\ge \mathbb{E}\Big (\sum_{i=1}^N\mathbf{1}_{V_{\varepsilon/2}}(X_i)\mathbf{1}_{[0,A_\varepsilon]}(\|X_i\|)\Big )>0.
$$ 

Fix $\varepsilon>0$ such that $\alpha\in U_{4\varepsilon}$. Noting that $\mathbb{E}\big (\sum_{i=1}^N\exp (\langle q_n|X_i\rangle-\widetilde P_X(q_n))\big )=1$ for all $n\in\mathbb N$, if $n\ge n_\varepsilon$ one has
\begin{align*}
\Big \|\mathbb{E}\big (\sum_{i=1}^N\mathbf{1}_{U_{n,3\varepsilon}}(X_i) X_ie^{\langle q_n|X_i\rangle-\widetilde P_X(q_n)}\big )\Big \| &\le \frac{\mathbb{E}\big (\sum_{i=1}^N\mathbf{1}_{U_{n,3\varepsilon}}(X_i) \|X_i\|\exp (\|q_n\|\langle e_n|X_i\rangle)\big )}{ \mathbb{E}\big (\sum_{i=1}^N\mathbf{1}_{V_{n,\varepsilon}}(X_i)\exp (\|q_n\|\langle e_n|X_i\rangle)\big )}\\
&\le  \frac{\mathbb{E}\big (\sum_{i=1}^N\mathbf{1}_{U_{n,3\varepsilon}}(X_i) \|X_i\|\exp (\|q_n\|(c-3\varepsilon)\big )}{ \mathbb{E}\big (\sum_{i=1}^N\mathbf{1}_{V_{n,\varepsilon}}(X_i)\exp (\|q_n\|(c-\varepsilon))\big )}\\
&\le  \frac{\mathbb{E}\big (\sum_{i=1}^N\|X_i\|\big)}{\mathbb{E}\Big (\sum_{i=1}^N\mathbf{1}_{V_{\varepsilon/2}}(X_i)\mathbf{1}_{[0,A_\varepsilon]}(\|X_i\|)\Big )}e^{-2\|q_n\|\varepsilon}.
\end{align*}
It follows that $
\lim_{n\to\infty} \big \|\mathbb{E}\big (\sum_{i=1}^N\mathbf{1}_{U_{n,3\varepsilon}}(X_i) X_ie^{\langle q_n|X_i\rangle-\widetilde P_X(q_n)}\big )\big \|=0.$ 
This yields that 
\begin{align*}
\langle e|\alpha\rangle&=\lim_{n\to\infty} \langle e_n|\alpha_n\rangle= \lim_{n\to\infty} \mathbb{E}\Big (\sum_{i=1}^N\mathbf{1}_{V_{n,3\varepsilon}}(X_i) \langle e_n|X_i\rangle e^{\langle q_n|X_i\rangle-\widetilde P_X(q_n)}\Big ) \ge c-3\varepsilon,
\end{align*}
since the previous estimates imply that  $\lim_{n\to\infty} \mathbb{E}\big (\sum_{i=1}^N\mathbf{1}_{V_{n,\varepsilon'}}(X_i) e^{\langle q_n|X_i\rangle-\widetilde P_X(q_n)}\big )=1$ for all $\varepsilon'>0$. However, $\langle e|\alpha\rangle\le c-4\varepsilon$ since $\alpha\in U_{4\varepsilon}$, so we meet  a contradiction.

\noindent
(2) Fix $H\in \mathcal H_X\setminus \widetilde {\mathcal H}_X$. Recall that $\mathbb{E}(N^H)<1$. Suppose that $H\cap I_X\neq\emptyset$. Let $\alpha\in H\cap I_X$ and $(q_n)_{n\in\mathbb N}\in J_X^{\mathbb N}$ such that $\lim_{n\to\infty} \nabla\widetilde P_X(q_n)=\alpha$. 

For $\varepsilon>0$,  set
$$
\begin{cases}
G_{X,\varepsilon}: (q,\beta)\in J_X\times[0,1]\mapsto \mathbb{E}\big(\sum_{i=1}^N \mathbf{1}_{U_\varepsilon}(X_i) e^{\beta(\langle q|X_i\rangle-\widetilde P_X(q))}\big )\\
\widetilde G_{X,\varepsilon}: (q,\beta)\in J_X\times[0,1]\mapsto \mathbb{E}\big (\sum_{i=1}^N \mathbf{1}_{V_\varepsilon}(X_i)e^{\beta(\langle q|X_i\rangle-\widetilde P_X(q))}\big).
\end{cases}$$

Note that given $q\in J_X$, the mapping $G_X(q,\cdot)=G_{X,\varepsilon}(q,\cdot)+\widetilde G_{X,\varepsilon}(q,\cdot)$, which neither depends on $\varepsilon$ nor on $H$, is convex, takes values $\mathbb{E}(N)>1$ and $1$ at  $\beta=0$ and $\beta=1$ respectively, and has $-\widetilde P_X^*(\nabla \widetilde P_X(q))<0$ as left derivative at $\beta=1$.  Thus, $G_X(q,\beta)>1$ for all $\beta\in (0,1)$. Below we prove that the existence of $\alpha$ contradicts this fact. 

Fix $\beta\in(0,1)$ and $\rho>0$ such that $\beta+(1-\beta)\mathbb{E}(N^H)+\rho<1$.  For any $\eta\in (0,1]$ and $q\in J_X$ one has, setting $W_{q,i}=\exp (\langle q|X_i\rangle-\widetilde P_X(q))$:
\begin{align*}
&G_{X,\varepsilon}(q,\beta)\le \eta^\beta \mathbb{E}\Big (\sum_{i=1}^N \mathbf{1}_{U_\varepsilon}(X_i) \mathbf{1}_{\{W_{q,i}\le \eta\}}\Big)+ \mathbb{E}\Big(\sum_{i=1}^N \mathbf{1}_{U_\varepsilon}(X_i)\mathbf{1}_{\{W_{q,i}> \eta\}}W_{q,i}^{\beta-1} W_{q,i}\Big ) \\
&\le \eta^\beta \mathbb{E}(N)+ \eta^{\beta-1} \mathbb{E}\Big(\sum_{i=1}^N \mathbf{1}_{U_\varepsilon}(X_i)\mathbf{1}_{\{W_{q,i}> \eta\}} W_{q,i}\Big) \le \eta^\beta \mathbb{E}(N)+\eta^{\beta-1} G_{X,\varepsilon}(q,1).
\end{align*}
Fix $\eta>0$ such that $ \eta^\beta\mathbb{E}(N)\le \rho/4$ and then $n=n(\rho,\varepsilon)\in \mathbb N$ such that $\eta^{\beta-1} G_{X,\varepsilon}(q_n,1)\le \rho/4$, and consequently  $G_{X,\varepsilon}(q_n,\beta)\le \rho/2$. This is possible since $\lim_{n\to\infty} G_{X,\varepsilon}(q_n,1)=0$. Indeed, 
$
\Big\langle e|\mathbb{E}\Big (\sum_{i=1}^N \mathbf{1}_{U_\varepsilon}(X_i) W_{q_n,i}^\beta X_i\Big)\Big \rangle \le (c-\varepsilon) G_{X,\varepsilon}(q_n,1),
$
so that due to Lemma~\ref{lemreduc},  $\langle e|\nabla\widetilde P_X(q_n)\rangle\le c\,  \widetilde G_{X,\varepsilon}(q_n,1)+(c-\varepsilon) G_{X,\varepsilon}(q_n,1)\le c -\varepsilon G_{X,\varepsilon}(q_n,1)$, and consequently  $c=\langle e|\alpha\rangle \le c-\varepsilon\limsup_{n\to\infty} G_{X,\varepsilon}(q_n,1)$. 

Now, note that  $\widetilde G_{X,\varepsilon}(q_n,0)$ tends to $\mathbb{E}(N^H)<1$ as $\varepsilon\to 0$ (due to Lemma~\ref{lemreduc} again), and $\widetilde G_{X,\varepsilon}(q_n,1)=G_X(q_n,1)-G_{X,\varepsilon}(q_n,1)\le 1$. It follows from  the convexity of  $\widetilde G_{X,\varepsilon}(q_n,\cdot)$ that if $\varepsilon$ is chosen small enough one has  $\widetilde G_{X,\varepsilon}(q_n,\beta)\le \beta+(1-\beta)\mathbb{E}(N_H) +\rho/2$. Finally, 
$$
G_X(q_n,\beta)=G_{X,\varepsilon}(q_n,\beta)+\widetilde G_{X,\varepsilon}(q_n,\beta)\le \beta+(1-\beta)\mathbb{E}(N^H)+\rho<1, 
$$
which is the expected contradiction.

\noindent
(3) Suppose that $ \mathcal H_X\setminus \widetilde {\mathcal H}_X\neq\emptyset$ and fix  $H\in \mathcal H_X\setminus \widetilde {\mathcal H}_X$. Fix also $\beta\in H\cap \mathcal C_X$ and $\alpha\in \widering I_X$. Since $\beta\not\in I_X$, there exists $\alpha'\in [\alpha,\beta)$ such that  $\alpha'\in \partial I_X$ and $\alpha'\not\in \partial \mathcal C_X$. By point (1) of this proposition, this implies that $\alpha'\in\widetilde I_X\setminus \widering I_X$. 

\noindent
(4) For $\varepsilon\in (0,1)$, define $\gamma_{\varepsilon}= \mathbb{E}(N^H)+\varepsilon \mathbb{E}(N^{H^c})$.  Then, for $1\le i\le N$ set $W_{\varepsilon,i}=\frac{1}{\gamma_{\varepsilon}}\mathbf{1}_H(X_i)+\frac{\varepsilon}{\gamma_{\varepsilon}} \mathbf{1}_{H^c}(X_i)$ and set $W_{\varepsilon,i}=0$ for $i>N$. Note that since $\mathbb{E}(N^H)\ge 1$, and \eqref{pP} implies that $\mathbb{E}(N^{H^c})>0$, one has $\sup_{i\ge 1}W_{\varepsilon,i}< 1$. Consequently $-\mathbb{E}\big (\sum_{i=1}^N W_{\varepsilon,i}\log W_{\varepsilon,i}\big )$ is positive; also, since $\mathbb{E}(N^p)<\infty$ for somme $p>1$, one has  $\mathbb{E}\big (\big (\sum_{i=1}^N W_{\varepsilon,i}\big )^p\big )<\infty$. In addition, $\mathbb{E}\big (\sum_{i=1}^N W_{\varepsilon,i} X_i\big )$ converges to $\alpha_H=\mathbb{E}\big (\sum_{i=1}^N \mathbf{1}_H(X_i)X_i\big )/\mathbb{E}(N^H)$ as $\varepsilon$ tends to~0.

For each $u\in \bigcup_{n\ge 0}\mathbb N^{\mathbb N}$, set 
$W_{\varepsilon,i}(u)=\frac{1}{\gamma_{\varepsilon}}\mathbf{1}_H(X_{ui})+\frac{\varepsilon}{\gamma_{\varepsilon}} \mathbf{1}_{H^c}(X_{ui})$ for $1\le i\le N_u$ and $W_{\varepsilon,i}(u)=0$ for $i>N_u$. Mimicking what was done to construct the measures $\mu_\varrho$ ($\varrho\in\mathcal R$) and determine the behavior of $(S_nX)_{n\in\mathbb{N}}$  almost everywhere with respect to such a  measure (Proposition~\ref{pp3'}), we can find a non increasing positive sequence $(\varepsilon_n)_{n\ge 0}$ converging to 0 such that the inhomogeneous Mandelbrot martingale associated with the random vectors $(W_{\varepsilon_{|u|},i}(u))_{i\in\mathbb N}$, $u\in \bigcup_{n\ge 0}\mathbb N^{\mathbb N}$, yields almost surely a positive measure $\mu^H$ supported on $E(X,\alpha_H)$.

\noindent
(5)  Let $F\in \widetilde {\mathcal F}^1_X\cup \overline {\mathcal F}^1_X$. Note that the fact that $F\cap I_X\neq\emptyset$ can be obtained directly in the same way as  $H\cap I_X\neq\emptyset$ when $H\in\widetilde H_X$.

Fix $H\in \widetilde {\mathcal H}_X$ such that $F\subset H$.

We distinguish the cases $F\in \widetilde {\mathcal F}^1_X$ and $F\in \overline{\mathcal F}^1_X$.  

{\bf Case 1:}  $F\in \widetilde {\mathcal F}^1_X$. Remember that in this case $\mathbb{E}(N^F)>1$. 

We first prove that $ H\cap I_X\subset F\cap I_X\subset  I^F_X$. 

Let $\alpha\in H\cap I_X$.  According to Lemma~\ref{approxi2}, we can take a sequence $(q_n,\alpha_n)_{n\in\mathbb{N}}$ of elements of $D$ such that $\displaystyle\lim_{n\to\infty}  \beta(q_n,\alpha_n)=\alpha$ and $\displaystyle\lim_{n\to\infty} \wt P_{X,\phi,\alpha_n}(q_n)-\langle q_n|\nabla \wt P_{X,\phi,\alpha_n}(q_n)\rangle=\wt P^*_{X,\phi,\alpha}(0)$. Set $W_{n,i}=\mathbf{1}_{[1,N]}(i)e^{\langle{q_n|X_i-\alpha_n\rangle }-\widetilde P_{X,\phi,\alpha_n}(q_n)}$ for  $i\ge 1$. With the notations \eqref{entropie}, \eqref{lyap} and \eqref{alphaXphi},  one has $h(q_n,\alpha_n)= -\mathbb E\big (\sum_{i=1}^N W_{n,i}\log W_{n,i}\big )$, $\lambda(q_n,\alpha_n)=\mathbb E\big (\sum_{i=1}^N W_{n,i}\phi_i\big )$, $\wt P_{X,\phi,\alpha_n}(q_n)-\langle q_n|\nabla \wt P_{X,\phi,\alpha_n}(q_n)\rangle=\frac{h(q_n,\alpha_n)}{\lambda(q_n,\alpha_n)}$, and $ \beta(q_n,\alpha_n)=\mathbb E\big (\sum_{i=1}^N W_{n,i}X_i\big )$.

For any Borel subset $V$ of $\R^d$ and any  $n\in\mathbb N$, set 
$$
p_{V,n}=\mathbb{E}\Big (\sum_{i=1}^N \mathbf{1}_V(X_i) W_{n,i}\Big).   
$$
Then, define the sequence of vectors $(\widetilde W_{n,i})_{i\ge 1}$, $n\in\mathbb{N}$, by $\widetilde W_{n,i}=\mathbf{1}_{[1,N]}(i)p_{F,n}^{-1}\mathbf{1}_F(X_i) W_{n,i}$.

\begin{lemma}\label{AAA}
\begin{enumerate} 
\item $\lim_{n\to\infty} \mathbb E\big (\sum_{i=1}^N \widetilde W_{n,i}X_i\big )=\alpha$. In particular $\alpha\in F$. 

\item  $\liminf_{n\to\infty} \frac{\sum_{k=1}^n-\mathbb E\big (\sum_{i=1}^N \widetilde W_{k,i}\log \widetilde W_{k,i}\big )}{\sum_{k=1}^n\mathbb E\big (\sum_{i=1}^N \widetilde W_{k,i}\phi_i\big )}\ge \widetilde P_{X,\phi,\alpha}^*(0)$. 
\end{enumerate}
\end{lemma}
Assume this lemma for a while. Let $(W^F_{i})_{i\ge 1}$ be defined by $W^F_{i}=\mathbf{1}_{[1,N]}(i)\mathbf{1}_F(X_i)/\mathbb{E}(N^F)$. For $\theta\in (0,1]$, $n\ge 1$ and $i\ge 1$ set $\widetilde W_{n,i}(\theta)=\theta W^F_{i}+(1-\theta)\widetilde W_{n,i}$. By convexity of the function $x\log(x)$, noting that $\mathbb E\Big (\sum_{i=1}^N W^F_{i}\log W^F_{i}\Big )=-\log(\mathbb{E}(N^F))$, one has 
\begin{align}
\label{compare entropies}\mathbb E\Big (\sum_{i=1}^N \widetilde W_{n,i}(\theta)\log \widetilde W_{n,i}(\theta)\Big )
\le -\theta\log(\mathbb{E}(N^F))+(1-\theta)\mathbb E\Big (\sum_{i=1}^N\widetilde W_{n,i}\log \widetilde W_{n,i}\Big ).
\end{align}
Then, Lemma~\ref{AAA}(2) together with $\mathbb{E}(N^F)>1$ yields  $-\mathbb E\big (\sum_{i=1}^N \widetilde W_{n,i}(\theta)\log \widetilde W_{n,i}(\theta)\big )>0$. Let now $(\theta_n)_{n\in\mathbb{N}}$ be a positive sequence converging to 0. One has 
$$
\liminf_{n\to\infty} \frac{ \mathbb E\big (\sum_{i=1}^N \widetilde W_{n,i}(\theta_n)\log \widetilde W_{n,i}(\theta_n)\big )}{\mathbb E\big (\sum_{i=1}^N \widetilde W_{n,i}\log \widetilde W_{n,i}\big)}\ge 1\text{ and }\lim_{n\to\infty} \frac{ \mathbb E\big (\sum_{i=1}^N \widetilde W_{n,i}(\theta_n)\phi_i\big )}{\mathbb E\big (\sum_{i=1}^N \widetilde W_{n,i}\phi_i\big )}=1,
$$
where the inequality comes from \eqref{compare entropies} and the equality is direct. Then, Lemma~\ref{AAA}(2) again implies $\liminf_{n\to\infty} \frac{\sum_{k=1}^n-\mathbb E\big (\sum_{i=1}^N \widetilde W_{k,i}(\theta_k)\log \widetilde W_{n,i}(\theta_k)\big )}{\sum_{k=1}^n\mathbb E\big (\sum_{i=1}^N \widetilde W_{k,i}(\theta_k)\phi_i\big )}\ge \widetilde P_{X,\phi,\alpha}^*(0)$. Now, note  that the random vectors $(\widetilde W_{n,i}(\theta_n))_{n\in\mathbb{N}}$ with positive ``entropy'' can be used to construct an inhomogeneous Mandelbrot martingale which, conditionally on $\partial \TT^F\neq\emptyset$, converges to a positive limit, and makes it possible to define a positive measure $\mu^{\alpha}$ fully supported on $\partial \TT^F$. Moreover, the martingale can be adjusted so that the conclusions of Propositions~\ref{lb2} hold, as well as that of Proposition~\ref{pp3'} applied to $\mu^\alpha$ and  $(S_nX)_{n\in\mathbb{N}}$ restricted to $\partial T^F$: $\underline \dim  (\mu^\alpha)\ge  \liminf_{n\to\infty} \frac{\sum_{k=1}^n-\mathbb E\big (\sum_{i=1}^N \widetilde W_{k,i}(\theta_k)\log \widetilde W_{n,i}(\theta_k)\big )}{\sum_{k=1}^n\mathbb E\big (\sum_{i=1}^N \widetilde W_{k,i}(\theta_k)\phi_i\big )}\ge \widetilde P_{X,\phi,\alpha}^*(0)$, and for $\mu^\alpha$-a.e. $t\in \partial \TT^F$, 
$\lim_{n\to\infty} n^{-1}S_nX(t)= \lim_{n\to\infty}n^{-1}\sum_{k=1}^n \mathbb E\Big (\sum_{i=1}^N \widetilde W_{k,i}(\theta_k)X_i\Big )=\alpha$, where Lemma~\ref{AAA}(1) was used to get the second equality. In other words, recalling the definitions introduced before the statement of Theorem~\ref{UNIFER2},  the $\vec{F}$-valued branching random walk  $(S_n (X_F-\alpha_F))_{n\in\mathbb{N}}$ on $\partial \widetilde \TT^F$ satisfies  that $\alpha-\alpha_F\in I_{X_F-\alpha_F}$, so $\alpha\in I_X^F$ by definition of $I_X^F$. 

Thus, we proved that $H\cap I_X=F\cap I_X\subset I_X^F$. Moreover, for all $\alpha\in F\cap I_X$, conditional on $\partial\TT^F\neq\emptyset$, one has $(\widetilde P_{X_F-\alpha_F,\phi_F,\alpha-\alpha_F})^*(0)\ge \dim (E(X,\alpha)\cap \partial\TT^F)\ge  \widetilde P_{X,\phi,\alpha}^*(0)$, where the second inequality was just proved and for the first inequality one  uses the fact that by definition  $E(X,\alpha)\cap \partial\TT^F=\boldsymbol{b}_F(E(X_F-\alpha_F,\alpha-\alpha_F))$, where $\boldsymbol{b}_F$ is an isometry between $(\partial\widetilde\TT^F,\mathrm{d}_{\phi_F})$ and $(\partial\TT^F,\mathrm{d}_{\phi})$, and the fact that Proposition~\ref{rn} holds for $(S_n(X_F-\alpha_F))_{n\in\mathbb N}$ on  $(\partial\widetilde\TT^F,\mathrm{d}_{\phi_F})$.

Let us now prove that $I_X^F\subset F\cap I_X$, as well as Theorem~\ref{UNIFER2}(4) and (5).  

This time, we use what we know about the $\vec{F}$-valued branching random walk $(S_nX-n\alpha_F)_{n\in\mathbb{N}}$ on $\partial\TT^F$, conditionally on $\partial \TT^F\neq\emptyset$, thanks to Theorems~\ref{thm-1.1} and~\ref{UNIFER} and their proofs applied to $(S_n(X_F-\alpha_F))_{n\in\mathbb{N}}$ on $\partial \widetilde\TT^F$ and the isometry $\boldsymbol{b}_F$ between $(\partial\widetilde\TT^F,\mathrm{d}_{\phi_F})$ and $(\partial\TT^F,\mathrm{d}_{\phi})$. We can consider a family $(\mu_\varrho^F)_{\varrho\in \mathcal {R}^F}$ of inhomogeneous Mandelbrot  measures simultaneously constructed and fully supported on $\partial\TT^F$, and dedicated to the study of $(S_nX-n\alpha_F))_{n\in\mathbb N}$ on $\partial\TT^F$. For each sequence $\rho\in \mathcal {R}^F$, the measure $\mu^F_\rho$ is constructed by using, at each generation $n\ge 1$ of the associated multiplicative cascade, independent copies of non negative random vectors $(\widetilde W_{\varrho_n})_{n\in\mathbb{N}}\in \R_+^{\mathbb{N}}$ simultaneously defined with $(N,(\phi_i)_{i\in\mathbb N} )$ such that $\mathbb{E}\big (\sum_{i=1}^NW_{\varrho_n,i}\big )=1=\mathbb{E}\big (\sum_{i=1}^N\mathbf{1}_{F}(X_i)W_{\varrho_n,i}\big )$. Let us denote these copies by $((W_{\varrho_n,ui})_{i\in\mathbb N} )_{u\in\mathbb N^{n-1}}$. For any positive sequence $\theta=(\theta_n)_{n\in\mathbb N}$ converging to $0$ define
$$
W_{\varrho_n, ui}(\theta)=(\lambda_{\varrho,\theta}(n))^{-1}\Big (\mathbf{1}_{[1, N_u]}(i)\cdot  (\mathbf{1}_{F}(X_{ui})W_{\varrho_n,ui}+\theta_n\mathbf{1}_{F^c}(X_{ui}))\Big )_{i\ge 1},
$$
where $\lambda_{\varrho,\theta}(n)=\mathbb{E}\Big (\sum_{i=1}^N\mathbf{1}_{F}(X_i)W_{\varrho_n,i}+\theta_n\mathbf{1}_{F^c}(X_i)\Big )$.  It is easily seen that following the lines of the proof of Theorem~\ref{thm-1.1},  one can choose the small perturbation $\theta$ so that these new weights make it possible to define almost surely a family $(\mu_\varrho)_{\varrho\in \mathcal {R}^F}$ of  inhomogeneous Mandelbrot measures, all fully supported on $\partial \TT$, and such that with probability 1: (1) for all $\rho\in\mathcal R^F$, $\underline \dim(\mu_\rho)$ equals the value taken by $\underline \dim(\mu^F_\rho)$ almost surely, conditionally on $\partial\TT^F\neq\emptyset$; (2) $\mu_\rho$ is supported on $E(X,\alpha)$ if and only if conditionally on $\partial\TT^F\neq\emptyset$, $\mu_\rho^F$ is supported on $E(X-\alpha_F,\alpha-\alpha_F)\cap \partial\TT^F$. This implies that if $\alpha\in I_X^F$, then $\alpha\in I_X\cap F$, and $\widetilde P_{X,\phi,\alpha}^*(0)=\dim E(X,\alpha)\ge (\widetilde P^F_{X_F-\alpha_F,\phi_F,\alpha-\alpha_F})^*(0)$ almost surely. This, together  with previous estimates shows that with probability 1, for all $\alpha\in I_X^F=F\cap I_X$, one has $\dim E(X,\alpha)=\widetilde P_{X,\phi,\alpha}^*(0)=(\widetilde P^F_{X_F-\alpha_F,\phi_F,\alpha-\alpha_F})^*(0)$, hence the conclusion of Theorem~\ref{UNIFER2}(4) for $F$. 

To get the conclusion of Theorem~\ref{UNIFER2}(5) for $F$, we proceed just as above, but this time we use a small perturbation of the martingales generating the family of measures $(\mu_\varrho^F)_{\varrho\in\bigcup_{n\in\mathbb{N}}\mathcal{R}^F(m)}$ that we would use in the proof of Theorem~\ref{UNIFER} to treat the case of the $\vec{F}$-valued branching random walk  $(S_nX-n\alpha_F)_{n\in\mathbb N}$ on $\partial \TT^F$. This is enough to conclude.

\begin{proof}[Proof of Lemma~\ref{AAA}] 

We begin with preliminary observations. 

At first, for all $n\in\mathbb N$, the fact that $(q_n,\alpha_n)\in J_{X,\phi}$ implies that the entropy $h(q_n,\alpha_n)=-\mathbb E\big (\sum_{i=1}^NW_{n,i}\log W_{n,i} \big )$ is positive. Consequently, 
$$
\mathbb E\big (\sum_{i=1}^NW_{n,i}\log^+ W_{n,i} \big )\le \mathbb E\big (\sum_{i=1}^NW_{n,i}\log^- W_{n,i} \big )\le e^{-1}\mathbb E(N).
$$ 
This uniform bound plays a crucial r\^ole in the estimates below and justifies the introduction of the assumption \eqref{psiX}. Also, let $\psi$ such that \eqref{psiX} holds. Without loss of generality we can assume that $\psi(x)=0$. Define the  non negative convex function $\Psi:x\ge 0\mapsto \exp(\psi(x))-1$.  Due to our assumptions on $\psi$, $\Psi$ satisfies $\lim_{x\to\infty} \Psi(x)/x=\infty$ as well as  $\Psi(0)=0$. In the language of Orlicz spaces theory \cite{RR,Neveu}, both $\psi$ and $\Psi$ are the restrictions to $\R_+$ of strict Young functions. The convex conjugate of $\Psi$ is the function $\Phi$ defined as $\Phi:y\in\R\mapsto \sup\{x|y|-\Psi(x):\, x\ge 0\}$. It is a strict Young function as well, and it is not difficult to see that $\lim_{x\to\infty} \psi(x)/x=\infty$ implies that $\lim_{x\to\infty} \Phi(x)/(x\log(x))=0$. Let $G:x\ge 0\mapsto \Phi^{-1}(x\log^+(x))$, where $\Phi^{-1}$ stands for the right-continuous  inverse of $\Phi$. One has $\lim_{x\to\infty}G(x)/x=\infty$. Set $g:x\ge 1\mapsto \sup\{z/G(z): z\ge x\}$. For all $a\ge 1$,
$$
\mathbb{E}\Big (\sum_{i=1}^N\mathbf{1}_{(a,\infty)} (W_{n,i}) W_{n,i}\|X_i\|\Big )\le g(a)\mathbb{E}\Big (\sum_{i=1}^N G(W_{n,i})\|X_i\|\Big ). 
$$
Moreover, H\"older's inequality for Orlicz spaces yields 
$$
\mathbb{E}\Big (\sum_{i=1}^N G(W_{n,i})\|X_i\|\Big )\le 2\, \|(G(W_{n,i}))_{1\le i\le N}\|_\Phi\cdot   \|(X_i)_{1\le i\le N}\|_\Psi,
$$
where  $\|(Z_i)_{1\le i\le N}\|_\Upsilon=\inf\{k>0:\,  \mathbb{E}\big (\sum_{i=1}^N \Upsilon(Z_i/k)\big ) \le 1\}$. 
 However, it is clear that $\|(X_i)_{1\le i\le N}\|_\Psi<\infty$ and since $\Phi(z/k)\le \Phi(z)/k$ for all $z\ge 0$ and $k\ge 1$ (by convexity and the fact that $\Phi(0)=0$), one has $\mathbb{E}\big (\sum_{i=1}^N \Phi (G(W_{n,i})/k)\big )\le k^{-1} \mathbb E\big (\sum_{i=1}^NW_{n,i}\log^+ W_{n,i} \big )\le k^{-1} e^{-1}\mathbb E(N)$ for $k\ge 1$, so $\|(G(W_{n,i}))_{1\le i\le N}\|_\Phi  \le \lfloor e^{-1}\mathbb E(N)\rfloor +1$ for all $n\ge 1$. Finally,  there exists  $C_{N,X,\psi}>0$ depending on  $(N,X,\psi)$ only such that 
\begin{equation*}\label{unifint}
\mathbb{E}\Big (\sum_{i=1}^N\mathbf{1}_{(a,\infty)} (W_{n,i}) W_{n,i}\|X_i\|\Big )\le C_{N,X,\psi} g(a) \quad \text{for all $a\ge 1$ and $n\in\mathbb N$}, 
\end{equation*}

Now we prove assertion (1). Since $\beta(q_n,\alpha_n)=\mathbb E\big (\sum_{i=1}^NW_{n,i}X_i\big )$ converges to $\alpha$ as $n\to\infty$, it is enough to prove that 
\begin{equation}\label{target}
\lim_{n\to\infty} \Big\|\mathbb E\Big (\sum_{i=1}^NW_{n,i}X_i\Big )-p_{F,n}\mathbb E\Big (\sum_{i=1}^N\widetilde W_{n,i}X_i\Big )\Big \|=0 \text{ and }\lim_{n\to\infty} p_{F,n}=1.
\end{equation}
Recall the notations \eqref{H1} and~\eqref{H2}. Due to Lemma~\ref{lemreduc}, for all $\varepsilon>0$ one has $\lim_{n\to\infty} \big (p_{V_\varepsilon,n}=\mathbb E\big (\sum_{i=1}^N \mathbf{1}_{V_\varepsilon}(X_i)W_{n,i}\big )\big )=1$. Also, for any  $\eta\in (0,1)$ there exists $\varepsilon=\varepsilon_\eta>0$ such that $\mathbb{E}\Big (\sum_{i=1}^N \big (\mathbf{1}_{V_\varepsilon}(X_i)-\mathbf{1}_F(X_i)\big )\|X_i\|\Big )\le \eta^2$. It is so since $\mathbb{E}(N^H)=\mathbb{E}(N^F)$, $\lim_{\varepsilon\to 0} 1_{V_\varepsilon}=1_H$ and $\mathbb E\big (\sum_{i=1}^ N  \|X_i\|\big )<\infty$.  It follows that 
\begin{align*}
&\Big\|\mathbb E\Big (\sum_{i=1}^NW_{n,i}X_i\Big )-p_{F,n}\mathbb E\Big (\sum_{i=1}^N\widetilde W_{n,i}X_i\Big )\Big \|\\&\le \Big\|\mathbb E\Big (\sum_{i=1}^N\mathbf{1}_{U_{\varepsilon_\eta}} W_{n,i}X_i\Big ) \Big \|+\eta^{-1}\mathbb{E}\Big (\sum_{i=1}^N \big (\mathbf{1}_{V_{\varepsilon_\eta}}(X_i)-\mathbf{1}_F(X_i)\big )\mathbf{1}_{[0,1/\eta]}(W_{n,i})\|X_i\|\Big )\\
&\qquad +   \mathbb{E}\Big (\sum_{i=1}^N\mathbf{1}_{(1/\eta,\infty)}(W_{n,i}) W_{n,i} \|X_i\|\Big)\le \Big\|\mathbb E\Big (\sum_{i=1}^N\mathbf{1}_{U_{\varepsilon_\eta}} W_{n,i}X_i\Big ) \Big \|+ \eta+ C_{N,X,\psi} g(1/\eta). 
\end{align*}
Fix $B_\eta>0$ such that $\mathbb E\big (\sum_{i=1}^N\mathbf{1}_{(B_\eta,\infty)}(\|X_i\|)\|X_i\|\big )\le \eta^2$. One also has 
\begin{align*}
&\Big\|\mathbb E\Big (\sum_{i=1}^N\mathbf{1}_{U_{\varepsilon_\eta}} W_{n,i}X_i\Big ) \Big \|\le \eta^{-1}\mathbb E\Big (\sum_{i=1}^N\mathbf{1}_{U_{\varepsilon_\eta}} \mathbf{1}_{[0,1/\eta]}(W_{n,i})\mathbf{1}_{(B_\eta,\infty)}(\|X_i\|)\|X_i\|\Big )\\
&\qquad + B_\eta\,\mathbb E\Big (\sum_{i=1}^N\mathbf{1}_{U_{\varepsilon_\eta}}W_{n,i}\mathbf{1}_{[0,B_\eta]}(\|X_i\|)\Big )+ \mathbb{E}\Big (\sum_{i=1}^N\mathbf{1}_{(1/\eta,\infty)}(W_{n,i}) W_{n,i} \|X_i\|\Big),
\end{align*}
so $\big\|\mathbb E\big (\sum_{i=1}^N\mathbf{1}_{U_{\varepsilon_\eta}} W_{n,i}X_i\big ) \big \|\le\eta+B_\eta(1-p_{V_{\varepsilon_\eta,n}})+C_{N,X,\psi} g(1/\eta)$.  Finally, for any $\rho>0$, if we fix $\eta\in(0,1)$ such that $3\eta+2C_{N,X,\psi} g(1/\eta)\le \rho$ and $n_\rho\in\N$ such that for all integers $n\ge n_\rho$ one has $B_\eta(1-p_{V_{\varepsilon_\eta,n}})\le \eta$, we get 
$\big\|\mathbb E\big (\sum_{i=1}^NW_{n,i}X_i\big )-p_{F,n}\mathbb E\big (\sum_{i=1}^N\widetilde W_{n,i}X_i\big )\big \|\le \rho$ for all $n\ge n_\rho$, hence the first assertion of \eqref{target}. 

The fact that $\lim_{n\to\infty} p_{F,n}=1$ follows from the properties  $\lim_{n\to\infty}p_{V_\varepsilon,n}=1$ for all $\varepsilon>0$, $\lim_{\varepsilon\to 0}\mathbb{E}\big (\sum_{i=1}^N \mathbf{1}_{V_\varepsilon}(X_i)-\mathbf{1}_{F}(X_i)\big )=0$, and the inequality $\mathbb{E}\big (\sum_{i=1}^N\mathbf{1}_{(a,\infty)}(W_{n,i}) W_{n,i}\big)\le (\log(a))^{-1} \mathbb{E}\big (\sum_{i=1}^NW_{n,i}\log^+W_{n,i}\big)\le (\log(a))^{-1}e^{-1}\mathbb E(N)$ for all $a>1$.

For assertion (2),  assume that $\widetilde P_{X,\phi,\alpha}^*(0)>0$, for otherwise the result is direct. Since we know that $\mathbb E\big (\sum_{i=1}^N W_{n,i}\phi_i\big )\ge -\log(\beta)>0$, where $\beta$ is as in Lemma~\ref{controlSnphi}, one has $h:=\inf_{n\in\mathbb N} - \mathbb E\big (\sum_{i=1}^N W_{n,i}\log W_{n,i}\big )>0$. Also, mimicking what was done above yields $\lim_{n\to\infty} \big |\mathbb E\big (\sum_{i=1}^N W_{n,i}\phi_i\big )-\mathbb E\big (\sum_{i=1}^N \widetilde W_{n,i}\phi_i\big )\big |=0$, hence  $\lim_{n\to\infty} \frac{ \mathbb E\big (\sum_{i=1}^N \widetilde W_{n,i}\phi_i\big )}{\mathbb E\big (\sum_{i=1}^N W_{n,i}\phi_i\big )}=1$. 
Since $\lim_{n\to\infty} \frac{-\mathbb E\big (\sum_{i=1}^N W_{n,i}\log W_{n,i}\big )}{\mathbb E\big (\sum_{i=1}^N W_{n,i}\phi_i\big )}=\widetilde P_{X,\phi,\alpha}^*(0)$, to conclude we only need to prove that 
\begin{equation}\label{quotient}
\liminf_{n\to\infty} \frac{\mathbb E\big (\sum_{i=1}^N \widetilde W_{n,i}\log \widetilde W_{n,i}\big )}{\mathbb E\big (\sum_{i=1}^N W_{n,i}\log W_{n,i}\big )}\ge 1.
\end{equation} 
To see this, write  
\begin{align}
\label{A}-\mathbb E\Big (\sum_{i=1}^N \widetilde W_{n,i}\log \widetilde W_{n,i}\Big )=\log(p_{F,n})-\mathbb E\Big (\sum_{i=1}^N \mathbf{1}_F(X_i) W_{n,i}\log  W_{n,i}\Big ),
\end{align}
and setting $Z_{n,i}= W_{n,i}\log  W_{n,i}$, note that 
\begin{align}
\label{AA}\mathbb E\Big (\sum_{i=1}^N \mathbf{1}_F(X_i) Z_{n,i}\Big )&\le \mathbb E\Big (\sum_{i=1}^N \mathbf{1}_F(X_i)\mathbf{1}_{[0,1]}(W_{n,i}) Z_{n,i}\Big ) + \mathbb E\Big (\sum_{i=1}^N \mathbf{1}_{(1,\infty)}(W_{n,i}) Z_{n,i}\Big ),
\end{align}
and that for all $\varepsilon>0$, 
\begin{align}
\nonumber&\Big |\mathbb E\Big (\sum_{i=1}^N \mathbf{1}_F(X_i)\mathbf{1}_{[0,1]}(W_{n,i}) Z_{n,i}\Big )-\mathbb E\Big (\sum_{i=1}^N \mathbf{1}_{[0,1]}(W_{n,i}) Z_{n,i}\Big ) \Big|\\
\label{BB}&\le \Big |\mathbb E\Big (\sum_{i=1}^N \big (\mathbf{1}_{V_\varepsilon}-\mathbf{1}_F\big )(X_i)\mathbf{1}_{[0,1]}(W_{n,i}) Z_{n,i}\Big )\Big|+\Big |\mathbb E\Big (\sum_{i=1}^N \mathbf{1}_{U_\varepsilon}(X_i)\mathbf{1}_{[0,1]}(W_{n,i}) Z_{n,i}\Big ) \Big|.
\end{align}
Since $|W_{n,i}\log  W_{n,i}|$ is bounded by $1/e$ over  $\{W_{n,i}\in [0,1]\}$, for  every $\eta\in (0,h)$ there exists $\varepsilon_\eta>0$ such that 
\begin{equation}\label{CC}
\Big |\mathbb E\Big (\sum_{i=1}^N \big (\mathbf{1}_{V_{\varepsilon_\eta}}(X_i)-\mathbf{1}_F(X_i)\big )\mathbf{1}_{[0,1]}(W_{n,i}) Z_{n,i}\Big )\Big|\le \eta/2.
\end{equation}
Set $a_{\eta,n}=(1-p_{V_{\varepsilon_\eta},n})^{-1/2}$ and $r_{\eta,n}=\big |\mathbb E\big (\sum_{i=1}^N \mathbf{1}_{U_{\varepsilon_\eta}}(X_i)\mathbf{1}_{[0,1]}(W_{n,i}) Z_{n,i}\big ) \big|$. 
Since by construction $\lim_{n\to\infty}a_{\eta,n}=\infty$, for $n$ large enough one has  $e^{-a_{\eta,n}}\le 1/e$, whence  $|Z_{n,i}|\le  a_{\eta,n}e^{-a_{\eta,n}}$ when $W_{n_k,i}\in \big[0,e^{-a_{\eta,n}}\big ]$. Consequently, for $n$ large enough
 $$
 r_{\eta,n}\le  a_{\eta,n}e^{-a_{\eta,n}}\mathbb{E}(N) +a_{\eta,n}\, \mathbb{E}\Big (\sum_{i=1}^N \mathbf{1}_{U_{\varepsilon_\eta}}(X_i) W_{n,i}\Big )=a_{\eta,n}e^{-a_{\eta,n}}\mathbb{E}(N)+a_{\eta,n}^{-1},
 $$
which implies that $\lim_{n\to\infty}r_{\eta,n}=0$. Putting this together with \eqref{A}--\eqref{CC} yields that for any $\eta>0$, for $n$ large enough, 
$
-\mathbb E\Big (\sum_{i=1}^N \widetilde W_{n,i}\log  \widetilde W_{n,i}\Big )\ge -\mathbb E\Big (\sum_{i=1}^N W_{n,i}\log  W_{n,i}\Big )-\eta$. This yields \eqref{quotient}.  \end{proof}

{\bf Case 2:}  $F\in \overline{\mathcal F}^1_X$.  We start by proving that $I_X\cap F\subset\{\alpha_F\}$ (we already proved that $\{\alpha_F\}\subset I_X\cap F$), which will establish that $I_X\cap F=I^F_X=\widetilde I^F_X$ by definition of $I^F_X$ and $\widetilde I^F_X$. 

If $\mathbb{E}(N^F)>1$ then $\dim F=0$, and the conclusion is direct. So we  assume that $\mathbb{E}(N^F)=1$ (and implicitly $d\ge 2$ for otherwise the discussion is trivial). Let $H\in \widetilde H_X$ such that $F\subset H$. By definition of $\overline{\mathcal F}^1_X$ one  has $\alpha_F=\alpha_H$, and it remains to prove that $I_X\cap H\subset\{\alpha_H\}$. To do so, fix $\alpha\in I_X\cap H$ and a sequence $(q_n)_{n\in\mathbb N}\in J_X^{\mathbb N}$ such that $\nabla\widetilde P_X(q_n)$ converges to $\alpha$ as $n\to\infty$.  Set $W_{n,i}=e^{\langle q_n|X_i\rangle-\widetilde P_X(q_n)}$, $i\ge 1$. For $\varepsilon>0$ and $n\in\mathbb N$, define 
$$
P_{n,\varepsilon}: \theta \in [0,1]\mapsto \mathbb{E}\Big (\sum_{i=1}^N \mathbf{1}_{V_\varepsilon}(X_i) W_{n,i}^\theta\Big ).
$$

\begin{lemma}\label{lastlem}
For all $r\in (0,1)$, there exist a positive sequence $(\varepsilon_k)_{k\in\mathbb N}$ converging to 0 and an increasing sequence of integers $(n_k)_{k\in\mathbb N}$ such that $\lim_{k\to\infty}\|P_{n_k,\varepsilon_k}-1\|_\infty=0$, as well as a sequence $(\theta_k)_{k\in\mathbb N}\in [0,1]^{\mathbb N}$ such that for all $k\in\mathbb N$:
\begin{equation}\label{ineq Wntheta}
\mathbb{E}\Big (\sum_{i=1}^N \mathbf{1}_{V_{\varepsilon_k}}(X_i) \mathbf{1}_{\{W_{n_k,i}\in [1-r,1+r]^c\}}W_{n_k,i}^{\theta_k}\Big )\le r.
\end{equation} 
\end{lemma}
Assume this lemma for a while. We know that $\lim_{k\to\infty}  \mathbb{E}\big (\sum_{i=1}^N \mathbf{1}_{V_{\varepsilon_k}}(X_i) X_i\big )=\alpha_H$ and $\lim_{k\to\infty}  \mathbb{E}\big (\sum_{i=1}^N\mathbf{1}_{V_{\varepsilon_k}}(X_i) W_{n_k,i} X_i\big )=\alpha$ (using the same arguments as  in the proof of Lemma~\ref{AAA}(1)). Proving that $\lim_{k\to\infty}  \mathbb{E}\big (\sum_{i=1}^N\mathbf{1}_{V_{\varepsilon_k}}(X_i) |W_{n_k,i}-1| \|X_i\|\big )=0$ will give us the conclusion. To do so, note that for any $r\in (0,1)$, for all $k\in\mathbb N$, due to \eqref{ineq Wntheta} and the fact that $1-\theta_k \in[0,1]$, one has
\begin{align*}
\mathbb{E}\Big (\sum_{i=1}^N \mathbf{1}_{V_{\varepsilon_k}}(X_i) \mathbf{1}_{ [1-r,1+r]}(W_{n_k,i})W_{n_k,i}\Big )
&\ge (1-r)^{1-\theta_k} \mathbb{E}\Big (\sum_{i=1}^N \mathbf{1}_{V_{\varepsilon_k}}(X_i) \mathbf{1}_{ [1-r,1+r]}(W_{n_k,i})W_{n_k,i}^{\theta_k}\Big )\\
&\ge (1-r)^2.
\end{align*}
Consequently, since $\mathbb{E}\big (\sum_{i=1}^N \mathbf{1}_{V_{\varepsilon_k}}(X_i) W_{n_k,i}\big )\le 1$, one obtains
\begin{equation}\label{3r}
 \mathbb{E}\Big (\sum_{i=1}^N \mathbf{1}_{V_{\varepsilon_k}}(X_i)\mathbf{1}_{ [1-r,1+r]^c}(W_{n_k,i})W_{n_k,i}\Big )\le 1-(1-r)^2\le 2r.
\end{equation}

For $\eta\in(0,1/2)$, let $B_\eta$ as in the proof of Lemma~\ref{AAA}(1). Assume without loss of generality that $B_\eta>1$ and set $r=r_\eta=\eta/B_\eta$; one has $1+r_\eta<\eta^{-1}$. For  $k\in\mathbb N$,  set $A_{r,k}=\mathbb{E}\Big (\sum_{i=1}^N\mathbf{1}_{V_{\varepsilon_k}}(X_i)  \mathbf{1}_{[0,1-r]}(W_{n_k,i}) |W_{n_k,i}-1| \|X_i\|\Big )$. One has $A_{r,k}\le B_\eta B_{r,k}+  \mathbb{E}\big (\sum_{i=1}^N \mathbf{1}_{(B_\eta,\infty)}(\|X_i\|) \|X_i\|\big )$, where
\begin{align*}
&B_{r,k}=\mathbb{E}\Big (\sum_{i=1}^N\mathbf{1}_{V_{\varepsilon_k}}(X_i)  \mathbf{1}_{[0,1-r]}(W_{n_k,i}) (1-W_{n_k,i})\Big )\\
&\le |P_{n_k,\varepsilon_k}(0)-P_{n_k,\varepsilon_k}(1)|+  \mathbb E\Big (\sum_{i=1}^N
\mathbf{1}_{V_{\varepsilon_k}}(X_i)\big ( \mathbf{1}_{[1-r,1+r]}(W_{n_k,i})r+\mathbf{1}_{[1-r,1+r]^c}(W_{n_k,i}) W_{n_k,i}\big )\Big ).
\end{align*}
Moreover, 
\begin{align*}
\mathbb{E}\Big (\sum_{i=1}^N\mathbf{1}_{V_{\varepsilon_k}}(X_i) |W_{n_k,i}-1| \|X_i\|\Big )&\le A_{r,k}+\mathbb{E}\Big (\sum_{i=1}^N\mathbf{1}_{[1-r,1+r]}(W_{n_k,i}) r  \|X_i\|\Big)\\&+\mathbb{E}\Big (\sum_{i=1}^N\mathbf{1}_{V_{\varepsilon_k}}(X_i) \mathbf{1}_{[1-r,1+r]^c}(W_{n_k,i})W_{n_k,i}\mathbf{1}_{[0,B_\eta]}(\|X_i\|)B_\eta\Big )\\
&\ \, +
\mathbb{E}\Big (\sum_{i=1}^N\mathbf{1}_{V_{\varepsilon_k}}(X_i) \mathbf{1}_{[0,1/\eta]}(W_{n_k,i}) \eta^{-1}\mathbf{1}_{(B_\eta,\infty)}(\|X_i\|) \|X_i\|\Big )\\&\ \ \  +\mathbb{E}\Big (\sum_{i=1}^N \mathbf{1}_{(1/\eta,\infty)}(W_{n_k,i}) W_{n_k,i}\|X_i\|\Big ).
\end{align*}
This, together with the  estimates obtained in the proof of Lemma~\ref{AAA} as well as \eqref{3r} yield, after setting $\delta_k=2\|P_{n_k,\varepsilon_k}-1\|_\infty$:
\begin{align*}
\mathbb{E}\big (\sum_{i=1}^N\mathbf{1}_{V_{\varepsilon_k}}(X_i) |W_{n_k,i}-1| \|X_i\|\big )
\le B_\eta\delta_k+\Big (3+\mathbb{E}(N)+\mathbb{E}\big (\sum_{i=1}^N\|X_i\|\big )\Big )\eta+C_{N,X,\psi}g(1/\eta),
\end{align*}
which is enough to conclude.

\begin{proof}[Proof of Lemma~\ref{lastlem}] Let us now prove the claim. Note first that $P_{n,\varepsilon}(0)=\mathbb{E}\big (N^{V_\varepsilon} \big)$ tends to $\mathbb{E}\big(N^H\big)=1$ as $\varepsilon\to 0$. Moreover, arguing like  in the proof of (2) shows that for any fixed $\varepsilon>0$, one has $\lim_{n\to\infty}P_{n,\varepsilon}(1)=1$. Consider  a positive sequence $(\varepsilon_k)_{k\in\mathbb N}$ converging to 0 and an increasing sequence of integers $(n_k)_{k\in\mathbb N}$ such that  $\lim_{k\to\infty}P_{n_k,\varepsilon_k}(1)=1$. Since the functions $P_{n_k,\varepsilon_k}$ are convex, the previous properties imply that to prove that $\lim_{k\to\infty}\|P_{n_k,\varepsilon_k}-1\|_\infty=0$, it is sufficient to prove that $\limsup_{k \to\infty} P_{n_k,\varepsilon_k}'(1)=0$. 

Suppose, by contradiction, that this is not the case.  Without loss of generality, fix $\eta>0$ such that for all $k\in\mathbb N$ one has
$
P_{n_k,\varepsilon_k}'(1)=\mathbb{E}\Big (\sum_{i=1}^N \mathbf{1}_{V_{\varepsilon_k}}(X_i) W_{n_k,i}\log W_{n_k,i} \Big )\ge \eta.
$
Define two sequences $\big(R_k=\mathbb{E}\big (\sum_{i=1}^N \mathbf{1}_{U_{\varepsilon_k}}(X_i) \mathbf{1}_{[0,1]}(W_{n_k,i})W_{n_k,i}\log W_{n_k,i} \big )\big)_{k\in\mathbb N}$ and $\big (a_k=(1-P_{n_k,\varepsilon_k}(1))^{-1/2}=\big (\mathbb{E}\big (\sum_{i=1}^N \mathbf{1}_{U_{\varepsilon_k}}(X_i) W_{n_k,i}\big )\big)^{-1/2}\big)_{k\in\mathbb N}$.  Using the same argument as in the proof of Lemma~\ref{AAA}(2) one can get 
$ |R_k|\le  a_ke^{-a_k}\mathbb{E}(N)+a_k^{-1}$
for $k$ large enough, so $\lim_{k\to\infty}R_k=0$, and finally 
$
\liminf_{k\to \infty} \mathbb{E}\big (\sum_{i=1}^N W_{n_k,i}\log W_{n_k,i} \big )
\ge \liminf_{k\to \infty}P_{n_k,\varepsilon_k}'(1)\ge \eta.
$
However, for all $k\in\mathbb N$ one has $\mathbb{E}\big (\sum_{i=1}^N W_{n_k,i}\log W_{n_k,i} \big )=-\widetilde P_X^*(\nabla \widetilde P_X(q_{n_k}))<0$, since $q_{n_k}\in J_X$. This is the desired contradiction.

Next, suppose that there exists $r\in (0,1)$ such that for $k$ large enough, for all $\theta\in[0,1]$, one has 
$
\mathbb{E}\big (\sum_{i=1}^N \mathbf{1}_{V_{\varepsilon_k}}(X_i) \mathbf{1}_{\{W_{n_k,i}\in [1-r,1+r]^c\}}W_{n_k,i}^{\theta}\big )> r.
$
This implies that for $k$ large enough, $P_{n_k,\varepsilon_k}''(\theta)\ge r( \log (1+r))^2 $ for all $\theta\in [0,1]$, which contradicts the fact that $\lim_{k\to\infty}\|P_{n_k,\varepsilon_k}-1\|_\infty=0$. This new contradiction yields the claim.
\end{proof}
Now we come to the validity of the conclusions of Theorem~\ref{UNIFER2}(4) and (5) for $F$.

Suppose again that $\mathbb{E}(N^F)=1$. The previous arguments and calculations show that $\lim_{k\to\infty} P_{n_k,\varepsilon_k}'(1)=0$ so 
$
\liminf_{k\to \infty} \mathbb{E}\big (\sum_{i=1}^N W_{n_k,i}\log W_{n_k,i} \big )\ge \liminf_{k\to \infty}P_{n_k,\varepsilon_k}'(1)=0. 
$
This implies that $ \widetilde P_X^*(\alpha_H)=\lim_{k\to\infty} \widetilde P_X^*(\nabla \widetilde P_X(q_{n_k}))=0$. Since under $\mathrm{d}_\phi$ the Hausdorff dimension of $E_X(\alpha)$ is bounded by  $(|\log (\beta)|)^{-1}$ times its Hausdorff dimension under $\mathrm{d}_1$ (where $\beta$ is as in \eqref{uniformboundSnphi}), we conclude that $\dim E(X,\alpha_H)=0$, hence  Theorem~\ref{UNIFER2}(4) for $F$.

If $F=\{\alpha_F\}$ and $\mathbb{E}(N^F)>1$, the same argument as when $\dim F\ge 1$ shows that $\widetilde P^*_{X,\phi,\alpha_F}(0)\le \widetilde P_{X_F-\alpha_F,\phi_F,0}(0)$, where $\widetilde P_{X_F-\alpha_F,\phi_F,0}(0)$ is the Hausdorff dimension of $\partial\TT^F$ conditionally on non extinction of $\TT^F$. Now, consider a positive sequence $\theta=(\theta_n)_{n\in\mathbb{N}}$ converging to 0, as well as $\widetilde W=\big (\widetilde W_i=\mathbf{1}_{[0,N]}(i) \mathbf{1}_{F}(X_{i})\exp(-\widetilde P_{X_F-\alpha_F,\phi_F,0}(0)\phi_i)\big )_{i\ge 1}$. Then, for $n\ge 1$ and $(u,i)\in\mathbb{N}^{n-1}\times\N$, set
$
W_{ui}(\theta)=\big (\lambda_{\theta}(n))^{-1}\big (\mathbf{1}_{[1, N_u]}(i)\cdot  (\widetilde W_{ui}+\theta_n\mathbf{1}_{F^c}(X_{ui})\big )\big )_{i\ge 1},
$
where $\lambda_{\theta}(n)=\mathbb{E}\big (\sum_{i=1}^N\widetilde W_i+\theta_n\mathbf{1}_{F^c}(X_{i})\big )$. We leave the reader check that this family of random vectors defines almost surely a non degenerate inhomogeneous Mandelbrot measure $\mu^F$ supported on $E(X,\alpha_F)$ and of Hausdorff dimension 
$\widetilde P_{X_F-\alpha_F,\phi_F,0}(0)$, so  $\widetilde P_{X_F-\alpha_F,\phi_F,0}(0)\le \dim E(X,\alpha_F)\le  \widetilde P^*_{X,\phi,\alpha_F}(0)$. This yields Theorem~\ref{UNIFER2}(4) for $F$. 

Finally, Theorem~\ref{UNIFER2}(5) for $F$ follows from calculations similar to those done in the proof of Theorem~\ref{UNIFER}, by considering $\mu^F$ instead of $\mu_\varrho$. 

\noindent
(6) Suppose that $F\in \widetilde {\mathcal F}_X^1$ and let $H$ be an element of $\mathcal H_X$ such that $F'\subset H$ and $F\not\subset H$. Suppose that 
$\widetilde I^F_X\cap\widetilde I^{F'}_X\neq\emptyset$. Due to point (1) of this proposition applied to the branching random walk $(S_nX)_{n\in\mathbb{N}}$ restricted to $\partial \TT^F$, one has $\widering{\mathcal C}_{X,F}\cap  {\mathcal C}_{X,F'}\neq\emptyset$. This implies that for any point $\alpha\in \widering{\mathcal C}_{X,F}\cap  {\mathcal C}_{X,F'}$  there are necessarily points $\alpha'$ and $\alpha''$ in a neighbourhood of $\alpha$ relative to $F$ such that $L_{e,c}(\alpha')<0$ and $L_{e,c}(\alpha'')>0$, for otherwise some neighbourhood of $\alpha$ relative to $F$ would be included in $F'$, hence $F\subset F'$, which would contradict the assumption that $F\in\widetilde {\mathcal{F}}^1_X$. However, the inquality $L_{e,c}(\alpha'')>0$ contradicts the fact that $\mathcal C_X\subset L_{e,c}^{-1}(\R_-)$. So $\widetilde I^F_X\cap\widetilde I^{F'}_X=\emptyset$.

By symmetry, we can now suppose that both $F$ and $F'$ belong to $\overline{\mathcal F}_X^1$. If $\mathbb{E}(N^F)>1$ one has $F=\{\alpha_F\}=\widetilde I^F_X$. So if $\widetilde I^F_X\cap\widetilde I^{F'}_X\neq\emptyset$, then $F\subset F'$, which contradicts the fact that $F\in\overline {\mathcal{F}}^1_X$. Still by symmetry, the only case which remains to be studied is $\mathbb{E}(N^F)=1=\mathbb{E}(N^{F'})$. Then $\widetilde I^F_X\cap \widetilde I^{F'}_X\neq\emptyset$ means that $\alpha_F=\alpha_{F'}\in H$. Since $\alpha_F=\mathbb{E}\Big (\sum_{i=1}^N \mathbf{1}_F(X_i)X_i\Big)$ and $F\subset L_{e,c}^{-1}(\R_-)$, this implies that with probability 1, for all $1\le i\le N$, $\mathbf{1}_F(X_i)=1$ implies $\mathbf{1}_H(X_i)=1$. This holds for any element $H$ of $\mathcal H_X$ containing $F'$, from which we conclude that $\mathbb{E}(N^{F\cap F'})=1$. Consequently, $F\cap F'=F$ because $F\in\widehat {\mathcal F}_X$, and we get $F\subset F'$, new contradiction.  Thus $\widetilde I^F_X\cap \widetilde I^{F'}_X=\emptyset$. 
\end{proof}

\section{Proofs of Corollaries~\ref{cor-1},~\ref{cor-2} and~\ref{cor-3}}\label{pfcor}

\begin{proof}[Proof of Corollary~\ref{cor-1}] For each $r\in \mathbb Q\cap(0,\infty)$, fix $\widetilde k_r\in\boldsymbol{\widetilde K}$ such that $\lim_{n\to\infty} \frac{\log(k_r(n))}{n}=r$. It follows from Theorem\,B that there exists $\Omega^*\subset \Omega$ of $\mathbb{P}$-probability 1, such that for all $\omega\in \Omega^*$, there exists a set $E^\omega$ of full $\nu$-measure such that for all $t\in E^\omega$ and all $r\in \mathbb Q\cap(0,\infty)$ the large deviation principle $\mathrm{LD}(\Lambda_\psi,\widetilde k_r)$ holds.

Now fix $\omega\in\Omega^*$ and $t\in E^\omega$. Then fix $\widetilde k\in \boldsymbol{\widetilde K}$. Set $\ell=\lim_{n\to\infty} \frac{\log(k(n))}{n}$. If $\ell =0$ there is nothing to prove. If $\ell>0$, and $\lambda \in\mathcal{D}_{\Lambda_\psi}$ such that $\ell> -\Lambda_\psi^*(\nabla \Lambda_\psi(\lambda))$, for any $r_1< \ell< r_2$ with $r_1,r_2\in \mathbb Q\cap(0,\infty)$ and $r_1> -\Lambda_\psi^*(\nabla \Lambda_\psi(\lambda))$, for $n$ large enough one has $k_{r_1}(n)\le k(n)\le k_{r_2}(n)$ so that 
$$
n^{-1}\log(k(n)/k_{r_1}(n))+n^{-1}\Lambda^t_{\tilde k_{r_1},n}(\lambda)\le 
 n^{-1}\Lambda^t_{\tilde k,n}(\lambda)\le n^{-1}\log(k_{r_2}(n)/k(n))+n^{-1}\Lambda^t_{\tilde k_{r_2},n}(\lambda),
 $$
 hence $
\ell-r_1+\Lambda_\psi(\lambda)\le  \liminf_{n\to\infty} n^{-1}\Lambda^t_{\tilde k,n}(\lambda)\le\limsup_{n\to\infty} n^{-1}\Lambda^t_{\tilde k,n}(\lambda) \le r_2-\ell +\Lambda_\psi(\lambda),
$
where we used that $\mathrm{LD}(\Lambda_\psi,\widetilde k_r)$ holds at $(\omega,t)$ for $r\in\{r_1,r_2\}$. Since $r_1$ and $r_2$ are arbitrary, we get that part (1) of $\mathrm{LD}(\Lambda_\psi,\widetilde k)$ holds as well at $(\omega,t)$. 

To get part (2) of $\mathrm{LD}(\Lambda_\psi,\widetilde k)$, suppose that $\ell<-\Lambda_\psi^*(\nabla \Lambda_\psi(\lambda))$ and take $r_2\in \mathbb Q\cap(0,\infty)$ such that $\ell<r_2 <-\Lambda_\psi^*(\nabla \Lambda_\psi(\lambda))$. The fact that $\mathrm{LD}(\Lambda_\psi,\widetilde k_{r_2})$ holds directly implies that there exists $\epsilon>0$ such that  $\Big \{0\le j\le nk(n)-1: \frac{S_{j+n}X(t)-S_{j}X(t)}{n}\in B(\nabla\Lambda_\psi(\lambda),\epsilon)\Big \} =\emptyset$. 
for $n$ large enough. 

Finally, the fact that part (3) of $\mathrm{LD}(\Lambda_\psi,\widetilde k)$ holds follows from the fact that part (1) holds and the arguments developed to derive \cite[Theorem 2.3(3)]{BL}.
\end{proof}

Corollaries~\ref{cor-2} and~\ref{cor-3} are proved similarly.  

\section{Possible relaxation of the assumptions in Theorems~\ref{thm-1.1} and \ref{UNIFER}}\label{Relaxation}
Set $\mathcal{D}_X=\mathrm{dom}(\widetilde P_X)$ and note that  $\mathcal{D}_X$ is closed (as shows a simple application of Fatou's lemma). 

Assume that $\mathcal{D}_X\neq\R^d$ and  $\mathcal{D}_X$ contains an open neighbourhood of $0$, denoted by $V$. We can assume that $\widetilde P_X^*(\nabla \widetilde P_X(q))>0$ for all $q\in \mathcal V$ as well. 

Trying to mimick what was done when $\mathcal{D}_X$ equals $\R^d$, set $\widehat I_X=\{\nabla \widetilde P_X(q):\,  q\in \widering {\mathcal D}_X,\, \widetilde P_X^*(\nabla \widetilde P_X(q))\ge 0\}$. In \eqref{JXphi}, replace $\R^d\times I_X$ by  $\widering {\mathcal D}_X\times \widehat I_X$ in the definition of $J_{X,\phi}$, and in \eqref{phibound} take the supremum  over $q\in  {\mathcal{D}}_X$. The other assumptions remain the same. Then, as when $\mathcal{D}_X=\R^d$, for all $\alpha\in \widehat I_X$ there exists a unique point $q_\alpha\in\mathcal{D}_X$ at which $\inf_{q\in \mathcal{D}_X}\widetilde P_{X,\phi,\alpha}(q)$ is reached. Moreover, if $\partial \mathcal{D}_X$ is compact, then $\alpha\in \widehat I_X\mapsto q_\alpha$ is continuous.  But it is not clear that in general $q_\alpha$ belongs to $\widering {\mathcal D}_X$. 

Denote by $\widetilde I_X$ the set of those $\alpha\in \widehat I_X$ such that  $ q_\alpha\in \widering {\mathcal D}_X$. The points of $\widetilde I_X$ of the form $\nabla \widetilde P_X(q)$ such that $\widetilde P_X^*(\nabla \widetilde P_X(q))>0$ are interior points of $\widetilde I_X$. Then, the part of Theorem~\ref{thm-1.1} regarding $\dim E(X,\alpha)$ is still valid if one replaces $I_X$ by $\widetilde I_{X}$ and that about $\dim K$ is valid if $K\subset \widetilde I_{X}$. Moreover,  Theorem~\ref{UNIFER} is valid if for every $\alpha\in \widetilde I_{X}$ the domain  $\mathcal{D}_{\Lambda_{\psi_\alpha}}$ of definition of $\Lambda_{\psi_\alpha}$  is taken equal to $\widering{\mathcal{D}}_X-q_\alpha$. Now, we wish $\widetilde I_X$ to be not empty. The set $\widetilde I_X$ is not easy to understand in general. It is obvious that if ${\mathrm d}_\phi={\mathrm d}_1$, then $\widetilde I_X=\widehat I_X$ and $q_\alpha=q$ if $\alpha=  \nabla \widetilde P_X(q)$. The same properties hold if more generally the components of $\phi=(\phi_i)_{i\ge 1}$ are identically distributed and $\phi$ is independent of $(N,X=(X_i)_{i\in\mathbb N})$, or if the components of $X$ are identically distributed and $X$ is independent of $(N,\phi)$. Also, $\widetilde I_X$ is not empty if $\phi$ is a small perturbation of $(1)_{i\in\N}$. 
 
However, let $\alpha_0=\mathbb{E}\Big (\sum_{i=1}^NX_i \exp(-\widetilde P_{X,\phi,\alpha}(0)\phi_i)\Big )$, where $\alpha$ is arbitrary in $\R^d$ (we already saw that $\widetilde P_{X,\phi,\alpha}(0)$ does not depend on $\alpha$). By construction one has $q_{\alpha_0}=0$ since $\nabla \widetilde P_{X,\phi,\alpha_0}(0)=\mathbb{E}\Big (\sum_{i=1}^N(X_i-\alpha_0) \exp(-\widetilde P_{X,\phi,\alpha_0}(0)\phi_i)\Big )=0$. Moreover, it is easily seen that  the differential of $\alpha\mapsto \nabla\widetilde P_{X,\phi,\alpha}(0)$ at $\alpha_0$ is invertible (precisely, it is a nontrivial  multiple of the identity). Consequently, one can apply the implicit function theorem to $f: (q,\alpha)\mapsto (q,\nabla\widetilde P_{X,\phi,\alpha}(q))$ at $(0,\alpha_0)$ where $f$ takes the value $(0,0)$, and obtain a neighborhood of $\alpha_0$ over which $q_\alpha\in \widering{\mathcal D}_X$, and $\widetilde P_{X,\phi,\alpha}(q_\alpha)\ge 0$.

It is thus natural to consider the open set $\widetilde I_{X,\phi}$ of those $\alpha\in I_X$
 such that $q_\alpha\in \widering {\mathcal D}_X$ and $f$ is a local diffeomorphism at $(q_\alpha,\alpha)$. Then, the extensions of Theorems~\ref{thm-1.1} and~\ref{UNIFER} given in the penultimate paragraph hold as well if one replaces $\widetilde I_X$ by $\widetilde I_{X,\phi}$. Moreover, when $I_X=\overline{ \{\nabla \widetilde P_X(q): q\in J_X\}}$, one has $\widetilde I_{X,\phi}\subset\widetilde I_X$.

\end{document}